\documentclass[11pt,english]{elsarticle}
\usepackage[T1]{fontenc}
\usepackage[latin9]{inputenc}
\usepackage[a4paper]{geometry}
\usepackage{color}
\usepackage{textcomp}
\usepackage{amsmath}
\usepackage{amsthm}
\usepackage{amssymb}
\usepackage{hyperref}
\usepackage{bm}

\makeatletter
\numberwithin{equation}{section}
\numberwithin{figure}{section}
\theoremstyle{plain}
\newtheorem{thm}{\protect\theoremname}
  \theoremstyle{remark}
  \theoremstyle{plain}
\newtheorem{cor}{\protect\corollaryname}
  \newtheorem{rem}{\protect\remarkname}
  \theoremstyle{plain}
  \newtheorem{lem}{\protect\lemmaname}
  \theoremstyle{definition}
  \newtheorem{defn}{\protect\definitionname}
  \theoremstyle{plain}
  \newtheorem{prop}{\protect\propositionname}
  \theoremstyle{plain}
  \newtheorem*{thm*}{\protect\theoremname}

\DeclareMathOperator{\supp}{supp}
\DeclareMathOperator{\esssup}{ess \ sup}
\DeclareMathOperator{\essinf}{ess \ inf}
\DeclareMathOperator{\inte}{int}
\DeclareMathOperator{\BMO}{BMO}
\usepackage{pgf,tikz}
\usetikzlibrary{arrows}

\def\XXint#1#2#3{{\setbox0=\hbox{$#1{#2#3}{\int}$}
\vcenter{\hbox{$#2#3$}}\kern-.5\wd0}}

\def\supp{{\rm supp}\ }
\def\bbZ{\mathbb{Z}}
\def\bbR{\mathbb{R}}

\newcommand{\mM}{{\mathcal M}}

\makeatletter
\def\ps@pprintTitle{%
 \let\@oddhead\@empty
 \let\@evenhead\@empty
 \def\@oddfoot{}%
 \let\@evenfoot\@oddfoot}
\makeatother

\makeatother

\usepackage{babel}
  \providecommand{\definitionname}{Definition}
  \providecommand{\lemmaname}{Lemma}
  \providecommand{\propositionname}{Proposition}
  \providecommand{\remarkname}{Remark}
  \providecommand{\theoremname}{Theorem}
\providecommand{\theoremname}{Theorem}
\providecommand{\corollaryname}{Corollary}

\begin{document}

\begin{frontmatter}{}

\title{   Vector-valued operators, optimal weighted estimates and the $C_p$ condition}

\author[BCAM,UNL]{Mar\'{\i}a Eugenia Cejas}

\ead{mec.eugenia@gmail.com }

\author[BCAM,Fund1,Fund4]{Kangwei Li}

\ead{kli@bcamath.org}

\author[EHU,Ikerbasque,BCAM,Fund1,Fund2]{Carlos P\'erez}

\ead{carlos.perezmo@ehu.es }

\author[EHU,BCAM,Fund1,Fund2,Fund3]{Israel Pablo Rivera-R\'{\i}os}

\ead{petnapet@gmail.com}

\fntext[Fund1]{Supported by the Basque Government through the BERC 2014-2017 program
and by Spanish Ministry of Economy and Competitiveness MINECO: BCAM
Severo Ochoa excellence accreditation SEV-2013-0323}

\fntext[Fund2]{Supported by the Spanish Ministry of Economy and Competitiveness MINECO through the project
MTM2014-53850-P.}

\fntext[Fund3]{Supported by Spanish Ministry of Economy and Competitiveness MINECO through the project MTM2012-30748}
\fntext[Fund4]{Supported by Juan de la Cierva - Formaci\'on 2015 FJCI-2015-24547}

\address[BCAM]{BCAM - Basque Center for Applied Mathematics, Bilbao, Spain}

\address[UNL]{CONICET, Facultad de Ciencias Exactas, Departamento de Matematica,
UNLP}

\address[EHU]{Department of Mathematics, University of the Basque Country, Bilbao,
Spain}

\address[Ikerbasque]{Ikerbasque, Basque Foundation of Science, Bilbao, Spain}

\begin{abstract}
Sharp weighted estimates and local exponential decay estimates are obtained for vector-valued extensions
of the Hardy-Littlewood maximal operator, Calder\'on-Zygmund operators, rough singular integrals
and commutators. Those estimates will rely upon
suitable pointwise estimates in terms of sparse operators.
We also prove some new results for the $C_p$ classes introduced by Muckenhoupt \cite{M}  and later extended by Sawyer \cite{S}, in particular we extend the result to the full expected range $p>0$, to the weak norm, to other operators and to the their vector-valued extensions.

\end{abstract}

\end{frontmatter}{}
\setcounter{tocdepth}{1}
\renewcommand*\contentsname{Summary}
\tableofcontents
\section{Introduction and main results }

We recall that a weight $w$, that is, a non-negative locally integrable
function, belongs to the Muckenhoupt $A_{p}$ class for $1<p<\infty$
if
\[
\left[w\right]_{A_{p}}=\sup_{Q}\left(\frac{1}{|Q|}\int_{Q}w\right)\left(\frac{1}{|Q|}\int_{Q}w^{-\frac{1}{p-1}}\right)^{p-1}<\infty,
\]
where the supremum is taken over all cubes in $\mathbb{R}^{n}$ with sides parallel to the axes. And
in the case $p=1$ we say that $w\in A_{1}$ if
\[
Mw\leq\kappa w\quad\text{a.e.}
\]
and we define $[w]_{A_{1}}=\inf\{\kappa>0\,:\,Mw\leq\kappa w\quad\text{a.e.}\}$.  The quantity $[w]_{A_{p}}$ is called the $A_{p}$ constant or characteristic
of the weight $w$.

In the last years the sharp dependence on the $A_{p}$ constant of
weighted inequalities for several operators in Harmonic Analysis has
been thoroughly studied by many authors. We
could set the begining of the ``history of quantitative weighted
estimates'' in the pioneering work of S.M. Buckley \cite{B} in which
he proved that the following estimate holds
\begin{equation}\label{Buck}
\|M\|_{L^{p}(w)\rightarrow L^{p}(w)}\leq c_{n,p}[w]_{A_{p}}^{\frac{1}{p-1}},
\end{equation}
where $M$ stands for the Hardy-Littlewood maximal function, $1<p<\infty$
and $w\in A_{p}$. He also proved that the exponent of the $A_{p}$
constant cannot be replaced by a smaller one. Almost ten years later,
the so-called $A_{2}$ conjecture for the Ahlfors-Beurling transform
was raised in \cite{AIS} and solved not much later by S. Petermichl
and A. Volberg in \cite{PV}. The solution of this problem motivated
the study of the sharp dependence on the $A_{p}$ constant of operators
such as the Hilbert transform \cite{P1} or the Riesz transforms \cite{P2}. Then
to prove the linear dependence on the $A_{2}$ constant
for general Calder\'on-Zygmund operators became an important problem
in the area that was finally solved by T. Hyt\"onen \cite{HAnnals}.  The best result until then for the case of sufficiently smooth operators Calder\'on-Zygmund operators was obtained in  \cite{CUMP2} using  a method that was sufficiently flexible to treat other operators as can be found in \cite{CUMP3}. This result was improved in \cite{HPAinfty} for the case $p=2$ as well as \eqref{Buck} for any $p$, by means of what are now called mixed $A_p-A_{\infty}$ estimates as the one ones we obtain in this paper.  The mixed $A_p-A_{\infty}$ estimate for Calder\'on-Zygmund operators and general $p$  was obtained in \cite{HLac,HLP}. Also new proofs and improvements of the mixed $A_p-A_{\infty}$ versions of \eqref{Buck} were obtained in \cite{HPR} and \cite{PR}.

In the attempt to simplify the proof of the $A_{2}$ conjecture, as well
as trying to solve other problems, a ``new technology'' based on
the domination of Calder\'on-Zygmund operators by positive objects was
developed. One paradigmatic result in that direction is the norms
domination of Calder\'on-Zygmund operators by sparse operators due to
A. Lerner \cite{L2} that reads as follows. Given a maximal Calder\'on-Zygmund
operator $T^{*}$ and a Banach function space $X$
\[
\|T^{*}f\|_{X}\leq\sup_{\mathcal{S}}\|\mathcal{A}_{\mathcal{S}}f\|_{X},
\]
where
\[
\mathcal{A}_{\mathcal{S}}f(x)=\sum_{Q\in\mathcal{S}}\frac{1}{|Q|}\int_{Q}|f(y)|dy\chi_{Q}(x)=: \sum_{Q\in\mathcal{S}}\langle |f| \rangle_Q \chi_{Q}(x)
\]
and each $\mathcal{S}$ is a sparse family contained in any general
dyadic grid (see Section \ref{subsec:Lerner's-Formula} for the precise
definition of sparse family).

Quite recently the former norm control was improved to a pointwise
control (see \cite{LN,CAR,La,HRT,L} and also \cite{CDPO,BFP} for
other interesting advances) which in its fully quantitative version
(see \cite{HRT,L}) reads as follows
\[
\left|Tf(x)\right|\leq c_{n}C_{T}\sum_{k=1}^{3^{n}}\mathcal{A}_{\mathcal{S}_{k}}|f|(x),
\]
where $C_{T}=C_{K}+\|T\|_{L^{2}\rightarrow L^{2}}+\|\omega\|_{\text{Dini}}$
(see Section \ref{subsec:Notation} for the precise definition of
$T$ and these constants).

The development of such a pointwise control has required the development
of several new approaches. Among them we would like stress \cite{L,LN,LORR}
since the techniques contained in them will be crucial in the present work.

\subsection{Estimates for vector-valued Hardy-Littlewood maximal operators}
Given $0<q<\infty$ we define
the vector-valued Hardy-Littlewood maximal operator $\overline{M}_{q}$
as
\[
\overline{M}_{q}\bm{f}(x)=\left(\sum_{j=1}^{\infty}Mf_{j}(x)^{q}\right)^{\frac{1}{q}}.
\]
This operator was introduced by Fefferman and Stein
in \cite{FS1} as a generalization of both the scalar maximal function
$M$ and the classical integral of Marcinkiewicz. Since then, it has played an important role in Harmonic Analysis.  It is well known
that this operator is of strong type $(p,p)$ and of weak type
$(1,1)$:
\begin{equation*}
||\overline{M}_{q}\bm{f}||_{L^{p}(\mathbb{R}^{n})}\leq c_{n,p,q} \||\bm{f}|_{q}\|_{L^{p}\left(\mathbb{R}^{n} \right)}  \qquad  1<p,q<\infty
\end{equation*}
and
\begin{equation*}
\|\overline{M}_{q}(\bm{f})\|_{L^{1,\infty}(\mathbb{R}^{n})}
\leq c_{n,q}\int_{\mathbb{R}^{n}}|\bm{f}(x)|_{q}\,dx,
\end{equation*}
where $|\bm{f}|_{q}= \left(   \sum_{j=1}^{\infty} |f_{j} |^q  \right)^{\frac{1}{q}}$
 (see \cite{FS1}). They are bounded on $L^p(w)$
if and only if $w\in A_{p}$  \cite{AJ,K}.  It is also known  the sharp dependence on the $A_{p}$ constant \cite{CUMP3}
and that they satisfy a Fefferman-Stein type inequality (\cite[Theorem 1.1]{P}).

Now we turn to our contribution in this work. In \cite{LN} A. Lerner
and F. Nazarov developed a new version of Lerner's formula (\cite{L1})
that allows to obtain a pointwise control of any function satisfying
minimum assumptions in terms of its ``oscillations'' (see Section
\ref{subsec:Lerner's-Formula} for more details). Using Lerner-Nazarov
formula we will obtain a sparse control for vector-valued Hardy-Littlewood
maximal operators, that somehow appears implicitly in \cite{CUMP3},
that reads as follows.
\begin{thm}
\label{ThmSparseMq}Let $1<q<\infty$ and $f=\{f_{j}\}_{j=1}^{\infty}$,
such that for each $\varepsilon>0$
\[
\left|\left\{ x\in[-R,R]^{n}\,:\,|\overline{M}_{q}\bm{f}(x)|>\varepsilon\right\} \right|=o(R^{n})\,\, \text{as}\,\,R\rightarrow \infty.
\]
Then there exists $3^{n}$ dyadic lattices $\mathcal{D}_{k}$ and
$3^{n}$ $\frac{1}{6}$-sparse families $\mathcal{S}_{k}\subseteq\mathcal{D}_{k}$
depending on $f$ such that
\[
|\overline{M}_{q}\bm{f}(x)|\leq c_{n,q}\sum_{k=1}^{3^{n}}\mathcal{A}_{\mathcal{S}_{k}}^{q}|\bm{f}|_{q}(x),
\]
where
$$\mathcal{A}_{\mathcal{S}_{k}}^{q}|\bm{f}|_{q}(x)=\left(\sum_{Q\in\mathcal{S}_{k}}  \langle |\bm{f}|_{q}\rangle_Q^q \chi_{Q}(x)\right)^{\frac{1}{q}}.$$
\end{thm}

Using the pointwise control, as a consequence of Theorems \ref{ThmFuerteAp}
and \ref{ThmDebilAp} (cf. Section \ref{subsec:ApAinfty} for the
statements and the notation)  we have the following mixed estimates
\begin{thm}
\label{ThmEstMq}Let $1<p,q<\infty$ and let $w$ and $\sigma$ be a
pair of weights. Then
\[
\|\overline{M}_{q}(\sigma \bm{f})\|_{L^{p}(w)}\lesssim[w,\sigma]_{A_{p}}^{\frac{1}{p}}\left([w]_{A_{\infty}}^{\left(\frac{1}{q}-\frac{1}{p}\right)_{+}}+[\sigma]_{A_{\infty}}^{\frac{1}{p}}\right)\||\bm{f}|_{q}\|_{L^{p}(\sigma)}.
\]
If we also have that $p\not=q$, then
\[
\|\overline{M}_{q}(\sigma \bm{f})\|_{L^{p,\infty}(w)}\lesssim[w,\sigma]_{A_{p}}^{\frac{1}{p}}[w]_{A_{\infty}}^{\left(\frac{1}{q}-\frac{1}{p}\right)_{+}}\||\bm{f}|_{q}\|_{L^{p}(\sigma)}.
\]

\end{thm}

This result reflects the interesting fact that $\overline{M}_{q}$ behaves like a Calder\'on-Zygmund operator as $q\rightarrow1$
while it behaves like the (scalar) Hardy-Littlewood maximal function as $q\rightarrow \infty$. This observation is also  reflected in some of the main results in \cite{P}.

We observe that the strong-type estimate is an improvement of \cite[Theorem 1.12]{CUMP3}
since we obtain here mixed $A_{p}-A_{\infty}$ constants. Also as
far as we are concerned, no dependence on the $A_{p}$ constant of
the weak type estimate of $\overline{M}_{q}$ appears to have been considered before.

We would like to point out that the case $p=1$ was established in \cite{P}, namely for every weight $w$ the following estimate holds
\[
\|\overline{M}_q(\bm{f})\|_{L^{1,\infty}(w)}\leq c_{n,q}\||\bm{f}|_q\|_{L^1(Mw)}
\]
and consequently, if $w\in A_1$ the dependence on the $A_1$ constant is linear. We do not know whether such a precise estimate follows from the sparse control for $\overline{M}_q$ provided in Theorem \ref{ThmSparseMq} above.

\begin{rem}
During the final stage of the elaboration of this work we have learned that some of  the results we provide for the Hardy-Littlewood maximal function have been independently obtained and extended in \cite{HaLo}.
\end{rem}

\subsection{Estimates for vector-valued Calder\'on-Zygmund operators}

Let $0<q<\infty$, $T$ be an $\omega$-
Calder\'on-Zygmund operator (see
Section \ref{subsec:Notation} for the precise definition) and $\bm{f}=\{f_{j}\}_{j=1}^{\infty}$.
We define the vector-valued $\omega$-
Calder\'on-Zygmund operator $\overline{T}_{q}$
as
\[
\overline{T}_{q}\bm{f}(x)=\left(\sum_{j=1}^{\infty}|Tf_{j}(x)|^{q}\right)^{\frac{1}{q}}.
\]
$L^{p}$-boundedness for these operators is obtained,
for example, in \cite{CF}. In that paper the authors point out the
fact that their result yields a new proof of a particular case of
the corresponding result in \cite{BCP} (in that work operators are
defined over Banach space valued functions).
Weighted estimates for the kind of vector-valued Calder\'on-Zygmund operators that we consider
in the current paper were obtained in \cite{CUP} and \cite{PT}.

In \cite{L} A.
Lerner devised a method to control  pointwise Calder\'on-Zygmund operators by sparse
operators that can be extended to other operators provided that their
grand maximal operator is of weak type $(1,1)$ (see Section \ref{subsec:LemmaSparseCZO}
for the details). Following ideas from \cite{L} we are going to establish
a pointwise domination for $\overline{T}_{q}$ which is a natural
analogous to the scalar case.
\begin{thm}
\label{ThmSparseTq}Let $T$ be an $\omega$-
Calder\'on-Zygmund operator with $\omega$ satisfying the Dini condition
and $1<q<\infty$. If $\bm{f}=\{f_{j}\}$ and $|\bm{f}|_{q}\in L^{1}(\mathbb{R}^{n})$
is a compactly supported function, then there exist $3^{n}$ dyadic
lattices $\mathcal{D}_{k}$ and $3^{n}$ $\frac{1}{2}$-sparse families
$\mathcal{S}_{k}\subseteq\mathcal{D}_{k}$. such that
\[
\left|\overline{T}_{q}\bm{f}(x)\right|\leq c_{n}C_{T}\sum_{k=1}^{3^{n}}\mathcal{A}_{\mathcal{S}_{k}}|\bm{f}|_{q}(x),
\]
where
$$\mathcal{A}_{\mathcal{S}}f(x)=\sum_{Q\in\mathcal{S}}\langle |f| \rangle_Q\chi_{Q}(x)$$
and $C_{T}=C_{K}+\|\omega\|_{\text{Dini}}+\|T\|_{L^{2}\rightarrow L^{2}}$.
\end{thm}
With that pointwise domination at our disposal we can derive several
corollaries. Relying upon results in \cite{HL} (see Theorems \ref{ThmFuerteAp} and \ref{ThmDebilAp} in Section \ref{subsec:ApAinfty}) we can obtain
the following weighted estimates.
\begin{thm}\label{Thm:ApT}
Let $1<p,\,q<\infty$, $w$ and $\sigma$ be a pair of weights, and $T$ be
an $\omega$-Calder\'on-Zygmund operator with $\omega$ satisfying the Dini condition. Then
\[
\|\overline{T}_{q}(\sigma \bm{f})\|_{L^{p}(w)}\leq c_{n,p,q}C_T[w,\sigma]_{A_{p}}^{\frac{1}{p}}\left([w]_{A_{\infty}}^{\frac{1}{p'}}+[\sigma]_{A_{\infty}}^{\frac{1}{p}}\right)\||\bm{f}|_{q}\|_{L^{p}(\sigma)}
\]
and also
\[
\|\overline{T}_{q}(\sigma \bm{f})\|_{L^{p,\infty}(w)}\leq c_{n,p,q}C_T[w,\sigma]_{A_{p}}^{\frac{1}{p}}[w]_{A_{\infty}}^{\frac{1}{p'}}\||\bm{f}|_{q}\|_{L^{p}(\sigma)}.
\]
\end{thm}
We observe that the sharp $A_2$ constant had been obtained before in \cite{HaHy} in a more general setting and that J. Scurry \cite{Scu} established the sharp $A_p$ dependence for the strong type estimate for Calder\'on-Zygmund operators satisfying a H\"older-Lipschitz condition.  Our result deals with Dini kernels and as far as we are concerned the mixed $A_p-A_\infty$ has not appeared before. The quantitative weak type estimate is also new.

Using the main result in \cite{DSLR} for sparse operators, we can
also derive an endpoint estimate for arbitrary weights.
\begin{thm}
Let $w$ be a weight and $q>1$. Then for every Young function $\Phi$
we have that
\[
\|\overline{T}_{q}(\bm{f})\|_{L^{1,\infty}(w)}
\lesssim c_{\Phi} \, \int_{\mathbb{R}^{n}}|\bm{f}(x)|_{q}M_{\Phi}w(x)dx,
\]
where $c_{\Phi}=\int_{1}^{\infty}\frac{\Phi^{-1}(t)}{t^{2}\log(e+t)}dt$.
\end{thm}
See Section \ref{subsec:Notation} for the definition of Young function
and $M_{\Phi}$.  The first result of this sort was obtained in \cite{P0}, with a bad control in the bound, recently improved in \cite{HP}.

From the preceding result we can derive the following
Corollary.
\begin{cor}
\label{ThmA1Weak}Let $w$ be a weight and $1<q<\infty$. Then the following
statements hold

\begin{eqnarray}
 \|\overline{T}_{q}(\bm{f})\|_{L^{1,\infty}(w)}
& \lesssim &(1+\log r') \,  \int_{\mathbb{R}^{n}}|\bm{f}(x)|_{q}M_{r}w(x)dx\qquad r>1. \label{eq:Cor1}
\\
\label{eqThmA1Weak2} \|\overline{T}_{q}(\bm{f})\|_{L^{1,\infty}(w)}
&\lesssim&\frac{1}{\varepsilon} \, \int_{\mathbb{R}^{n}}|\bm{f}(x)|_{q}M_{L(\log\log L)^{1+\varepsilon}}w(x)dx\qquad 0<\varepsilon <1.
\end{eqnarray}
Furthermore if $w\in A_{1}$ then
\begin{equation}\label{eq:A1AinfCor1}
\|\overline{T}_{q}(\bm{f})\|_{L^{1,\infty}(w)}
\lesssim[w]_{A_{1}}\log(e+[w]_{A_{\infty}}) \, \int_{\mathbb{R}^{n}}|\bm{f}(x)|_{q}w(x)dx.
\end{equation}

\end{cor}
At this point we would like to stress the fact that \eqref{eqThmA1Weak2} is an improvement
of \cite[Theorem 1.4]{PT} in terms of the maximal operator on the
right hand side of the estimate and also in terms of the quantitative
control of the blow up when $\varepsilon\rightarrow0$ that we obtain.
On the other hand we also would like to observe that the mixed $A_{1}-A_{\infty}$
is the analogous result to the scalar case (\cite{LOP1}) but doesn't
seem to have appeared in the literature before for the vector-valued
setting.

Another interesting result is the $A_{1}$ estimate for $1<p<\infty$,
which is the natural counterpart for the sharp result obtained in
the scalar case in \cite{LOP1}.
\begin{thm}
\label{TheoremA1}Let $p,q\in(1,\infty)$, $T$ be an $\omega$-Calder\'on-Zygmund operator with $\omega$ satisfying the Dini condition and $w$ a weight. Then the following inequality holds
\begin{equation}
||\overline{T}_{q}\bm{f}||_{L^{p}(w)}\leq c_{n,q}C_{T}pp'(r')^{\frac{1}{p'}}\||\bm{f}|_{q}\|_{L^{p}\left(M_{r}w\right)}\qquad r>1.\label{eq:TwoWeightsA1}
\end{equation}
Furthermore, if $w\in A_{1}$ then we have the following estimate
\begin{equation}
||\overline{T}_{q}\bm{f}||_{L^{p}(w)}\leq c_{n,q}C_{T}pp'[w]_{A_{1}}^{\frac{1}{p}}[w]_{A_{\infty}}^{\frac{1}{p'}}\||\bm{f}|_{q}\|_{L^{p}(w)},\label{eq:A1Lp}
\end{equation}
and if $w\in A_{s}$ , with $1\leq s<p<\infty$.  Furthermore, it follows from \cite[Corollary 4.3]{D},
\begin{equation}
||\overline{T}_{q}\bm{f}||_{L^{p}(w)}\leq c_{n,q,p,T}[w]_{A_{s}}\||\bm{f}|_{q}\|_{L^{p}(w)}.\label{eq:DepAs}
\end{equation}
\end{thm}

We will not provide a proof of this result. It suffices to follow, for instance, the arguments in \cite{HP} where \eqref{eq:TwoWeightsA1} is obtained using a sparse domination result in norm.

\begin{rem}
Although (\ref{eq:TwoWeightsA1}) can be generalized replacing $M_{r}$
by a general $M_{\Phi}$ and replacing the constants depending on
$p$ and $r'$ in the right hand side of the inequality by suitable
ones in terms of $M_{\Phi}$ and $p$ (a qualitative version of that
kind of result was obtained in \cite[Theorem 1.1]{PT}), for the sake
of clarity, we prefer to include the proof and the statement of this
estimate just in terms of $M_{r}$. We also remark that we can improve  \eqref{eq:DepAs} to the mixed $A_s-A_\infty$ estimate, see \cite{Li17}.
\end{rem}

\subsection{Estimates for vector-valued commutators}

Let $1<q<\infty$, $T$ be an $\omega$-Calder\'on-Zygmund operator  with $\omega$ satisfying the Dini condition (see
Section \ref{subsec:Notation} for the precise definition). Let $b$ be a
locally integrable function and $\bm{f}=\{f_{j}\}_{j=1}^{\infty}$. We
define the vector-valued commutator $\overline{[b,T]}_{q}$ as
\[
\overline{[b,T]}_{q}\bm{f}(x)=\left(\sum_{j=1}^{\infty}|\overline{[b,T]}f_{j}(x)|^{q}\right)^{\frac{1}{q}}.
\]
This operator was considered in \cite{PT} where some weighted estimates
were obtained.  In this work we will obtain quantitative versions of some of the results
of \cite{PT}. Exploiting the tecnhniques devised by A. Lerner in
\cite{L}, and also A. Lerner, S. Ombrosi and the fourth author in
\cite{LORR} we are able to obtain a pointwise domination for $\overline{[b,T]}_{q}$
by sparse operators in terms of the function and of the symbol which
is a natural analogous to the scalar case.
\begin{thm}
\label{ThmSparseConmm}Let $T$ be an $\omega$-Calder\'on-Zygmund operator with $\omega$ satisfying the Dini condition and $1<q<\infty$. If $\bm{f}=\{f_{j}\}$ and $|\bm{f}|_{q}\in L^{\infty}(\mathbb{R}^{n})$
is a compactly supported function and $b\in L_{\text{loc}}^{1}$,
then there exist $3^{n}$ dyadic lattices $\mathcal{D}_{k}$ and $3^{n}$
$\frac{1}{2}$-sparse families $\mathcal{S}_{k}\subseteq\mathcal{D}_{k}$
such that
\[
|\overline{[b,T]}_{q}\bm{f}(x)|\leq c_{n}C_{T}\sum_{j=1}^{3^{n}}\left(\mathcal{T}_{\mathcal{S},b}|\bm{f}|_{q}(x)+\mathcal{T}_{\mathcal{S},b}^{*}|\bm{f}|_{q}(x)\right),
\]
where
\[
\begin{split}\mathcal{T}_{\mathcal{S},b}f(x) & =\sum_{Q\in\mathcal{S}}|b(x)-b_{Q}|\langle |f| \rangle_{Q}\chi_{Q}(x),\\
\mathcal{T}_{\mathcal{S},b}^{*}f(x) & =\sum_{Q\in\mathcal{S}}\langle |b-b_{Q}||f|\rangle_{Q}\chi_{Q}(x),
\end{split}
\]
 and $C_{T}=C_{K}+\|\omega\|_{\text{Dini}}+\|T\|_{L^{2}\rightarrow L^{2}}$.
\end{thm}

Having the sparse domination in Theorem \ref{ThmSparseConmm} at our
disposal, the same proof of \cite[Theorem 1.1]{LORR} allows us to
derive an endpoint estimate for $\overline{[b,T]}_{q}$ and arbitrary
weights that reads as follows.
\begin{thm}
Let $w$ be a weight and $1<q<\infty$ and $b\in \BMO$. Then for every
Young function $\varphi$ we have that
\[
w\left(\left\{ x\in\mathbb{R}^{n}\,:\,\overline{[b,T]}_{q}\bm{f}(x)>t\right\} \right)\lesssim c_{T}\frac{c_{\varphi}}{t}\int_{\mathbb{R}^{n}}\Phi\left(\|b\|_{\BMO}\frac{|\bm{f}(x)|_{q}}{t}\right)M_{(\Phi\circ\varphi)(L)}w(x)dx,
\]
where $\Phi(t)=t\log(e+t)$ and $c_{\varphi}=\int_{1}^{\infty}\frac{\varphi^{-1}(t)}{t^{2}\log(e+t)}dt$.
\end{thm}
See Section \ref{subsec:Notation} for the definition of Young function
and $M_{\Phi}$. From the preceding result we can derive the following
Corollary.
\begin{cor}
\label{ThmA1Weak-1}Let $w$ be a weight and $1<q<\infty$. Then for every $ 0<\varepsilon<1$,
\begin{eqnarray}\label{eq:Cor2}
w\left(\left\{ x\in\mathbb{R}^{n}\,:\,\overline{[b,T]}_{q}\bm{f}>t\right\} \right)& \lesssim &\frac{1}{\varepsilon}\int_{\mathbb{R}^{n}}
\Phi\left(\frac{|\bm{f}|_q\|b\|_{\BMO}}{t}  \right)  M_{L(\log L)^{1+\varepsilon}}wdx.\qquad\qquad\qquad \nonumber \\
w\left(\left\{ x\in\mathbb{R}^{n}\,:\,\overline{[b,T]}_{q}\bm{f}>t\right\} \right) & \lesssim &\frac{1}{\varepsilon}\int_{\mathbb{R}^{n}}
\Phi\left(\frac{|\bm{f}|_q\|b\|_{\BMO}}{t}  \right) M_{L(\log L)(\log\log L)^{1+\varepsilon}}wdx\label{eq:ThmA1Weak-1}.
\end{eqnarray}
Furthermore if $w\in A_{1}$ then
\begin{equation}
\label{eq:Cor2A1Ainfty}
w\left(\left\{ x\in\mathbb{R}^{n}\,:\,\overline{[b,T]}_{q}\bm{f}(x)>t\right\} \right)\lesssim[w]_{A_{1}} [w]_{A_{\infty}}\log(e+[w]_{A_{\infty}})\int_{\mathbb{R}^{n}}\Phi\left(\frac{|\bm{f}|_q\|b\|_{\BMO}}{t}  \right) wdx.
\end{equation}
\end{cor}
We would like to draw the attention to the fact that \eqref{eq:ThmA1Weak-1} is an improvement
of \cite[Theorem 1.8]{PT} in terms of the maximal operator on the
right hand side of the estimate and also in terms of the quantitative
control of the blow up when $\varepsilon\rightarrow0$ that we obtain.
On the other hand we also would like to observe that the mixed $A_{1}-A_{\infty}$
is the analogous result to the scalar case (\cite{LORR}, see also
\cite{OC,PRRR}) but doesn't seem to have appeared in the literature
before for the vector-valued setting.

Relying upon a generalization of the so called conjugation method that appeared for the first time in \cite{CRW} and was further exploited in \cite{CPP}
(also in \cite{HPAinfty})  we can obtain the following estimate for $\overline{[b,T]}_q$.
\begin{thm}\label{Thm:ApComm}
Let $1<p,q<\infty$, $T$ be an $\omega$-Calder\'on-Zygmund operator with $\omega$ satisfying the Dini condition and $b\in \BMO$. If $w\in A_p$ then
\[\|\overline{[b,T]}_q\bm{f}\|_{L^p(w)}\leq c_{n,p,q}C_T[w]_{A_{p}}^{\frac{1}{p}}\left([w]_{A_{\infty}}^{\frac{1}{p'}}+[\sigma]_{A_{\infty}}^{\frac{1}{p}}\right)\left([w]_{A_{\infty}}+[\sigma]_{A_{\infty}}\right)\|\bm{f}\|_{L^p(w)}.\]
\end{thm}

Another interesting consequence of the sparse representation is an
extension of a two weights estimate for the commutator. In \cite{Bl}
Bloom obtained a two-weighted result for the commutator of the Hilbert
transform $H$. Given $\mu,\lambda\in A_{p}$ $1<p<\infty$, $\nu=\left(\frac{\mu}{\lambda}\right)^{\frac{1}{p}}$
and $b\in \BMO_{\nu}$, then
\[
\|[b,H]f\|_{L^{p}(\lambda)}\leq c(p,\mu,\lambda)\|b\|_{\BMO_{\nu}}\|f\|_{L^{p}(\mu)},
\]
where $\BMO_{\nu}$ is the space of locally integrable functions such
that
\[
\|b\|_{\BMO_{\nu}}=\sup_{Q}\frac{1}{\nu(Q)}\int_{Q}|b-b_{Q}|dx<\infty.
\]
Quite recently I. Holmes, M. Lacey and B. Wick \cite{HLW} extended
Bloom's estimate to Calder\'on-Zygmund operators satisfying a H\"older-Lipschitz
condition. In the particular case when $\mu=\lambda=w\in A_{2}$ their
approach (see \cite{HW}) allowed then to recover the following estimate
\[
\|[b,T]f\|_{L^{2}(w)}\leq c[w]_{A_{2}}^{2}\|f\|_{L^{2}(w)}
\]
due to D. Chung, C. Pereyra and the third author \cite{CPP}.

Even more recently a quantitative version of Bloom-Holmes-Lacey-Wick
estimate was obtained in \cite{LORR}. Here we extend it to the vector-valued
setting.
\begin{thm}
Let $T$ be an $\omega$-Calder\'on Zygmund operator with $\omega$ satisfying
the Dini condition. Let $1<q<\infty$, $\mu,\lambda\in A_{p}$, and
$\nu=\left(\frac{\mu}{\lambda}\right)^{\frac{1}{p}}$. If $b\in \BMO_{\nu}$,
then
\[
\|\overline{[b,T]}_{q}\bm{f}\|_{L^{p}(\lambda)}\leq c_{n,p,q}C_{T}\max\left\{ [\mu]_{A_{p}}[\lambda]_{A_{p}}\right\} ^{\max\left\{ 1,\,\frac{1}{p-1}\right\} }\|b\|_{\BMO_{\nu}}\||\bm{f}|_{q}\|_{L^{p}(\mu)}.
\]
\end{thm}
Another interesting result is the $A_{1}$ estimate for $1<p<\infty$,
which is the natural counterpart for the sharp result obtained in
the scalar case in \cite{OC}.
\begin{thm}
\label{TheoremA1-1}Let $1<p,\,q<\infty$, $T$ be an $\omega$-Calder\'on-Zygmund
operator, $b\in \BMO$ and $w$ be a weight. Then the following inequality
holds
\begin{equation}
||\overline{[b,T]}_{q}\bm{f}||_{L^{p}(w)}\leq c_{n,q}C_{T}\|b\|_{\BMO}\left(pp'\right)^{2}(r')^{1+\frac{1}{p'}}\||\bm{f}|_{q}\|_{L^{p}\left(M_{r}w\right)}\qquad r>1.\label{eq:TwoWeightsA1-1}
\end{equation}
Furthermore, if $w\in A_{1}$ then we have the following estimate
\begin{equation}
||\overline{[b,T]}_{q}\bm{f}||_{L^{p}(w)}\leq c_{n,q}C_{T}\|b\|_{\BMO}\left(pp'\right)^{2}[w]_{A_{1}}^{\frac{1}{p}}[w]_{A_{\infty}}^{1+\frac{1}{p'}}\||\bm{f}|_{q}\|_{L^{p}(w)},\label{eq:A1Lp-1}
\end{equation}
and if $w\in A_{s}$ , with $1\leq s<p<\infty$. Futhermore, it follows from \cite[Corollary 4.3]{D},
\begin{equation}
||\overline{[b,T]}_{q}\bm{f}||_{L^{p}(w)}\leq c_{n,p,q,T}\|b\|_{\BMO}[w]_{A_{s}}^{2}\||\bm{f}|_{q}\|_{L^{p}(w)}.\label{eq:DepAs-1}
\end{equation}
\end{thm}
We don't provide a proof for the preceding result since it suffices to use arguments in \cite{PRRR}.
\begin{rem}
We would like to observe that it is possible
to derive a quantitative estimate similar to \ref{eq:TwoWeightsA1-1}
with $M_{L(\log L)^{2p-1+\delta}}$ on the right hand side in the
spirit of the estimate obtained in \cite{PRRR}, but for the sake
of clarity we state just the case $M_{r}$. We note as well that using the
ideas in \cite{RR2017} it is possible to provide $A_{s}-A_\infty$ estimates.
\end{rem}

\subsection{Estimates for vector-valued rough singular integral operators and commutators}\label{rough}

Let $T_\Omega$ be a rough singular integral operator defined as
\begin{equation*}
T_{\Omega}f(x)=\text{p.v.}\int_{\mathbb{R}^{n}}\frac{\Omega(x')}{|x|^{n}}f(x-y)dy,
\end{equation*}
where $x'=x/{|x|}$ and  $\Omega \in L^1(\mathbb{S}^{n-1})$ satisfies $\int_{\mathbb{S}^{n-1}}\Omega=0$.
We can define the associated maximal operator by
\begin{equation*}
T^*_{\Omega}f(x)=\sup_{\varepsilon>0}\left|\int_{|x-y|>\varepsilon}\frac{\Omega(x')}{|x|^{n}}f(x-y)dy\right|.
\end{equation*}
$T_{\Omega}$ was introduced and also was proved
to be bounded on $L^{p}$ in \cite{CZ}. Later on J. Duoandikoetxea
and J. L. Rubio de Francia \cite{DRdF} established the weighted $L^{p}$
boundedness of $T_{\Omega}$ for $A_{p}$ weights.  It is also a central
fact, due to A. Seeger \cite{See}, that $T_{\Omega}$ is of weak
type $(1,1)$. Fan and Sato \cite{FanSato} settled the weighted endpoint estimate.

Coming back to quantitative weighted estimates, the first works considering
that question for $T_{\Omega}$ are quite recent. The first work in this direction
is due to T. P. Hytönen, L. Roncal and O. Tapiola \cite{HRT}. Their approach relies upon
an elaboration of a classical decomposition of these operators into Dini pieces  combined with interpolation with change of measure (see also \cite{DRdF, DRough, HRT, W}). However, up until now the best known bound, namely
\begin{equation}\label{eq:ApRough}
\begin{split}
\|T_{\Omega}\|_{L^{p}(w)\rightarrow L^{p}(w)}&\leq c_{n,p}\|\Omega\|_{L^{\infty}}[w]_{A_{p}}^{\frac{1}{p}}\left([w]_{A_{\infty}}^{1+\frac{1}{p'}}+[\sigma]_{A_{\infty}}^{1+\frac{1}{p}}\right)\min \left\{[w]_{A_{\infty}},[\sigma]_{A_{\infty}}\right\}\\
&\leq c_{n,p}\|\Omega\|_{L^{\infty}}[w]_{A_p}^{p'},
\end{split}
\end{equation}
was established in \cite{LPRRR}.

Also the first quantitative $A_{1}$
estimate for that class of operators was obtained quite recently by
the third author, L. Roncal and the fourth author \cite{PRRR} and was even more recently improved in \cite{LPRRR} where it was established that the dependence is linear.

In the case of commutators, the conjugation method yields that the weighted $L^p$ estimate is the same as the dependence in \eqref{eq:ApRough} with an extra $[w]_{A_\infty}$ term. For the sake of simplicity, just in terms of the $A_p$ constant we have that
\[
\|[b,T_{\Omega}]\|_{L^{p}(w)\rightarrow L^{p}(w)}\leq c_{n,p}\|\Omega\|_{L^{\infty}}\|b\|_{\BMO}[w]_{A_p}^{1+p'}.
\]
In the case of quantitative $A_1$ estimates the dependence was shown to be at most cubic in \cite{PRR} and very  recently it was proved in \cite{RR2017} that the dependence is quadratic as in the case of Calder\'on-Zygmund operators.

We also recall that the operator $B_{(n-1)/2}$, the Bochner--Riesz multiplier at the critical index,  which is defined by
\begin{equation*}
\widehat{B_{(n-1)/2}(f)}(\xi)= (1-|\xi|^2)_+^{(n-1)/2}\hat f(\xi).
\end{equation*}
satisfies the same estimates mentioned above for $T_\Omega$ (see \cite{LPRRR}).

In a recent paper by Conde-Alonso, Culiuc, Plinio and Ou \cite{CCDO}, the sparse domination formula of integral type for $T_\Omega$ and $B_{(n-1)/2}$ are obtained (see also an alternative proof by Lerner \cite{L2017}). We reformulate it as follows
\begin{thm}[{\cite[Theorems A and B]{CCDO}}]
\label{Thm:Sparse}Let $T$ be $T_\Omega$ or $B_{(n-1)/2}$. Then for all $1<p<\infty$, $f\in L^{p}(\mathbb{R}^{n})$ and $g\in L^{p'}(\mathbb{R}^{n})$,
we have that
\[
\Big|\int_{\mathbb{R}^{n}}T(f)gdx\Big|\leq c_n C_Ts'\sum_{Q\in\mathcal{S}}\langle |f| \rangle_Q\langle |g|\rangle_{s,Q}|Q|,
\]
where $\mathcal{S}$ is a sparse family of some dyadic lattice $\mathcal{D}$,
\[
\begin{cases}
1<s<\infty & \text{if }T=B_{(n-1)/2}\text{ or }T=T_{\Omega}\text{ with }\Omega\in L^{\infty}(\mathbb{S}^{n-1})\\
q' \le s<\infty& \text{if }T=T_{\Omega}\text{ with }\Omega\in L^{q,1}\log L(\mathbb{S}^{n-1})
\end{cases}
\] and
\begin{equation}
\label{eq:CT}
C_{T}=\begin{cases}
\|\Omega\|_{L^{\infty}(\mathbb{S}^{n-1})}, & \text{if }T=T_{\Omega}\text{ with }\Omega\in L^{\infty}(\mathbb{S}^{n-1})\\
\|\Omega\|_{L^{q,1}\log L(\mathbb{S}^{n-1})} & \text{if } \Omega\in L^{q,1}\log L(\mathbb{S}^{n-1})\\
1 & \text{if }T=B_{(n-1)/2}.
\end{cases}
\end{equation}
\end{thm}

For $T^*_\Omega$ with $\Omega\in L^\infty(\mathbb{S}^{n-1})$ the following sparse domination was provided in \cite{DiPHLi}.

\begin{equation}\label{eq:SparseRoughMax}
\left|\int_{\mathbb{R}^n}T_{\Omega}^*(f)g\right|\leq c_n \|\Omega\|_{L^{\infty}(\mathbb{S}^{n-1})}s'\sum_{Q\in\mathcal{S}}\langle |f| \rangle_{s,Q}\langle |g|\rangle_{s,Q}|Q|\qquad 1<s<\infty.
\end{equation}

In the case of commutators, the following result was recently obtained in \cite{RR2017}, hinging upon techniques in  \cite{L2017}.

\begin{thm}\label{Thm:SparseCommRough} Let $T_{\Omega}$ be a rough homogeneous singular integral with $\Omega\in L^{\infty}(\mathbb{S}^{n-1})$.
Then, for every compactly supported $f,g\in \mathcal{C}^\infty({\mathbb R}^n)$ every $b\in \BMO$ and $1<p<\infty$, there exist $3^n$ dyadic lattices $\mathcal{D}_j$ and $3^n$ sparse families ${\mathcal S}_j\subset\mathcal{D}_j$ such that
\begin{equation}
|\langle [b,T_\Omega]f,g \rangle|\le C_ns'\|\Omega\|_{L^{\infty}(\mathbb{S}^{n-1})}\sum_{j=1}^\infty\left(\mathcal{T}_{\mathcal{S}_j,1,s}(b,f,g)+\mathcal{T}^*_{\mathcal{S}_j,1,s}(b,f,g)\right),
\end{equation}
where
\[
\begin{split}
\mathcal{T}_{\mathcal{S}_j,r,s}(b,f,g)&=\sum_{Q\in {\mathcal S}_j}\langle |f|\rangle_{r,Q}\langle |(b-b_Q)g|\rangle_{s,Q}|Q|\\
\mathcal{T}^*_{\mathcal{S}_j,r,s}(b,f,g)&=\sum_{Q\in {\mathcal S}_j}\langle|(b-b_Q)f|\rangle_{r,Q}\langle |g|\rangle_{s,Q}|Q|.
\end{split}
\]
\end{thm}

Analogously as we did in the preceding sections, if $1<q<\infty$ and $T$ is $T_\Omega$ or $B_{(n-1)/2}$ and $b\in L^1_{\text{loc}}$, we consider the corresponding vector-valued versions of $T$ and $[b,T]$ that are defined as follows
\[
\begin{split}
\overline T_q \mathbf f (x)&= \Big( \sum_{j=1}^\infty |T(f_j)|^q \Big)^{\frac 1q}\\
\overline{[b,T]}_q \mathbf f (x)&= \Big( \sum_{j=1}^\infty |[b,T](f_j)|^q \Big)^{\frac 1q}.
\end{split}
\]
Exploiting a beautiful observation due to  Culiuc, Di Plinio and Ou \cite{CDO17} we are able to extend the preceding results to vector-valued setting.
\begin{thm}\label{thm:sparserough}
Let $\Omega\in L^\infty(\mathbb{S}^{n-1})$.
If $T$ is $T_\Omega$ or $B_{(n-1)/2}$ and $1\le s<\frac {q'+1}2$, then there exists a sparse collection $\mathcal S$ such that
\[
\Big|\sum_{j\in\mathbb Z}\int_{\mathbb{R}^{n}}T(f_j)g_jdx \Big|\le c_{n,q} C_Ts' \sum_{Q\in \mathcal S} \langle |\mathbf f|_q \rangle_Q  \langle |\mathbf g|_{q'} \rangle_{s, Q} |Q|.
\]
If $1<s<\frac{\min\{q',q\}+1}{2}$, then there exists a sparse collection $\mathcal S$ such that
\[
\Big|\sum_{j\in\mathbb Z}\int_{\mathbb{R}^{n}}T^*_\Omega(f_j)g_jdx \Big|\le c_{n,q} \|\Omega\|_{L^{\infty}(\mathbb{S}^{n-1})}s' \sum_{Q\in \mathcal S} \langle |\mathbf f|_q \rangle_{s,Q}  \langle |\mathbf g|_{q'} \rangle_{s, Q} |Q|.
\]
If $1<s<\frac{q'+1}{2}$, $1<r<{\frac{q+1}2}$ and $b\in\BMO$ then
\[
\Big|\sum_{j\in\mathbb Z}\int_{\mathbb{R}^{n}}[b,T_\Omega](f_j)g_jdx \Big|\le c_{n,q} \|b\|_{\BMO} \|\Omega\|_{L^{\infty}(\mathbb{S}^{n-1})} s'\max\{r',s'\} \sum_{Q\in \mathcal S} \langle |\mathbf f|_q \rangle_{r,Q}  \langle |\mathbf g|_{q'} \rangle_{s, Q} |Q|.
\]
\end{thm}

The preceding result combined with arguments  in \cite{DiPHLi} and in \cite{LPRRR} allow us to establish the following result that provides the natural counterparts to the results mentioned above.

\begin{thm}\label{v-valued-rough}
Let $\Omega\in L^\infty(\mathbb{S}^{n-1})$  and $T=T_\Omega\text{ or }B_{(n-1)/2}$. Then
\begin{enumerate}
\item For any $1<p<\infty$,
\[\begin{split}
\|\overline{T}_q\|_{L^p(w)}&\leq c_{n,p,q} \ c_T[w]_{A_p}^\frac{1}{p}([w]_{A_\infty}^\frac{1}{p'}+[\sigma]_{A_\infty}^\frac{1}{p})\min\{[\sigma]_{A_\infty},[w]_{A_\infty}\},\\
\|\overline{(T^*_\Omega)}_q\|_{L^p(w)}&\leq c_{n,p,q}\|\Omega\|_{L^\infty(\mathbb{S}^{n-1})}[w]_{A_p}^\frac{1}{p}([w]_{A_\infty}^\frac{1}{p'}+[\sigma]_{A_\infty}^\frac{1}{p})\max\{[\sigma]_{A_\infty},[w]_{A_\infty}\},\\
 \|\overline{[b,T_\Omega]}_q\|_{L^p(w)}&\leq c_{n,p,q}\|\Omega\|_{L^\infty(\mathbb{S}^{n-1})}\|b\|_{\BMO}[w]_{A_p}^\frac{1}{p}([w]_{A_\infty}^\frac{1}{p'}+[\sigma]_{A_\infty}^\frac{1}{p})\max\{[\sigma]_{A_\infty},[w]_{A_\infty}\}^2.
\end{split}
\]
Consequently
\[
\begin{split}
\|\overline{T}_q\|_{L^p(w)}&\leq c_{n,p,q}\ c_T[w]_{A_p}^{p'},\\
\|\overline{(T^*_\Omega)}_q\|_{L^p(w)}&\leq c_{n,p,q}\|\Omega\|_{L^\infty(\mathbb{S}^{n-1})}[w]_{A_p}^{2\max\left\{1,\frac{1}{p-1}\right\}},\\
 \|\overline{[b,T_\Omega]}_q\|_{L^p(w)}&\leq c_{n,p,q}\|\Omega\|_{L^\infty(\mathbb{S}^{n-1})}\|b\|_{\BMO}[w]_{A_p}^{3\max\left\{1,\frac{1}{p-1}\right\}}.
\end{split}
\]
\item The following Fefferman-Stein type inequalities for $1<r$ small enough hold
\[
\begin{split}
\|\overline{T}_q(\bm{f})\|_{L^p(w)}&\leq c_{n,p,q} c_T(r')^\frac{1}{p'}\||\bm{f}|_q\|_{L^p(M_rw)},\\
\|\overline{(T^*_\Omega)}_q(\bm{f})\|_{L^p(w)}&\leq c_{n,p,q}\|\Omega\|_{L^\infty(\mathbb{S}^{n-1})}(r')^{1+\frac{1}{p'}}\||\bm{f}|_q\|_{L^p(M_rw)},\\
 \|\overline{[b,T_\Omega]}_q\|_{L^p(w)}&\leq c_{n,p,q}\|b\|_{\BMO}\|\Omega\|_{L^\infty(\mathbb{S}^{n-1})}(r')^{2+\frac{1}{p'}}\||\bm{f}|_q\|_{L^p(M_rw)}.\end{split}
\]
\item If $1<s<p$ then
\[
\begin{split}
\|\overline{T}_q(\bm{f})\|_{L^p(w)}&\leq c_{n,p,q}[w]_{A_s}^{\frac{1}{p}}[w]_{A_\infty}^{\frac{1}{p'}} \||\bm{f}|_q\|_{L^p(w)},\\
\|\overline{(T^*_\Omega)}_q(\bm{f})\|_{L^p(w)}&\leq c_{n,p,q}\|\Omega\|_{L^\infty(\mathbb{S}^{n-1})}[w]_{A_s}^{\frac{1}{p}}[w]_{A_\infty}^{1+\frac{1}{p'}} \||\bm{f}|_q\|_{L^p(w)},\\
 \|\overline{[b,T_\Omega]}_q(\bm{f})\|_{L^p(w)}&\leq c_{n,p,q}\|b\|_{\BMO}\|\Omega\|_{L^\infty(\mathbb{S}^{n-1})}[w]_{A_s}^{\frac{1}{p}}[w]_{A_\infty}^{2+\frac{1}{p'}} \||\bm{f}|_q\|_{L^p(w)}.
 \end{split}
\]

\end{enumerate}

\end{thm}

\subsection{Coifman-Fefferman estimates and the $C_p$ condition} \label{Cp}

Calder\'on-Zygmund principle states that for each singular operator there exists
a maximal operator that ``controls'' it. A paradigmatic example
of that principle is the Coifman-Fefferman estimate, namely, for each
$0<p<\infty$ and every $w\in A_{\infty}$ there exists $c=c_{n,w,p}>0$
such that
\begin{equation}\label{eq:CF}
\|T^*f\|_{L^p(w)}\leq c\|Mf\|_{L^p(w)}.
\end{equation}
where $T^{*}$ stands for the maximal Calder\'on-Zygmund operator.  This kind of estimate plays a central role in modern euclidean Harmonic Analysis. In particular we emphasize the key role in the main result in \cite{LOP2}. \eqref{eq:CF}  raises a natural question,
is the $A_\infty$ condition necessary for \eqref{eq:CF} to hold? B. Muckenhoupt \cite{M} provided a negative answer to the question.
 He proved that in the case when $T$ is the Hilbert transform, \eqref{eq:CF} does not imply that $w$ satisfies the $A_{\infty}$ condition. He showed that if \eqref{eq:CF} holds with $p>1$ and $T$ is the Hilbert transform,  then $w\in C_p$, that is if  there exist $c,\delta>0$ such that for every cube $Q$ and every subset $E\subseteq Q$ we have that
\begin{equation}\label{eq:CpIntro}
w(E)\leq c\left(\frac{|E|}{|Q|}\right)^\delta \int_{\mathbb{R}^n}M(\chi_Q)^pw.
\end{equation}
Observe that that $A_\infty\subset C_p$ for every $p>0$. B. Muckenhoupt showed, in dimension one, that if $w\in A_p$, $1<p<\infty$, then $w\chi_{ [0,\infty) } \in C_p$.  In the same paper it was conjectured that the $C_p$ condition is also sufficient for \eqref{eq:CF} to hold, question which remains still open. Not much later, the necessity of the  $C_p$ condition was extended to arbitrary  dimension and a converse result was provided by E.T. Sawyer \cite{S}. More precisely
he proved the following result.

\begin{thm}[{E. Sawyer} \cite{S}] \label{sawyer}
Let $1<p<\infty$ and let  $w\in C_{p+\epsilon}$ for some $\epsilon>0$. Then
\begin{equation}\label{sawyerCp}
\|T^*(f)\|_{L^p(w)} \leq c\|Mf\|_{L^p(w)}
\end{equation}
\end{thm}

Also, relying upon Sawyer's techniques, K. Yabuta \cite[Theorem 2]{Y} established the following result extending the classical result of C. Fefferman and E. Stein relating $M$ and the sharp maximal $M^{\#}$ function \cite{FS2,Jo} that we state as a lemma.

\begin{lem}  [{K. Yabuta} \cite{Y}]   \label{yabuta}
Let $1<p<\infty$ and let  $w\in C_{p+\epsilon}$ for some $\epsilon>0$. Then
\begin{equation}\label{eq:Ya}
\|M(f)\|_{L^p(w)} \leq c\|M^{\#}f\|_{L^p(w)}
\end{equation}
\end{lem}

The proof of this result, although based on a key lemma from \cite{S}, is simpler than the proof of  \eqref{sawyerCp} by Sawyer. In this paper we will present a different approach for proving \eqref{sawyerCp} based on Yabuta's lemma  which is conceptually much simpler and  much flexible. Furthermore, we extend estimate \eqref{sawyerCp} to the full expected range, namely $0<p<\infty$ and to some vector-valued operators.  We remark that in the last case, the classical good-$\lambda$ seems not be applicable. None of the known  methods yield this result.

\begin{rem}
We remark that we do not know how to extend Lemma \ref{yabuta} to the full range  $0<p<\infty$ as in Theorems \ref{strong p} and \ref{strong pmultilinear} below.  However, this lemma is key for proving these theorems in the full range.
\end{rem}

We remark that more recently, A. Lerner \cite{L3} provided another proof of Yabuta's result \eqref{eq:Ya}  improving it slightly. He established, using a different argument, that if a weight $w$ satisfies the following estimate
\[w(E)\leq\left(\frac{|E|}{|Q|}\right)^\delta \int_{\mathbb{R}^n}\varphi_p\left(M(\chi_Q)\right)^pw,\]
where
\[\int_{0}^1\varphi_p(t)\frac{dt}{t^{p+1}}<\infty\]
then \eqref{eq:Ya} holds.

Let us now turn attention to our contribution.
We say that an operator $T$ satisfies the
condition (D) if  there are  some constants, $\delta \in (0,1)$ and  $c>0$ such
that for all $f$,
\[
\tag{$D$}
M_\delta^\#(Tf)(x) \leq c Mf(x).
\]

Some examples of operators satisfying condition ($D$) are:
\begin{itemize}
\item
{\bf Calder\'on-Zygmund operators} These operators are
generalization of the regular singular integral operators as
defined above. This was observed in
\cite{AP}.

\item {\bf Weakly strongly singular integral operators}
These operators were considered by C. Fefferman in \cite{F}.

\item {\bf Pseudo-differential operators}. To be more precise, the
pseudo-differential operators satisfying condition (D) are those that
belong to the H\"ormander class (\cite{Ho}).

\item {\bf  Oscillatory integral operators}
These operators were by introduced by Phong and Stein \cite{PS}.
\end{itemize}
The proof that last three cases satisfy condition $(D)$ can be found in \cite{AHP}.

It is also possible to consider a suitable variant of condition $(D)$ which will allow us to treat some vector-valued operators. Given an operator $T$ we say that it satisfies the $(D_q)$ condition with $1<q<\infty$ if for every $0<\delta<1$ there exists a finite constant
$c=C_{\delta,q,T}>$ such that

\begin{equation}\tag{$D_q$}
M^\#_\delta \Big( \overline{T}_{q} \bm{f}\Big )(x)
\leq
c\, M( |\bm{f}|_q )(x).
\end{equation}

Two examples of operators satisfying the $(D_q)$ condition are the Hardy-Littlewood maximal operator  (\cite{CGMP} Proposition 4.4)
and any  Calder\'on-Zygmund operator (\cite{PT} Lemma 3.1).

Next theorems extend and improve  the main result from  \cite{S} since we are able to provide some answers for the range $0<p\leq 1$ and to consider vector-valued extensions.
It is not clear that the method of \cite{S} can be extended to cover both situations. Furthermore, we can extend this result to the multilinear context and other operators like fractional integrals.

\begin{thm} \label{strong p} Let $T$ be an operator satisfying the (D) condition. Let $0< p<\infty$ and  let  $w \in C_{\max\{1,p\}+\epsilon}$
for some $\epsilon>0$. Then
\begin{equation}\label{extended strong Cp}
\|Tf\|_{L^p(w)} \leq c\|Mf\|_{L^p(w)}.
\end{equation}
Additionally, if $1<q<\infty$ and $T$ satisfies the $(D_q)$ condition then
\begin{equation*}
\|\overline{T}_q\bm{f}\|_{L^p(w)} \leq c\|M(|\bm{f}|_q)\|_{L^p(w)}.
\end{equation*}
\end{thm}

\begin{rem}
We don't know how to extend \eqref{extended strong Cp} to rough singular integral operators or to the Bochner--Riesz multiplier at the critical index as considered in Section \ref{rough}. Indeed, it is not known whether any of these operators satisfies condition  ($D$) above.
\end{rem}

\begin{rem}  \label{fractional case }
Following a similar strategy as in the proof of \eqref{extended strong Cp} the following result holds. Let $I_{{\alpha}}$, $0<\alpha<n$, be a  fractional operator and let $1<p<\infty$. Let  $w \in C_{p+\epsilon}$
for some $\epsilon>0$. Then
\begin{equation}\label{FractionalCp}
\|I_{\alpha}f\|_{L^p(w)} \leq c\|M_{\alpha}f\|_{L^p(w)}.
\end{equation}
\end{rem}

It is possible to extend these kind of results to the multilinear setting as follows.  Following \cite{GT1}, we say that $T$ is an $m$-linear
Calder\'on-Zygmund operator if, for some $1\le q_j<\infty$, it
extends to a bounded multilinear operator from
$L^{q_1}\times\dots\times L^{q_m}$ to $L^q$, where
$\frac1q=\frac{1}{q_1}+\dots+\frac{1}{q_m}$, and if
 there exists a function
$K$, defined
 off the diagonal $x=y_1=\dots=y_m$ in $({\mathbb
R}^n)^{m+1}$,  satisfying
$$T(f_1,\dots,f_m)(x) =\int\limits_{({\mathbb R}^{n})^m}
K(x,y_1,\dots,y_m)f_1(y_1)\dots f_m(y_m)\,dy_1\dots dy_m
$$
for all $x \notin \cap_{j=1}^{m } \supp f_j$; and also satisfies similar size and regularity conditions as that in Section \ref{subsec:Notation}.

It was shown in \cite{LOPTT}, following the Calder\'on-Zygmund principle mentioned above, that the right maximal operator
that ``controls'' these $m$-linear
Calder\'on-Zygmund operators is defined by

$$\mathcal M(\vec f\,)(x)=\sup_{Q\ni
x}\prod_{i=1}^m\frac{1}{|Q|}\int_Q|f_i(y_i)|dy_i,$$
where  $\vec f=(f_1,\dots,f_m)$ and where the supremum is taken over all cubes $Q$ containing $x$. In fact, these  $m$-linear Calder\'on-Zygmund operators satisfy a version of the $(D)$ condition mentioned above as can be found in \cite[Theorem 3.2]{LOPTT}.

\begin{lem}
Let $T$ be an m-linear Calder\'on-Zygmund operator and $\delta \in (0,\frac1m)$. Then, there is a constant $c$ such that
\begin{equation}\label{pointw}
M_{\delta}^{\#}(T(\vec f\,))(x)\le c\, {\mathcal M}(\vec f\,)(x).
\end{equation}
\end{lem}

This estimate is very sharp since it is false in the case $\delta=\frac1m$. Also this estimate is quite useful since one can deduce the following multilinear version of Coifman-Fefferman estimate \eqref{eq:CF},
\begin{equation*}
\| T(\vec f) \|_{L^p(w)} \leq c\, \|\mM(\vec f)\|_{L^p(w)} \qquad 0<p<\infty, \, w\in A_{\infty},
\end{equation*}
which can be found in \cite{LOPTT}  leading to the characterization of the class of (multilinear) weights for which any mulitilinear Calder\'on-Zygmund operators are bounded.

Relying upon the pointwise estimate \eqref{pointw} it is possible to establish the following extension of \eqref{extended strong Cp}.
\begin{thm} \label{strong pmultilinear}

Let $T$ be an m-linear Calder\'on-Zygmund operator, and let $0< p<\infty$. Also let  $w \in C_{m\max\{1,p\}+\epsilon}$
for some $\epsilon>0$. Then
\begin{equation}\label{strongcf}
\| T(\vec f) \|_{L^p(w)} \leq c\, \|\mM(\vec f)\|_{L^p(w)}.
\end{equation}
\end{thm}
We emphasize that the method of Sawyer in \cite{S} does not produce the preceding result even for the case $p>1$.

For commutators, the following estimates are known (see \cite{P95,PT}). For every $0<\varepsilon<\delta<1$,
\begin{eqnarray}
M^{\#}_{\varepsilon}([b,T]f)(x)&\le& c_{\delta,T} \|b\|_{BMO} \left( M_{\delta}(Tf)   +M^2(f)(x) \right), \label{Pointwise}\\
M^{\#}_{\varepsilon}(\overline{[b,T]}_q\bm{f})(x)&\le& c_{\delta,T} \|b\|_{BMO} \left( M_{\delta}(\overline{T}_q\bm{f})   +M^2(|\bm{f}|_q)(x) \right), \qquad  1<q<\infty , \label{Pointwiseq}
\end{eqnarray}
where $T$ is a Calder\'on-Zygmund operator satisfying a log-Dini condition. Relying upon them we obtain the following result.
\begin{thm} \label{commutator} Let $T$ be an $\omega$-Calder\'on-Zygmund operator with $\omega$ satisfying a log-Dini condition and let $b\in \BMO$.  Let  $0<p<\infty$ and  let  $w \in C_{\max\{1,p\}+\epsilon}$
for some $\epsilon>0$. Then there is a constant $c$ depending on the  $C_{\max\{1,p\}+\epsilon}$ condition such that
\begin{equation*}
\|[b,T]f\|_{L^p(w)} \leq c\,\|b\|_{\BMO}\, \|M^2f\|_{L^p(w)}.
\end{equation*}
Additionally, if $1<q<\infty$ then Then there is a constant $c$ depending on the  $C_{\max\{1,p\}+\epsilon}$ condition such that
\begin{equation*}
\|\overline{[b,T]}_q\bm{f}\|_{L^p(w)} \leq c\,\|b\|_{\BMO}\, \|M^2(|\bm{f}|_q)\|_{L^p(w)}.
\end{equation*}
\end{thm}

\begin{rem} 
We remark that a similar estimate can be derived for the general k-th iterated commutator: let $0<p<\infty$ and  let  $w \in C_{\max\{1,p\}+\epsilon}$ for some $\epsilon>0$. Then  there is a constant $c$ depending on the $C_{\max\{1,p\}+\epsilon}$ condition such that
$$
\|T_b^{k}f\|_{L^p(w)} \leq c\,\|b\|^k_{\BMO}\, \|M^{k+1}f\|_{L^p(w)}.
$$
\end{rem}

In the following results we observe that rephrasing Sawyer's method \cite{S} in combination with sparse domination we will be able to obtain  estimates like \eqref{sawyerCp} where the strong norm $\|\cdot\|_{L^p(w)}$ is replaced by the weak norm $\|\cdot\|_{L^{p,\infty}(w)}$.   However, we can only get the result only within the range $1<p<\infty$.

\begin{thm}\label{Thm:Cp}Let $T$ be an $\omega$-Calder\'on-Zygmund operator with $\omega$ satisfying the Dini condition. Let $1< p<\infty$ and let  $w\in C_{p+\epsilon}$ for some $\epsilon>0$. Then there exists $c=c_{T,p,\epsilon,w}$ such that
\begin{equation*}
\| T f\|_{L^{p,\infty}(w)}\leq c\|Mf\|_{L^{p,\infty}(w)}.
\end{equation*}
If aditionally $1<q<\infty$ then
\begin{equation*}
\|\overline{T}_q \bm{f}\|_{L^{p,\infty}(w)} \leq c\|M(|\bm{f}|_q)\|_{L^{p,\infty}(w)}.
\end{equation*}
\end{thm}
We also obtain some results for commutators which are completely new in both the scalar and the vector valued case.
\begin{thm}\label{Thm:CpComm}Let $T$ be an $\omega$-Calder\'on-Zygmund operator with $\omega$ satisfying a Dini condition and $b\in \BMO$.
Let $1< p<\infty$ and $w\in C_{p+\epsilon}$ for some $\epsilon>0$. Then there exists $c=c_{T,p,\epsilon,w}$ such that
\[\begin{split}
\| [b,T] f\|_{L^p(w)}&\leq c\|b\|_{\BMO}\|M^2f\|_{L^p(w)},\\
\| [b,T] f\|_{L^{p,\infty}(w)}&\leq c\|b\|_{\BMO}\|M^2f\|_{L^{p,\infty}(w)}.
\end{split}\]
If aditionally $1<q<\infty$ then
\[\begin{split}
\|\overline{[b,T]}_q \bm{f}\|_{L^p(w)}&\leq c\|b\|_{\BMO}\|M^2(|\bm{f}|_q)\|_{L^p(w)},\\
\|\overline{[b,T]}_q \bm{f}\|_{L^{p,\infty}(w)}&\leq c\|b\|_{\BMO}\|M^2(|\bm{f}|_q)\|_{L^{p,\infty}(w)}.
\end{split}\]
\end{thm}

We would like to note that the preceding result extends results based in the $M^\sharp$ approach that hold for Calder\'on-Zygmund operators satisfying a $\log$-Dini condition to operators satisfying just a Dini condition.

\subsection{Local exponential decay estimates for vector valued operators}
In the previous section we recalled the classical Coifman-Fefferman estimate, namely that if $w\in A_\infty$ then $\eqref{eq:CF}$ holds. To obtain such an estimate Coifman and Fefferman used the good-$\lambda$
technique introduced by Burkholder and Gundy \cite{BG}. The use of
that technique was based on the following estimate
\[
\left|\left\{ x\in\mathbb{R}^{n}\,:\,T^{*}f(x)>2\lambda,\,Mf(x)\leq\lambda\gamma\right\} \right|\leq c\gamma\left|\left\{ x\in\mathbb{R}^{n}\,:\,T^{*}f(x)>\lambda \right\} \right|.
\]
The first step to prove such an inequality is to localize the level set $\left\{ x\in\mathbb{R}^{n}\,:\,T^{*}f(x)>\lambda \right\} $
by means of a Whitney decomposition. Then the problem is reduced to studying a
local estimate, namely
\[
\left|\left\{ x\in Q\,:\,T^{*}f(x)>2\lambda,\,Mf(x)\leq\lambda\gamma\right\} \right|\leq c\gamma\left|Q\right|
\]
where each $Q$ is a Whitney cube and $f$ is supported on $Q$.

In \cite{B}, trying to obtain a quantitative weighted estimate for
Calder\'on-Zygmund operators by means of the good-$\lambda$ technique,
Buckley obtained an exponential decay in $\gamma$ that can be stated
as follows,
\begin{equation}\label{expoDecayGoodlambda}
\left|\left\{ x\in Q\,:\,T^{*}f(x)>2\lambda,\,Mf(x)\leq\lambda\gamma\right\} \right|\leq ce^{-\frac{c}{\gamma}}\left|Q\right|.
\end{equation}
This result is interesting for many reasons. One of these reasons is because it yields an estimate like
\begin{equation*}
\|Tf\|_{L^p(w)} \leq cp[w]_{A_{\infty}}\|Mf\|_{L^p(w)},
\end{equation*}
which is optimal in both $p$ and $[w]_{A_{\infty}}$. These estimates play a major role in the $L \log L$ estimate derived in  \cite{LOP2}.

Later on, Karagulyan \cite{Ka} obtained an improved version of  \eqref{expoDecayGoodlambda}, namely, he proved that there exist constants $\alpha,c>0$ depending just on  $T$ and $n$ such that
\[
\left|\left\{ x\in Q\,:\,T^{*}f(x)>tMf(x)\right\} \right|\leq ce^{-\alpha t}\left|Q\right| \qquad t>0.
\]
This estimate was later generalized, using other methods, to several operators  by C. Ortiz-Caraballo, the third author and E. Rela in \cite{OCPR}.

Now we come to our contribution in this paper. Our first result is the following.
\begin{thm}
\label{Thm:LocalDecay}Let $1<q<\infty$, $T$ be an $\omega$-Calder\'on-Zygmund
operator with $\omega$ satisfying the Dini condition and $b\in \BMO$. Assume also that $\supp|f|_{q}\subseteq Q$.
Then
\begin{align}
\left|\left\{ x\in Q\,:\,\overline{M_{q}}\bm{f}(x)>tM\left(|\bm{f}|_{q}\right)(x)\right\} \right| & \leq c_{1}e^{-c_{2}t^{q}}|Q|,\label{eq:MqLoc}\\
\left|\left\{ x\in Q\,:\,\overline{T_{q}}\bm{f}(x)>tM\left(|\bm{f}|_{q}\right)(x)\right\} \right| & \leq c_{1}e^{-c_{2}t}|Q|,\label{eq:TqLoc}\\
\left|\left\{ x\in Q\,:\,\left|\overline{[b,T]_{q}}\bm{f}(x)\right|>tM^{2}\left(|\bm{f}|_{q}\right)(x)\right\} \right| & \leq c_{1}e^{-\sqrt{c_{2}\frac{t}{\|b\|_{\BMO}}}}|Q|.\label{eq:bTqLoc}
\end{align}
\end{thm}
That theorem is a direct consequence of the combination of the sparse
domination results obtained in the previous section and the following
Lemma.
\begin{lem}
\label{Lem:Decay}Let $\mathcal{S}$ be an $\eta$-sparse family and $b\in \BMO$.
Then if $\supp f\subseteq Q$ we have that
\[
\begin{split}\left|\left\{ x\in Q\,:\,\mathcal{A}_{\mathcal{S}}^{r}|f|>tMf(x)\right\} \right| & \leq c_{1}e^{-c_{2}t^{r}}|Q|\quad r>0,\\
\left|\left\{ x\in Q\,:\,\mathcal{T}_{\mathcal{S},b}f(x)+\mathcal{T}_{\mathcal{S},b}^{*}f(x)>tM^{2}f(x)\right\} \right| & \leq c_{1}e^{-\sqrt{c_{2}\frac{t}{\|b\|_{\BMO}}}}|Q|.
\end{split}
\]
where $c_{1},c_{2}$ only depend on $n$ and $\eta$.
\end{lem}
At this point we would like to point out that both (\ref{eq:MqLoc})
and (\ref{eq:TqLoc}) were obtained in \cite{OCPR}.   Here we provide
a unified proof of those estimates. In the case of (\ref{eq:bTqLoc})
the estimate is new. We would also like to remark that combining Lemma
\ref{Lem:Decay} and the sparse domination of $[b,T]$ we also obtain
a new proof for the corresponding estimate already obtained in \cite{OCPR}.

We will finish this section with a similar type result for rough singular integrals.
\begin{thm}
Let $\Omega\in L^\infty(\mathbb{S}^{n-1})$  and $T=T_\Omega\text{ or }B_{(n-1)/2}$. Let also $Q$ be a cube and $f$ such that
$\supp f\subseteq Q$, then there exist some constants $c, \alpha >0$ such that
$$
\left|\left\{ x\in Q\,:|Tf(x)|>tMf(x)\right\} \right|\leq ce^{-\sqrt{\alpha t}}\left|Q\right|, \qquad t>0.
$$

\end{thm}

\begin{rem}
We believe that the preceding estimate is not sharp, we conjecture that the decay should be exponential instead of subexponential.
\end{rem}

The rest of the paper is organized as follows. In
Section \ref{sec:Preliminaries} we gather some basic definitions
and properties and fix some notation. Section \ref{sec:ProofsMax}
is devoted to the the proofs of the results for the vector-valued
extensions of Hardy-Littlewood maximal operator. Section \ref{sec:Sparse}
contains the proof of the sparse domination for both vector-valued
Calder\'on-Zygmund operators and commutators. A proof of Theorem \ref{Thm:ApComm} is provided in section \ref{Sec:ProofApComm}.
In section \ref{Sec:ProofThmCp} we give the proofs of the results related to the $C_p$ condition. Section \ref{Sec:ProofLocalDecay} is devoted to establish the local exponential and sub-exponential estimates. We end this paper with an Appendix that contains precise statements of some quantitative unweighted estimates that are essentially implicit
in the literature and that are needed for our fully-quantitative sparse
estimates.

\section{Preliminaries\label{sec:Preliminaries}}

\subsection{\label{subsec:Notation}Notations and basic definitions}

In this Section we fix the notation that we will use in the rest of
the paper. First we recall the definition of $\omega$-Calder\'on-Zygmund
operator.
\begin{defn}
A Calder\'on-Zygmund operator $T$ is a linear operator bounded
on $L^{2}(\mathbb{R}^{n})$ that admits the following representation
\[
Tf(x)=\int K(x,y)f(y)dy
\]
with $f\in\mathcal{C}_{c}^{\infty}(\mathbb{R}^{n})$ and $x\not\in\text{supp }f$
and where $K:\mathbb{R}^{n}\times\mathbb{R}^{n}\setminus\{(x,x)\,:\,x\in\mathbb{R}^{n}\}\longrightarrow\mathbb{R}$
has the following properties

\begin{description}
\item [{Size~condition}] $|K(x,y)|\leq C_{K}\frac{1}{|x-y|^{n}}$, $\qquad x\not=0$.
\item [{Smoothness~condition}] Provided that $|y-z|<\frac{1}{2}|x-y|$, then
\[
|K(x,y)-K(x,z)|+|K(x,y)-K(z,y)|\leq\frac{1}{|x-y|^{n}}\omega\left(\frac{|y-z|}{|x-y|}\right),
\]
where the modulus of continuity $\omega:[0,\infty)\rightarrow[0,\infty)$
is a subaditive, increasing function such that $\omega(0)=0$.
\end{description}
\end{defn}

It is possible to impose different conditions on the modulus of continuity $\omega$. The most general one is the
Dini condition. We say that a modulus of continuity $\omega$ satisfies a Dini condition if \[\|\omega\|_{\text{Dini}}=\int_{0}^{1}\omega(t)\frac{dt}{t}<\infty.\]
We will say that the modulus of continuity $\omega$ satisfies a log-Dini condition if
\[\|\omega\|_{\text{log-Dini}}=\int_{0}^{1}\omega(t)\log\left(\frac{1}{t}\right)\frac{dt}{t}<\infty.\]
Clearly $\|\omega\|_{\text{Dini}}\leq \|\omega\|_{\text{log-Dini}}$
We recall also that if $\omega(t)=ct^{\delta}$ we are in the case of the classical H\"older-Lipschitz
condition.


\begin{defn}
Let $\omega$ be a modulus of continuity and $K$ be a kernel satisfying
the properties in the preceding definition. We define the maximal
 Calder\'on-Zygmund operator $T^{*}$ as
\[
T^{*}f(x)=\sup_{\varepsilon>0}\left|\int_{|x-y|>\varepsilon}K(x,y)f(y)dy\right|.
\]
\end{defn}
Given $1<q<\infty$ and $\bm{f}=\{f_{j}\}_{j=1}^{\infty}$ we will call
$|\bm{f}|_{q}$ and we will denote indistinctly $|\bm{f}(x)|_{q}=|\bm{f}|_{q}(x)$
the function defined as
\[
|\bm{f}(x)|_{q}=\left(\sum_{j=1}^{\infty}|f_{j}(x)|^{q}\right)^{\frac{1}{q}}.
\]
Similarly we will define extensions for the Hardy-Littlewood maximal
operator and for Calder\'on-Zygmund operators.
\begin{defn}
Let  $\mathcal{D}$ be a dyadic lattice and $1<q<\infty$. We
define the vector-valued Hardy-Littlewood maximal operator $\overline{M}_{q}^{\mathcal{D}}$
as
\[
\overline{M}_{q}^{\mathcal{D}}\bm{f}(x)=\left(\sum_{j=1}^{\infty}M^{\mathcal{D}}f_{j}(x)^{q}\right)^{\frac{1}{q}},
\]
where
\[
M^{\mathcal{D}}f(x)=\sup_{x\in Q\in\mathcal{D}}\frac{1}{|Q|}\int_{Q}|f(y)|dy.
\]
\end{defn}

\begin{defn}
Let $1<q<\infty$, $T$ be an $\omega$-Calder\'on-Zygmund operator with $\omega$ satisfying the Dini condition and
$\bm{f}=\{f_{j}\}$. We define the vector-valued maximal $\omega$-Calder\'on-Zygmund
operator $\overline{T^*}_{q}$ as
\[
\overline{T^*}_{q}\bm{f}(x)=\left(\sum_{j=1}^{\infty}|T^{*}f_{j}(x)|^{q}\right)^{\frac{1}{q}}.
\]
\end{defn}

To end the Section we would like to recall also the definitions of
some variants and ge\-ne\-ra\-li\-za\-tions of the Hardy-Littlewood maximal
function. We will denote $M_{s}f(x)=M(|f|^{s})(x)^{\frac{1}{s}}$
where $s>0$.

Now we recall that we say that $\Phi$ is a Young function if it is
a continuous, convex increasing function that satisfies $\Phi(0)=0$
and such that $\Phi(t)\rightarrow\infty$ as $t\rightarrow\infty$.

Let $f$ be a measurable function defined on a set $E\subset\mathbb{R}^{n}$
with finite Lebesgue measure. The $\Phi$-norm of $f$ over $E$ is
defined by
\[
\|f\|_{\Phi(L),E}:=\inf\left\{ \lambda>0\,:\,\frac{1}{|E|}\int_{E}\Phi\left(\frac{|f(x)|}{\lambda}\right)dx\leq1\right\}.
\]
Using this $\Phi$-norm we define, in the natural way, the Orlicz
maximal operator $M_{\Phi(L)}$ by
\[
M_{\Phi(L)}f(x)=\sup_{x\in Q}\|f\|_{\Phi(L),Q}.
\]
Some particular cases of interest are
\begin{itemize}
\item $M_{r}$ for $r>1$ given by the Young function $\Phi(t)=t^{r}$.
\item $M_{L(\log L)^{\delta}}$ with $\delta>0$  given
by the Young function $\Phi(t)=t\log(e+t)^{\delta}$. It is a well known fact that
\[M^{(k+1)}f\simeq M_{L(\log L)^k}f,\]where $M^k=M\circ \stackrel{(k)}{\cdots}\circ M$.
\item $M_{L(\log\log L)^{\delta}}$ with $\delta>0$
given by the Young function $\Phi(t)=t(\log\log(e^{e}+t))^{\delta}.$
\item $M_{L(\log L)(\log\log L)^{\delta}}$ with $\delta>0$
given by the function $\Phi(t)=t\log(e+t)(\log\log(e^{e}+t))^{\delta}.$
\end{itemize}
One basic fact about this kind of maximal operators that follows from
the definition of the norm is the following. Given $\Psi$ and $\Phi$
Young functions such that for some $\kappa,c>0$ $\Psi(t)\leq\kappa\Phi(t)$,
then
\[
\|f\|_{\Psi(L),Q}\leq(\Psi(c)+\kappa)\|f\|_{\Phi(L),Q},
\]
and consequently
\[
M_{\Psi(L)}f(x)\leq(\Psi(c)+\kappa)M_{\Phi(L)}f(x).
\]

Associated to each Young function $A$ there exists a complementary function $\bar{A}$ that can be  defined as follows
\[\bar{A}(t)=\sup_{s>0}\{st-A(s)\}.\]
That complementary function is  a Young function as well and it satisfies the following pointwise estimate
\[t\leq A^{-1}(t)\bar{A}^{-1}(t)\leq 2t.\]
An interesting property of this associated function is that the following estimate holds
\begin{equation*}
\frac{1}{|Q|}\int_Q|fg|dx\leq 2\|f\|_{A,Q}\|g\|_{\bar{A},Q}.
\end{equation*}
A case of interest for us is the case $A(t)=t\log(e+t)$. In that case we have that
\begin{equation*}
\frac{1}{|Q|}\int_Q|fg|dx\leq c\|f\|_{L\log L,Q}\|g\|_{\exp(L),Q}.
\end{equation*}
From that estimate taking into account John-Nirenberg's theorem, if $b\in \BMO$, then
\begin{equation}\label{eq:HGen}
\frac{1}{|Q|}\int_Q|f(b-b_Q)|dx\leq c\|f\|_{L\log L,Q}\|b-b_Q\|_{\exp(L),Q}\leq c\|f\|_{L\log L,Q}\|b\|_{\BMO}.
\end{equation}

For a detailed account about the ideas presented in the end of this
Section we refer the reader to \cite{PKJ,RR}.

\subsection{\label{subsec:Lerner's-Formula}Lerner-Nazarov formula}

In this Section we recall the definitions of the local oscillation and the Lerner-Nazarov oscillation and we show that the latter is controlled
by the former. Built upon Lerner-Nazarov oscillation we will also
introduce formula, which will be a quite useful tool for us. Most
of the ideas covered in this Section are borrowed from \cite{LN}.
Among them, we start with the definition of dyadic lattice.

Let us call $\mathcal{D}(Q)$ the dyadic grid obtained repeatedly
subdividing $Q$ and its descendents in $2^{n}$ cubes with the same side length.
\begin{defn}
A dyadic lattice $\mathcal{D}$ in $\mathbb{R}^{n}$ is a family of
cubes that satisfies the following proporties

\begin{enumerate}
\item If $Q\in\mathcal{D}$ then each descendant of $Q$ is in $\mathcal{D}$
as well.
\item For every 2 cubes $Q_{1},Q_{2}$ we can find a common ancestor, that
is, a cube $Q\in\mathcal{D}$ such that $Q_{1},Q_{2}\in\mathcal{D}(Q)$.
\item For every compact set $K$ there exists a cube $Q\in\mathcal{D}$
such that $K\subseteq Q$.
\end{enumerate}
\end{defn}
A way to build such a structure is to consider an increasing sequence of cubes
$\{Q_{j}\}$ expanding each time from a different vertex. That choice
of cubes gives that $\mathbb{R}^{n}=\cup_{j}Q_{j}$ and it is not hard
to check that
\[
\mathcal{D}=\bigcup_{j}\{Q\in\mathcal{D}(Q_{j})\}
\]
is a dyadic lattice.

\begin{lem}
Given a dyadic lattice $\mathcal{D}$ there exist $3^{n}$ dyadic
lattices $\mathcal{D}_{j}$ such that
\[
\{3Q\,:\,Q\in\mathcal{D}\}=\bigcup_{j=1}^{3^{n}}\mathcal{D}_{j}
\]
and for every cube $Q\in\mathcal{D}$ we can find a cube $R_{Q}$
in each $\mathcal{D}_{j}$ such that $Q\subseteq R_{Q}$ and $3l_{Q}=l_{R_{Q}}$
\end{lem}
\begin{rem}
\label{Rem}Fix $\mathcal{D}$. For an arbitrary cube $Q\subseteq\mathbb{R}^{n}$
there is a cube $Q'\in\mathcal{D}$ such that $\frac{l_{Q}}{2}<l_{Q'}\leq l_{Q}$
and $Q\subseteq3Q'$ . It suffices to take the cube $Q'$ that contains
the center of $Q$ . From the lemma above it follows that $3Q'=P\in\mathcal{D}_{j}$
for some $j\in\{1,\dots,3^{n}\}$. Therefore, for every cube $Q\subseteq\mathbb{R}^{n}$
there exists $P\in\mathcal{D}_{j}$ such that $Q\subseteq P$ and
$l_{P}\leq3l_{Q}$. From this follows that $|Q|\leq|P|\leq3^{n}|Q|$

\end{rem}
\begin{defn}
$\mathcal{S}\subseteq\mathcal{D}$ is a $\eta$-sparse family with
$\eta\in(0,1)$ if for each $Q\in\mathcal{S}$ we can find a measurable
subset $E_{Q}\subseteq Q$ such that
\[
\eta|Q|\leq|E_{Q}|
\]
and all the $E_{Q}$ are pairwise disjoint.
\end{defn}
We also recall here the definition of Carleson family.
\begin{defn}
We say that a family $\mathcal{S}\subseteq\mathcal{D}$ is $\varLambda$-Carleson
with $\varLambda>1$ if for each $Q\in\mathcal{S}$ we have that
\[
\sum_{P\in\mathcal{S},\,P\subseteq Q}|P|\leq\varLambda|Q|.
\]
The following result that establishes the relationship between Carleson
and sparse families was obtained in \cite{LN} and reads as follows.
\end{defn}
\begin{lem}
\label{Lem:CarlesonSparse}If $\mathcal{S}\subseteq\mathcal{D}$ is
a $\eta$-sparse family then it is a $\frac{1}{\eta}$-Carleson family.
Conversely if $\mathcal{S}$is $\varLambda$-Carleson then it is $\frac{1}{\varLambda}$-sparse.
\end{lem}
Now we turn to recall the definition of the local oscillation \cite{L1}
which is given in terms of decreasing rearrangements.
\begin{defn}
[Local oscillation]Given $\lambda\in(0,1)$, a measurable function
$f$ and a cube $Q$. We define
\[
\tilde{w}_{\lambda}(f;Q):=\inf_{c\in\mathbb{R}}\left((f-c)\chi_{Q}\right)^{*}(\lambda|Q|).
\]
\end{defn}
For any function $g$, its decreasing rearrangement $g^{*}$ is given
by
\[
g^{*}(t)=\inf\left\{ \alpha>0\,:\,\left|\left\{ x\in\mathbb{R}^{n}\,:\,|g|>\alpha\right\} \right|\leq t\right\} .
\]
In particular,
\[
\left((f-c)\chi_{Q}\right)^{*}(\lambda|Q|)=\inf\left\{ \alpha>0\,:\,\left|\left\{ x\in Q\,:\,|f-c|>\alpha\right\} \right|\leq\lambda|Q|\right\} .
\]

Now we define Lerner-Nazarov oscillation \cite{LN}. We would like
to observe that decreasing rearrangements are not involved in the
definition.
\begin{defn}
[Lerner-Nazarov oscillation]Given $\lambda\in(0,1)$, a measurable
function $f$ and a cube $Q$. We define the $\lambda$-oscillation
of $f$ on $Q$ as
\[
w_{\lambda}(f;Q):=\inf\left\{ w(f;E)\,:\,E\subseteq Q,\,|E|\geq(1-\lambda)|Q|\right\},
\]
where
\[
w(f;E)=\sup_{E}f-\inf_{E}f.
\]
\end{defn}
Now as we announced we are going to prove that the local oscillation
controls Lerner-Nazarov oscillation.
\begin{lem}
\label{Lem:Oscil}Given a measurable function $f$ we have that for
every $\lambda\in(0,1)$,
\[
w(f;Q)\leq2\tilde{w}_{\lambda}(f;Q).
\]
\end{lem}
\begin{proof}
We start the proof of this lemma recalling a useful
identity that appears stated in \cite{H}:
\[
f^{*}(t)=\inf_{|E|\leq t}\|f\chi_{E^{c}}\|_{L^{\infty}},
\]
where $E$ is any measurable set contained in $\mathbb{R}^n$. Taking that identity into account it is clear that
\[
\tilde{w}_{\lambda}(f;Q)=\inf_{c\in\mathbb{R}}\inf_{E\subseteq Q,\,|E|\leq\lambda|Q|}\left\Vert \left(f-c\right)\chi_{Q\setminus E}\right\Vert _{L^{\infty}},
\]
since it allows us to write
\[
\inf_{E\subseteq Q,\,|E|\leq\lambda|Q|}\left\Vert \left(f-c\right)\chi_{Q\setminus E}\right\Vert _{L^{\infty}}=\inf\left\{ \alpha>0\,:\,\left|\left\{ x\in Q\,:\,|f-c|>\alpha\right\} \right|\leq\lambda|Q|\right\}.
\]
Now we observe that
\[
w_{\lambda}(f;Q)=\inf\left\{ w(f;Q\setminus E)\,:\,E\subseteq Q,\,\lambda|Q|\geq|E|\right\} .
\]
Let $c>0$. We see that
\[
\begin{split}w(f;Q\setminus E) & =\sup_{Q\setminus E}f-\inf_{Q\setminus E}f=\sup_{Q\setminus E}f-c+c-\inf_{Q\setminus E}f\\
 & =\sup_{Q\setminus E}(f-c)+\sup_{Q\setminus E}(-f+c)\leq2\|(f-c)\chi_{Q\setminus E}\|_{L^{\infty}}.
\end{split}
\]
And taking infimum on both sides of the inequality
\[
w_{\lambda}(f;Q)\leq2\tilde{w}_{\lambda}(f;Q).
\]
To end this Section we introduce Lerner-Nazarov formula (cf. \cite{LN}).
\end{proof}
\begin{thm}
[Lerner-Nazarov formula]\label{LernerFormula-1}Let $f:\mathbb{R}^{n}\rightarrow\mathbb{R}$
be a measurable function such that for each $\varepsilon>0$
\[
\left|\left\{ x\in[-R,R]^{n}\,:\,|f(x)|>\varepsilon\right\} \right|=o(R^{n})\,\, \text{as}\,\,R\rightarrow \infty.
\]
Then for each dyadic lattice $\mathcal{D}$ and every $\lambda\in(0,2^{-n-2}]$
we can find a regular $\frac{1}{6}$-sparse family of cubes $\mathcal{S\subseteq D}$
(depending on $f$) such that
\[
|f(x)|\leq\sum_{Q\in\mathcal{S}}w_{\lambda}(f;Q)\chi_{Q}(x)\qquad\text{a.e.}
\]
\end{thm}

\subsection{\label{subsec:ApAinfty}Mixed $A_{p}-A_{\infty}$ estimates for generalized
dyadic square functions}

The content of this Section is essentially borrowed from \cite{HL}.
We begin with the definition of a generalization of dyadic square
functions.
\begin{defn}
Let $r>0$ and $\eta\in(0,1)$. Let $\mathcal{D}$ be a dyadic lattice
and $\mathcal{S}\subseteq\mathcal{D}$ be an $\eta$-sparse family. We
define the operator $\mathcal{A}_{\mathcal{S}}^{r}$ as
\[
\mathcal{A}_{\mathcal{S}}^{r}f(x)=\left(\sum_{Q\in\mathcal{S}}\left(\frac{1}{|Q|}\int_{Q}|f(y)|dy\right)^{r}\chi_{Q}(x)\right)^{\frac{1}{r}}.
\]
\end{defn}
We will also follow the notation of \cite{HL} for the $A_{p}$ and
related weights. Given a pair of weights $w$ and $\sigma$ we denote
\[
[w,\sigma]_{A_{p}}=\sup_{Q}\left(\frac{1}{|Q|}\int_{Q}w\right)\left(\frac{1}{|Q|}\int_{Q}\sigma\right)^{p-1}
\]
and
\[
[w]_{A_{\infty}}=\sup_{Q}\frac{1}{w(Q)}\int_{Q}M(w\chi_{Q}).
\]
We recall that $A_{\infty}=\bigcup_{p\geq1}A_{p}$ and that $w\in A_{\infty}$
if and only if $[w]_{A_{\infty}}<\infty$. We also note that if $w\in A_{p}$
and $\sigma=w^{1-p'}$ then $[w,\sigma]_{A_{p}}=[w]_{A_{p}}$.

The definition of the $A_{\infty}$ constant $[w]_{A_{\infty}}$
was shown to be very suitable in \cite{HPAinfty} being one of the main results of that work the following optimal reverse
H\"older inequality (see also \cite{HPR} for a different proof).
\begin{thm}[Reverse H\"older inequality]
 \label{Thm:RHI}Let $w\in A_{\infty}$. Then there exists a dimensional
constant $\tau_{n}$ such that
\[
\left(\frac{1}{|Q|}\int_{Q}w^{r_{w}}\right)^{\frac{1}{r_{w}}}\leq\frac{2}{|Q|}\int_{Q}w,
\]
where
\[
r_{w}=1+\frac{1}{\tau_{n}[w]_{A_{\infty}}}.
\]

\end{thm}
To end this Section we borrow the main results from \cite{HL}. They
will be crucial for our purposes. In both statements
\[
(\alpha)_{+}=\begin{cases}
\alpha & \text{if }\alpha>0,\\
0 & \text{otherwise.}
\end{cases}
\]

\begin{thm}
\label{ThmFuerteAp}Let $1<p<\infty$ and $r>0$. Let $w,\sigma$ be
a pair of weights. Then
\[
\|\mathcal{A}_{\mathcal{S}}^{r}f\|_{L^{p}(w)}\lesssim[w,\sigma]_{A_{p}}^{\frac{1}{p}}\left([w]_{A_{\infty}}^{\left(\frac{1}{r}-\frac{1}{p}\right)_{+}}+[\sigma]_{A_{\infty}}^{\frac{1}{p}}\right)\|f\|_{L^{p}(w)}.
\]
\end{thm}

\begin{thm}
\label{ThmDebilAp}Let $r>0$ and $1<p<\infty$ such that $p\not=r$.
Let $w,\sigma$ be a pair of weights. Then
\[
\|\mathcal{A}_{\mathcal{S}}^{r}f\|_{L^{p,\infty}(w)}\lesssim[w,\sigma]_{A_{p}}^{\frac{1}{p}}[w]_{A_{\infty}}^{\left(\frac{1}{r}-\frac{1}{p}\right)_{+}}\|f\|_{L^{p}(w)}.
\]
\end{thm}

\section{Proofs of the results for the vector-valued Hardy-Littlewood maximal
operators.\label{sec:ProofsMax}}

\subsection{Sparse domination for vector-valued Hardy-Littlewood maximal operators}
\label{sec:ProofsMaxSparse}
We are going to prove
\begin{equation}
\overline{M}_{q}\bm{f}(x)\leq c_{n,q}\sum_{k=1}^{3^{n}}\left(\sum_{Q\in\mathcal{S}_{k}}\left(\frac{1}{|Q|}\int_{Q}|\bm{f}|_{q}\right)^{q}\chi_{Q}(x)\right)^{\frac{1}{q}}.\label{eq:SparseMax}
\end{equation}
First we observe that from Remark \ref{Rem} it readily follows that
\[
Mf(x)\leq c_{n}\sum_{k=1}^{3^{n}}M^{\mathcal{D}_{k}}f(x).
\]
Taking that into account it is clear that
\begin{equation}
\overline{M}_{q}\bm{f}(x)\leq c_{n}\sum_{k=1}^{3^{n}}\overline{M}_{q}^{\mathcal{D}_{k}}\bm{f}(x).\label{eq:MqMDk}
\end{equation}
The following estimate for local oscillations
\[
\tilde{w}_{\lambda}\left(\left(\overline{M}_{q}^{\mathcal{D}}\bm{f}\right)^{q};Q\right)\leq\frac{c_{n,q}}{\lambda^{q}}\left(\frac{1}{|Q|}\int_{Q}|\bm{f}|_{q}\right)^{q},
\]
was established in \cite[Lemma 8.1]{CUMP3}. Now we recall that by Lemma
\ref{Lem:Oscil}
\[
w_{\lambda}\left(\left(\overline{M}_{q}^{\mathcal{D}}\bm{f}\right){}^{q};Q\right)\leq2\tilde{w}_{\lambda}\left(\left(\overline{M}_{q}^{\mathcal{D}}\bm{f}\right)^{q};Q\right).
\]
Then
\[
w_{\lambda}\left(\left(\overline{M}_{q}^{\mathcal{D}}\bm{f}\right)^{q};Q\right)\leq\frac{c_{n,q}}{\lambda^{q}}\left(\frac{1}{|Q|}\int_{Q}|\bm{f}|_{q}\right)^{q}.
\]
Using now Lerner-Nazarov formula (Theorem \ref{LernerFormula-1}) there
exists a $\frac{1}{6}$-sparse family $\mathcal{S}\subset\mathcal{D}$
such that
\[
\begin{split}\overline{M}_{q}^{\mathcal{D}}\bm{f}(x)^{q} & \leq\sum_{Q\in\mathcal{S}}w_{\lambda}\left(\left(\overline{M}_{q}^{\mathcal{D}}\bm{f}\right)^{q};Q\right)\chi_{Q}(x)\\
 & \leq\frac{2c_{n,q}}{\lambda^{q}}\sum_{Q\in\mathcal{S}}\left(\frac{1}{|Q|}\int_{Q}|\bm{f}|_{q}\right)^{q}\chi_{Q}(x).
\end{split}
\]
Consequently
\[
\overline{M}_{q}^{\mathcal{D}}\bm{f}(x)\leq c_{n,q}\left(\sum_{Q\in\mathcal{S}}\left(\frac{1}{|Q|}\int_{Q}|\bm{f}|_{q}\right)^{q}\chi_{Q}(x)\right)^{\frac{1}{q}}.
\]
Applying this to each $\overline{M}_{q}^{\mathcal{D}_{k}}\bm{f}(x)$ in
(\ref{eq:MqMDk}) we obtain the desired estimate.

\subsection{Weighted estimates for vector-valued Hardy-Littlewood maximal operators}

The weighted estimates in Theorem \ref{ThmEstMq} are a direct consequence
of Theorems \ref{ThmFuerteAp} and \ref{ThmDebilAp} (cf. Section
\ref{subsec:ApAinfty}) applied to $\mathcal{A}_{\mathcal{S}_{k}}^{q}(|\bm{f}|_{q}).$

\section{Sparse domination for vector-valued Calder\'on-Zygmund operators and
vector-valued commutators\label{sec:Sparse}}

\subsection{\label{subsec:LemmaSparseCZO}A technical lemma}

Let $T$ be an $\omega$-CZO with $\omega$ satisfying Dini condition
and $1<q<\infty$. We define the grand maximal truncated operator
$\mathcal{M}_{T_{q}}$ by
\[
\mathcal{M}_{\overline{T}_{q}}\bm{f}(x)=\sup_{Q\ni x}\underset{{\scriptscriptstyle \xi\in Q}}{\esssup}\left|\overline{T}_{q}(\bm{f}\chi_{\mathbb{R}^{n}\setminus3Q})(\xi)\right|.
\]
We also consider a local version of this operator
\[
\mathcal{M}_{\overline{T}_{q},Q_{0}}\bm{f}(x)=\sup_{x\in Q\subseteq Q_{0}}\underset{{\scriptscriptstyle \xi\in Q}}{\esssup}\left|\overline{T}_{q}(\bm{f}\chi_{3Q_{0}\setminus3Q})(\xi)\right|.
\]

\begin{lem}
\label{LemmaTec}Let $T$ be an $\omega$-CZO with $\omega$ satisfying
Dini condition and $1<q<\infty$. The following pointwise estimates
hold:

\begin{enumerate}
\item For a.e. $x\in Q_{0}$
\begin{equation}
\label{eq:LemaTec1}
|\overline{T}_{q}(\bm{f}\chi_{3Q_{0}})(x)|\leq c_{n}\|\overline{T}_{q}\|_{L^{1}\rightarrow L^{1,\infty}}|\bm{f}|_{q}(x)+\mathcal{M}_{\overline{T}_{q},Q_{0}}f(x).
\end{equation}
\item For all $x\in\mathbb{R}^{n}$
\begin{equation}\label{eq:LemaTec2}
\mathcal{M}_{\overline{T}_{q}}\bm{f}(x)\leq c_{n,q}(\|\omega\|_{\text{Dini}}+C_{K})M_{q}f(x)+\overline{T^{*}}_{q}\bm{f}(x).
\end{equation}
Furthermore
\begin{equation}\label{eq:LemaTec3}
\left\Vert \mathcal{M}_{\overline{T}_{q}}\right\Vert _{L^{1}\rightarrow L^{1,\infty}}\leq c_{n,q}C_{T},
\end{equation}
where $C_{T}=C_{K}+\|\omega\|_{\text{Dini}}+\|T\|_{L^{2}\rightarrow L^{2}}$.
\end{enumerate}
\end{lem}
\begin{proof}
First we prove the estimate in \eqref{eq:LemaTec1}. Fix $x\in\inte Q_{0}$, and let
$x$ be a point of approximate continuity of $\overline{T}_{q}(\bm{f}\chi_{3Q_{0}})$
(see \cite[p. 46]{EG}). For every $\varepsilon>0$, set
\[
E_{s}(x)=\left\{ y\in B(x,s)\ :\ |\overline{T}_{q}(\bm{f}\chi_{3Q_{0}})(y)-\overline{T}_{q}(\bm{f}\chi_{3Q_{0}})(x)|<\varepsilon\right\}.
\]
Then we have that $\lim_{s\rightarrow0}\frac{|E_{s}(x)|}{|B(x,s)|}=1,$
where $B(x,s)$ is the open ball centered at $x$ of radius $s$.

Denote by $Q(x,s)$ the smallest cube centered at $x$ and containing
$B(x,s)$. Let $s>0$ be so small that $Q(x,s)\subset Q_{0}$. Then
for a.e. $y\in E_{s}(x)$,
\[
|\overline{T}_{q}(\bm{f}\chi_{3Q_{0}})(x)|\leq|\overline{T}_{q}(\bm{f}\chi_{3Q_{0}})(y)|+\varepsilon\leq|\overline{T}_{q}(\bm{f}\chi_{3Q(x,s)})(y)|+\mathcal{M}_{\overline{T}_{q},Q_{0}}\bm{f}(x)+\varepsilon.
\]
Now we can apply the weak type $(1,1)$ estimate of $\overline{T}_{q}$.
Then
\[
\begin{split} & |\overline{T}_{q}(\bm{f}\chi_{3Q_{0}})(x)|\\
 & \leq\underset{y\in E_{s}(x)}{\essinf}|\overline{T}_{q}(\bm{f}\chi_{3Q(x,s)})(y)|+\mathcal{M}_{\overline{T}_{q},Q_{0}}\bm{f}(x)+\varepsilon\\
 & \leq\|\overline{T}_{q}\|_{L^{1}\rightarrow L^{1,\infty}}\frac{1}{|E_{s}(x)|}\int_{3Q(x,s)}|\bm{f}|_{q}+\mathcal{M}_{\overline{T}_{q},Q_{0}}\bm{f}(x)+\varepsilon.
\end{split}
\]
Assuming that $x$ is a Lebesgue point of $|\bm{f}|_{q}$ and letting $s\rightarrow0$
and $\varepsilon\rightarrow0$, leads to obtain  \eqref{eq:LemaTec1}.

Now we focus on part  \eqref{eq:LemaTec2}. Let $x,\xi\in Q$. Denote by $B_{x}$ the
closed ball centered at $x$ of radius $2\text{diam}Q$. Then $3Q\subset B_{x}$,
and we obtain
\[
\begin{split}|\overline{T}_{q}(\bm{f}\chi_{\mathbb{R}^{n}\setminus3Q})(\xi)| & \leq|\overline{T}_{q}(\bm{f}\chi_{\mathbb{R}^{n}\setminus B_{x}})(\xi)+\overline{T}_{q}(\bm{f}\chi_{B_{x}\setminus3Q})(\xi)|\\
 & \leq|\overline{T}_{q}(\bm{f}\chi_{\mathbb{R}^{n}\setminus B_{x}})(\xi)-\overline{T}_{q}(\bm{f}\chi_{\mathbb{R}^{n}\setminus B_{x}})(x)|\\
 & +|\overline{T}_{q}(\bm{f}\chi_{B_{x}\setminus3Q})(\xi)|+|\overline{T}_{q}(\bm{f}\chi_{\mathbb{R}^{n}\setminus B_{x}})(x)|.
\end{split}
\]

By the smoothness condition, since $| |a|^{r}-|b|^{r} | \leq 2^{ \max\{1,r\}-1}  |a-b|^{r}$
for every $r>0$ we have that
\[
\begin{split} & |\overline{T}_{q}(\bm{f}\chi_{\mathbb{R}^{n}\setminus B_{x}})(\xi)-\overline{T}_{q}(\bm{f}\chi_{\mathbb{R}^{n}\setminus B_{x}})(x)|\\
 & =\left|\left(\sum_{j=1}^{\infty}|T(f_{j}\chi_{\mathbb{R}^{n}\setminus B_{x}})(\xi)|^{q}\right)^{\frac{1}{q}}-\left(\sum_{j=1}^{\infty}|T(f_{j}\chi_{\mathbb{R}^{n}\setminus B_{x}})(x)|^{q}\right)^{\frac{1}{q}}\right|\\
 & \leq \left(\sum_{j=1}^{\infty}\left|T(f_{j}\chi_{\mathbb{R}^{n}\setminus B_{x}})(\xi)-T(f_{j}\chi_{\mathbb{R}^{n}\setminus B_{x}})(x)\right|^{q}\right)^{\frac{1}{q}}.
\end{split}
\]
Now using the smoothness condition (see \cite[Proof of Lemma 3.2 (ii)]{L} \[\left|T(f_{j}\chi_{\mathbb{R}^{n}\setminus B_{x}})(\xi)-T(f_{j}\chi_{\mathbb{R}^{n}\setminus B_{x}})(x)\right|\leq c_{n}\|\omega\|_{\text{Dini}}Mf_{j}(x),\]
and then we have that
\[
|\overline{T}_{q}(\bm{f}\chi_{\mathbb{R}^{n}\setminus B_{x}})(\xi)-\overline{T}_{q}(\bm{f}\chi_{\mathbb{R}^{n}\setminus B_{x}})(x)|\leq c_{n,q}\|\omega\|_{\text{Dini}}\left(\sum_{j=1}^{\infty}\left|Mf_{j}(x)\right|^{q}\right)^{\frac{1}{q}}=c_{n,q}\|\omega\|_{\text{Dini}}\overline{M}_{q}\bm{f}(x).
\]
On the other hand the size condition of the kernel yields
\[
\begin{split}|\overline{T}_{q}(\bm{f}\chi_{B_{x}\setminus3Q})(\xi)| & \leq\left|\left(\sum_{j=1}^{\infty}|T(f_{j}\chi_{B_{x}\setminus3Q})(\xi)|^{q}\right)^{\frac{1}{q}}\right|\\
 & \leq c_{n}C_{K}\left|\left(\sum_{j=1}^{\infty}\left(\frac{1}{|B_{x}|}\int_{B_{x}}|f_{j}|\right)^{q}\right)^{\frac{1}{q}}\right|\\
 & \leq c_{n}C_{K}\left|\left(\sum_{j=1}^{\infty}\left(Mf_{j}(x)\right)^{q}\right)^{\frac{1}{q}}\right|\leq c_{n}C_{K}\overline{M}_{q}\bm{f}(x).
\end{split}
\]
To end the proof of the pointwise estimate we observe that
\[
|\overline{T}_{q}(\bm{f}\chi_{\mathbb{R}^{n}\setminus B_{x}})(x)|\leq\overline{T}_{q}^{*}\bm{f}(x).
\]
Now, taking into account the pointwise estimate we have just obtained
and Theorem \ref{Thm:MaxTq} below it is clear that
\[
\left\Vert \mathcal{M}_{\overline{T}_{q}}\right\Vert _{L^{1}\rightarrow L^{1,\infty}}\leq c_{n,q}C_{T}.
\]
This ends the proof.
\end{proof}

\subsection{Proof of Theorem \ref{ThmSparseTq}}

We fix a cube $Q_{0}\subset\mathbb{R}^{n}$ such that $\supp|\bm{f}|_{q}\subseteq Q_{0}$.

We start claiming that here exists a $\frac{1}{2}$-sparse family $\mathcal{F}\subseteq\mathcal{D}(Q_{0})$
such that for a.e. $x\in Q_{0}$
\begin{equation}
\left|\overline{T}_{q}(\bm{f}\chi_{3Q_{0}})(x)\right|\leq c_{n}C_{T}\sum_{Q\in\mathcal{F}}\langle |\bm{f}|_{q}\rangle _{3Q}\chi_{Q}(x).\label{eq:Claim}
\end{equation}

Suppose that we have already proved the claim \eqref{eq:Claim}. Let us take a partition
of $\mathbb{R}^{n}$ by cubes $Q_{j}$ such that $\supp(|\bm{f}|_{q})\subseteq3Q_{j}$
for each $j$. We can do it as follows . We start with a cube $Q_{0}$
such that $\supp(|\bm{f}|_{q})\subset Q_{0}.$ And cover $3Q_{0}\setminus Q_{0}$
by $3^{n}-1$ congruent cubes $Q_{j}$. Each of them satisfies $Q_{0}\subset3Q_{j}$.
We do the same for $9Q_{0}\setminus3Q_{0}$ and so on. Now we apply the claim to each cube $Q_{j}$. Then we have that since
$\supp|\bm{f}|_{q}\subseteq3Q_{j}$ the following estimate holds a.e. $x\in Q_{j}$
\[
\left|\overline{T}_{q}\bm{f}(x)\right|\chi_{Q_{j}}(x)=\left|\overline{T}_{q}(\bm{f}\chi_{3Q_{j}})(x)\right|\leq c_{n}C_{T}\sum_{Q\in\mathcal{F}_{j}}\langle|\bm{f}|_{q}\rangle_{3Q}\chi_{Q}(x),
\]
where each $\mathcal{F}_{j}\subseteq\mathcal{D}(Q_{j})$ is a $\frac{1}{2}$-sparse
family. Taking $\mathcal{F}=\bigcup\mathcal{F}_{j}$ we have that
$\mathcal{F}$ is a $\frac{1}{2}$-sparse family and
\[
\left|\overline{T}_{q}\bm{f}(x)\right|\leq c_{n}C_{T}\sum_{Q\in\mathcal{F}}\langle |\bm{f}|_{q}\rangle_{3Q}\chi_{Q}(x).
\]
From Remark \ref{Rem} it follows that there exist $3^n$ dyadic lattices $\mathcal{D}_j$ such that for each cube $3Q$ there exists some cube $R_Q\in\mathcal{D}_j$ such that $3Q\subset R_{Q}$ and $|R_{Q}|\leq3^{n}|3Q|$. Consequently $|h|_{3Q}\leq c_{n}|h|_{R_{Q}}$. Now setting
\[
\mathcal{S}_{j}=\left\{ R_{Q}\in\mathcal{D}_{j}\,:\,Q\in\mathcal{F}\right\}
\]
and using that $\mathcal{F}$ is $\frac{1}{2}$-sparse, we obtain
that each family $\mathcal{S}_{j}$ is sparse as well. Then we can conclude that
\[
\left|\overline{T}_{q}\bm{f}(x)\right|\leq c_{n}C_{T}\sum_{j=1}^{3^{n}}\sum_{R\in\mathcal{S}_{j}}\langle |\bm{f}|_{q}\rangle_{R}\chi_{R}(x).
\]

\subsubsection*{Proof of the claim (\ref{eq:Claim})}

To prove the claim it suffices to prove the following recursive estimate:
There exist pairwise disjoint cubes $P_{j}\in\mathcal{D}(Q_{0})$
such that $\sum_{j}|P_{j}|\leq\frac{1}{2}|Q_{0}|$ and for a.e. $x\in Q_0$,
\[
\begin{split}  \left|\overline{T}_{q}(\bm{f}\chi_{3Q_{0}})(x)\right|\chi_{Q_{0}}(x)
\leq c_{n}C_{T}\langle |\bm{f}|_{q}\rangle_{3Q}+\sum_{j}\left|\overline{T}_{q}(\bm{f}\chi_{3P_{j}})(x)\right|\chi_{P_{j}}(x).
\end{split}
\]
Iterating this estimate we obtain (\ref{eq:Claim})
with $\mathcal{F}=\{P_{j}^{k}\}$ where $\{P_{j}^{0}\}=\{Q_{0}\}$,
$\{P_{j}^{1}\}=\{P_{j}\}$ and $\{P_{j}^{k}\}$ are the cubes obtained
at the $k$-th stage of the iterative process. It is also clear that
$\mathcal{F}$ is a $\frac{1}{2}$-sparse family.

Now we observe that for any arbitrary family of disjoint cubes $P_{j}\in\mathcal{D}(Q_{0})$
we have that
\[
\begin{split} & \left|\overline{T}_{q}(\bm{f}\chi_{3Q_{0}})(x)\right|\chi_{Q_{0}}(x)\\
= & \left|\overline{T}_{q}(\bm{f}\chi_{3Q_{0}})(x)\right|\chi_{Q_{0}\setminus\bigcup_{j}P_{j}}(x)+\sum_{j}\left|\overline{T}_{q}(\bm{f}\chi_{3Q_{0}})(x)\right|\chi_{P_{j}}(x)\\
\leq & \left|\overline{T}_{q}(\bm{f}\chi_{3Q_{0}})(x)\right|\chi_{Q_{0}\setminus\bigcup_{j}P_{j}}(x)+\sum_{j}\left|\overline{T}_{q}(\bm{f}\chi_{3Q_{0}\setminus3P_{j}})(x)\right|\chi_{P_{j}}(x)+\sum_{j}\left|\overline{T}_{q}(\bm{f}\chi_{3P_{j}})(x)\right|\chi_{P_{j}}(x).
\end{split}
\]
So it suffices to show that we can choose a family of pairwise disjoint
cubes $P_{j}\in\mathcal{D}(Q_{0})$ with $\sum_{j}|P_{j}|\leq\frac{1}{2}|Q_{0}|$
and such that for a.e. $x\in Q_{0}$
\begin{equation}
\begin{split} & \left|\overline{T}_{q}(\bm{f}\chi_{3Q_{0}})(x)\right|\chi_{Q_{0}\setminus\bigcup_{j}P_{j}}(x)+\sum_{j}\left|\overline{T}_{q}(\bm{f}\chi_{3Q_{0}\setminus3P_{j}})(x)\right|\chi_{P_{j}}(x)\\
 & \leq c_{n}C_{T}\langle |\bm{f}|_{q}\rangle_{3Q_{0}}.
\end{split}
\label{eq:3.2}
\end{equation}
Now we define the set $E$ where
\[
E=\left\{ x\in Q_{0}\,:\,|\bm{f}|_{q}>\alpha_{n}\langle |\bm{f}|_{q}\rangle_{3Q_{0}}\right\} \cup\left\{ x\in Q_{0}\,:\,\mathcal{M}_{\overline{T}_{q},Q_{0}}\bm{f}>\alpha_{n}C_{T}\langle |\bm{f}|_{q}\rangle_{3Q_{0}}\right\}.
\]
Since $\mathcal{M}_{\overline{T}_{q}}$ is of weak type $(1,1)$
with
\[
\|\mathcal{M}_{\overline{T}_{q}}\|_{L^{1}\rightarrow L^{1,\infty}}\leq c_{n}C_{T}.
\]
we have that, from this point the rest of the proof is analogous to \cite[Theorem 3.1]{L}.

\subsection{Proof of Theorem \ref{ThmSparseConmm}}

From Remark \ref{Rem} it follows that there exist $3^{n}$ dyadic
lattices such that for every cube $Q$ of $\mathbb{R}^{n}$ there
is a cube $R_{Q}\in\mathcal{D}_{j}$ for some $j$ for which $3Q\subset R_{Q}$
and $|R_{Q}|\leq9^{n}|Q|$

Let us fix a cube $Q_{0}\subset\mathbb{R}^{n}$. We claim that there exists a $\frac{1}{2}$-sparse family $\mathcal{F}\subseteq\mathcal{D}(Q_{0})$
such that for a.e. $x\in Q_{0}$
\begin{equation}
\left|\overline{[b,T]}_{q}(\bm{f}\chi_{3Q_{0}})(x)\right|\leq c_{n}C_{T}\sum_{Q\in\mathcal{F}}\left(|b(x)-b_{R_{Q}}|\langle |\bm{f}|_{q}\rangle_{3Q}+\langle|(b-b_{R_{Q}})||\bm{f}|_{q}\rangle_{3Q}\right)\chi_{Q}(x).\label{eq:Claim-1}
\end{equation}

Proceeding analogously as we did in the proof of Theorem \ref{ThmSparseTq},
 from the claim \eqref{eq:Claim-1} it follows that there exists a $\frac{1}{2}$-sparse family $\mathcal{F}$ such that for every $x\in \mathbb{R}^n$,
\[
\left|\overline{[b,T]}_{q}\bm{f}(x)\right|\leq c_{n}C_{T}\sum_{Q\in\mathcal{F}}\left(|b(x)-b_{R_{Q}}|\langle |\bm{f}|_{q}\rangle_{3Q}+\langle |(b-b_{R_{Q}})||\bm{f}|_{q}\rangle_{3Q}\right)\chi_{Q}(x).
\]
Now we observe that since $3Q\subset R_{Q}$ and $|R_{Q}|\leq3^{n}|3Q|$
we have that $|h|_{3Q}\leq c_{n}|h|_{R_{Q}}$. Setting
\[
\mathcal{S}_{j}=\left\{ R_{Q}\in\mathcal{D}_{j}\,:\,Q\in\mathcal{F}\right\}
\]
and using that $\mathcal{F}$ is $\frac{1}{2}$-sparse, we obtain
that each family $\mathcal{S}_{j}$ is $\frac{1}{2\cdot9^{n}}$-sparse.
Then we have that
\[
\left|\overline{[b,T]}_{q}\bm{f}(x)\right|\leq c_{n}C_{T}\sum_{j=1}^{3^{n}}\sum_{R\in\mathcal{S}_{j}}\left(|b(x)-b_{R}|\langle |\bm{f}|_{q}\rangle_{R}+\langle |(b-b_{R})||\bm{f}|_{q}\rangle_{R}\right)\chi_{R}(x).
\]

\subsubsection*{Proof of the claim (\ref{eq:Claim-1})}

To prove the claim it suffices to prove the following recursive estimate:
There exist pairwise disjoint cubes $P_{j}\in\mathcal{D}(Q_{0})$
such that $\sum_{j}|P_{j}|\leq\frac{1}{2}|Q_{0}|$ and for a.e. $x\in Q_0$,
\[
\begin{split} & \left|\overline{[b,T]}_{q}(\bm{f}\chi_{3Q_{0}})(x)\right|\chi_{Q_{0}}(x)\\
 & \leq c_{n}C_{T}\left(|b(x)-b_{R_{Q_{0}}}|\langle |\bm{f}|_{q}\rangle_{3Q_{0}}+\langle |(b-b_{R_{Q_{0}}})||\bm{f}|_{q}\rangle_{3Q_{0}}\right)+\sum_{j}\left|\overline{\left[b,T\right]}_q(\bm{f}\chi_{3P_{j}})(x)\right|\chi_{P_{j}}(x).
\end{split}
\]
 Iterating this estimate we obtain the claim with
$\mathcal{F}=\{P_{j}^{k}\}$ where $\{P_{j}^{0}\}=\{Q_{0}\}$, $\{P_{j}^{1}\}=\{P_{j}\}$
and $\{P_{j}^{k}\}$ are the cubes obtained at the $k$-th stage of
the iterative process. Now we observe that for any arbitrary family
of disjoint cubes $P_{j}\in\mathcal{D}(Q_{0})$ we have that by the
sublinearity of $\overline{[b,T]}_{q}$,
\[
\begin{split} & \left|\overline{[b,T]}_{q}(\bm{f}\chi_{3Q_{0}})(x)\right|\chi_{Q_{0}}(x)\leq\left|\overline{[b,T]}_{q}(\bm{f}\chi_{3Q_{0}})(x)\right|\chi_{Q_{0}\setminus\bigcup_{j}P_{j}}(x)\\
+ & \sum_{j}\left|\overline{[b,T]}_{q}(\bm{f}\chi_{3Q_{0}\setminus3P_{j}})(x)\right|\chi_{P_{j}}(x)+\sum_{j}\left|\overline{[b,T]}_{q}(\bm{f}\chi_{3P_{j}})(x)\right|\chi_{P_{j}}(x).
\end{split}
\]
So it suffices to show that we can choose a family of pairwise disjoint
cubes $P_{j}\in\mathcal{D}(Q_{0})$ with $\sum_{j}|P_{j}|\leq\frac{1}{2}|Q_{0}|$
and such that for a.e. $x\in Q_{0}$,
\[
\begin{split} & \left|\overline{[b,T]}_{q}(\bm{f}\chi_{3Q_{0}})(x)\right|\chi_{Q_{0}\setminus\bigcup_{j}P_{j}}(x)+\sum_{j}\left|\overline{[b,T]}_{q}(\bm{f}\chi_{3Q_{0}\setminus3P_{j}})(x)\right|\chi_{P_{j}}(x)\\
 & \leq c_{n}C_{T}\left(|b(x)-b_{R_{Q_{0}}}|\langle |\bm{f}|_{q}\rangle_{3Q_{0}}+\langle |(b-b_{R_{Q_{0}}})||\bm{f}|_{q}\rangle_{3Q_{0}}\right).
\end{split}
\]
Now we recall that $[b,T]f=[b-c,T]f=(b-c)Tf-T((b-c)f)$ for every
$c\in\mathbb{R}$. Then
\[
\begin{split} & \overline{[b,T]}_{q}(\bm{f}\chi_{3Q_{0}})(x)\chi_{Q_{0}\setminus\bigcup_{j}P_{j}}(x)=\left(\sum_{k=1}^{\infty}\left|[b-b_{R_{Q_{0}}},T](f_{k}\chi_{3Q_{0}})(x)\right|^{q}\right)^{\frac{1}{q}}\chi_{Q_{0}\setminus\bigcup_{j}P_{j}}(x)\\
\leq & \left(\sum_{k=1}^{\infty}\left|\left(b(x)-b_{R_{Q_{0}}}\right)T(f_{k}\chi_{3Q_{0}})(x)-T\left(\left(b-b_{R_{Q_{0}}}\right)f_{k}\chi_{3Q_{0}}\right)(x)\right|^{q}\right)^{\frac{1}{q}}\chi_{Q_{0}\setminus\bigcup_{j}P_{j}}(x)\\
\leq & \left(\sum_{k=1}^{\infty}\left(\left|\left(b(x)-b_{R_{Q_{0}}}\right)T(f_{k}\chi_{3Q_{0}})(x)\right|+\left|T\left(\left(b-b_{R_{Q_{0}}}\right)f_{k}\chi_{3Q_{0}}\right)(x)\right|\right)^{q}\right)^{\frac{1}{q}}\chi_{Q_{0}\setminus\bigcup_{j}P_{j}}(x)\\
= & \left|b(x)-b_{R_{Q_{0}}}\right|\overline{T}_{q}\left(\bm{f}\chi_{3Q_{0}}\right)(x)\chi_{Q_{0}\setminus\bigcup_{j}P_{j}}(x)+\overline{T}_{q}\left(\left(b-b_{R_{Q_{0}}}\right)\bm{f}\chi_{3Q_{0}}\right)(x)\chi_{Q_{0}\setminus\bigcup_{j}P_{j}}(x).
\end{split}
\]
Analogously we also have that
\[
\begin{split} & \sum_{j}\left|\overline{[b,T]}_{q}(\bm{f}\chi_{3Q_{0}\setminus3P_{j}})(x)\right|\chi_{P_{j}}(x)\\
 & \le \sum_{j}\left(\left|b(x)-b_{R_{Q_{0}}}\right|\overline{T}_{q}\left(\bm{f}\chi_{3Q_{0}\setminus3P_{j}}\right)(x)+\overline{T}_{q}\left(\left(b-b_{R_{Q_{0}}}\right)\bm{f}\chi_{3Q_{0}\setminus3P_{j}}\right)\right)\chi_{P_{j}}(x).
\end{split}
\]
And combining both estimates
\begin{eqnarray*}
 &  & \left|\overline{[b,T]}_{q}(\bm{f}\chi_{3Q_{0}})(x)\right|\chi_{Q_{0}\setminus\bigcup_{j}P_{j}}(x)+\sum_{j}\left|\overline{[b,T]}_{q}(\bm{f}\chi_{3Q_{0}\setminus3P_{j}})(x)\right|\chi_{P_{j}}(x)\leq I_{1}+I_{2},
\end{eqnarray*}
where
\begin{equation}
I_{1}=\left|b(x)-b_{R_{Q_{0}}}\right|\left(\left|\overline{T}_{q}(\bm{f}\chi_{3Q_{0}})(x)\right|\chi_{Q_{0}\setminus\bigcup_{j}P_{j}}(x)+\sum_{j}\left|\overline{T}_{q}(\bm{f}\chi_{3Q_{0}\setminus3P_{j}})(x)\right|\chi_{P_{j}}(x)\right)\label{eq:3.2-1}
\end{equation}
and
\begin{equation}
I_{2}=\left|\overline{T}_{q}\left((b-b_{R_{Q_{0}}})\bm{f}\chi_{3Q_{0}}\right)(x)\right|\chi_{Q_{0}\setminus\bigcup_{j}P_{j}}(x)+\sum_{j}\left|\overline{T}_{q}\left((b-b_{R_{Q_{0}}})\bm{f}\chi_{3Q_{0}\setminus3P_{j}}\right)(x)\right|\chi_{P_{j}}(x).\label{eq:3.3-1}
\end{equation}
Now we define the set $E=E_{1}\cup E_{2}$ where
\[
E_{1}=\left\{ x\in Q_{0}\,:\,|\bm{f}|_{q}>\alpha_{n}\langle |\bm{f}|_{q}\rangle_{3Q_{0}}\right\} \cup\left\{ x\in Q_{0}\,:\,\mathcal{M}_{T_{q},Q_{0}}\bm{f}>\alpha_{n}C_{T}\langle |\bm{f}|_{q}\rangle_{3Q_{0}}\right\}
\]
and
\[
\begin{split}E_{2} & =\left\{ x\in Q_{0}\,:\,|b-b_{R_{Q_{0}}}||\bm{f}|_{q}>\alpha_{n}\langle |b-b_{R_{Q_{0}}}||\bm{f}|_{q}\rangle _{3Q_{0}}\right\} \\
 & \cup\left\{ x\in Q_{0}\,:\,\mathcal{M}_{T,Q_{0}}\left((b-b_{R_{Q_{0}}})\bm{f}\right)>\alpha_{n}C_{T}\langle |b-b_{R_{Q_{0}}}||\bm{f}|_{q}\rangle_{3Q_{0}}\right\}.
\end{split}
\]
Since $\mathcal{M}_{\overline{T}_{q}}$ is of weak type $(1,1)$ with
\[
\|\mathcal{M}_{\overline{T}_{q}}\|_{L^{1}\rightarrow L^{1,\infty}}\leq c_{n}C_{T}.
\]
from this point it suffices to follow the arguments given in \cite[Theorem 1.1]{LORR} to end the proof.

\section{Sparse domination for rough singular integrals and commutators with rough singular integrals}

To obtain the sparse domination for $\overline T_q$ and $\overline{[b,T]}_q$, unlike our previous approach, we do not need to go through the original proof.   This is due to a very nice observation by Culiuc, Di Plinio and Ou \cite{CDO17}. We have
\begin{align*}
\Big|\sum_{j\in\mathbb Z}\int_{\mathbb{R}^{n}}T(f_j)g_jdx \Big|&\le c_n C_Ts'\sum_j\sum_{Q\in\mathcal{S}_j}\langle |f_j| \rangle_Q\langle |g_j|\rangle_{s,Q}|Q|\\
&\le 2c_n C_Ts'\int_{\mathbb R^n}\sum_j  \mathcal M_{1,s}(f_j,g_j)(x) dx,
\end{align*}
where
\[
\mathcal M_{r,s}(f,g)(x)=\sup_{Q\ni x} \langle |f |\rangle_{r,Q}\langle |g|\rangle_{s,Q},
\] and $\langle |h| \rangle_{u, Q}= \langle |h|^u \rangle_Q^{\frac 1u}$ with $u>1$.

In the case of $T_\Omega^*$, taking into account \eqref{eq:SparseRoughMax} and arguing as above
\[
\Big|\sum_{j\in\mathbb Z}\int_{\mathbb{R}^{n}}T_\Omega^*(f_j)g_jdx \Big|\leq 2c_n C_Ts'\int_{\mathbb R^n}\sum_j  \mathcal M_{s,s}(f_j,g_j)(x) dx.\]
For the commutator $[b,T_\Omega]$ with $b\in\BMO$ and $\Omega \in L^\infty(\mathbb{S}^{n-1})$, taking into account Theorem \ref{Thm:SparseCommRough}, we observe that choosing $u=\frac{s+1}{2}$ then $u'\leq2s'$ and we have that
\[
\begin{split}\sum_{Q\in{\mathcal{S}}_{j}}\langle f\rangle_{Q}\langle(b-b_{Q})g\rangle_{u,Q}|Q| & \leq\sum_{Q\in{\mathcal{S}}_{j}}\langle f\rangle_{Q}\langle b-b_{Q}\rangle_{u\left(\frac{s}{u}\right)',Q}\langle g\rangle_{s,Q}|Q|\\
 & \leq c_{n}u\left(\frac{s}{u}\right)'\|b\|_{BMO}\sum_{Q\in{\mathcal{S}}_{j}}\langle f\rangle_{Q}\langle g\rangle_{s,Q}|Q|\\
 & \leq c_{n}u\left(\frac{s}{u}\right)'\|b\|_{BMO}\sum_{Q\in{\mathcal{S}}_{j}}\langle f\rangle_{r,Q}\langle g\rangle_{s,Q}|Q|\\
 & \leq c_{n}s'\|b\|_{BMO}\sum_{Q\in{\mathcal{S}}_{j}}\langle f\rangle_{r,Q}\langle g\rangle_{s,Q}|Q|.
\end{split}
\]
On the other hand
\[
\sum_{Q\in{\mathcal{S}}_{j}}\langle(b-b_{Q})f\rangle_{Q}\langle g\rangle_{u,Q}|Q|  \leq c_{n}r'\|b\|_{BMO}\sum_{Q\in{\mathcal{S}}_{j}}\langle f\rangle_{r, Q}\langle g\rangle_{s,Q}|Q|,
\]
from which it readily follows that
\[
|\langle[b,T_{\Omega}]f,g\rangle|\le c_{n}s'(s'+r')\|b\|_{\BMO}\sum_{Q\in{\mathcal{S}}_{j}}\langle f\rangle_{r,Q}\langle g\rangle_{s,Q}|Q|.
\]
Consequently
\[\Big|\sum_{j\in\mathbb Z}\int_{\mathbb{R}^{n}}[b,T](f_j)g_jdx \Big|\leq c_{n}s'\max\{s',r'\}\|b\|_{\BMO}
\int_{\mathbb R^n}\sum_j  \mathcal M_{r,s}(f_j,g_j)(x) dx.\]
where
\[
\mathcal M_{r,s}(f,g)(x)=\sup_{Q\ni x} \langle |f |\rangle_{r,Q}\langle |g|\rangle_{s,Q}.
\]
The consideratons above reduce the proof of Theorem  \ref{thm:sparserough} to provide a sparse domination for $\overline{(\mathcal M_{r,s})}_1(\mathbf f,\mathbf g)$. That was already done in \cite{CDO17}. Here we would like to track the constants, so  present an alternative proof.
\begin{lem}\label{lem:sparsemaxi}
Let $1<q<\infty$, $1\le s< \frac{q'+1}2$ and $1\le r<\frac{q+1}2$. Then there exists a sparse family of dyadic cubes $\mathcal S$  such that
\[
\overline{(\mathcal M_{1,s})}_1(\mathbf f,\mathbf g)\le c_n qq' \sum_{Q\in \mathcal S} \langle |\mathbf f|_q \rangle_{r,Q}  \langle |\mathbf g|_{q'} \rangle_{s, Q} \chi_Q.\]
\end{lem}
\begin{proof}
Again, we use the three lattice theorem to reduce the problem to
study the related dyadic maximal operator. Namely, we shall prove
\[
\overline{(\mathcal{M}_{1,s}^{\mathcal{D}})}_{1}(\mathbf{f},\mathbf{g})\le c_{n}qq'\sum_{Q\in\mathcal{S}}\langle|\mathbf{f}|_{q}\rangle_{r,Q}\langle|\mathbf{g}|_{q'}\rangle_{s,Q}\chi_{Q},
\]
where $\mathcal{D}$ is a dyadic grid and
\[
\mathcal{M}_{1,s}^{\mathcal{D}}(f,g)(x)=\sup_{\substack{Q\ni x\\
Q\in\mathcal{D}
}
}\langle|f|\rangle_{r,Q}\langle|g|\rangle_{s,Q}.
\]
We shall use the Lerner-Nazarov formula. So we only need to calculate
the local mean oscillation. For every $x\in Q_{0}$, notice that
\[
\mathcal{M}_{r,s}^{\mathcal{D}}(f,g)(x)=\max\{\mathcal{M}_{r,s}^{\mathcal{D}}(f\chi_{Q_{0}},g\chi_{Q_{0}})(x),\sup_{\substack{Q\in\mathcal{D}\\
Q\supset Q_{0}
}
}\langle|f|\rangle_{r,Q}\langle|g|\rangle_{s,Q}\}.
\]
the second term on the right is constant, so based on this we define
\[
K_{0}=\sum_{j\in\mathbb{Z}}\sup_{\substack{Q\in\mathcal{D}\\
Q\supset Q_{0}
}
}\langle|f_{j}|\rangle_{r,Q}\langle|g_{j}|\rangle_{s,Q}.
\]
Then
\begin{align*}
\big|\{x\in Q_{0}:|\overline{(\mathcal{M}_{1,s}^{\mathcal{D}})}_{1}(\mathbf{f},\mathbf{g})(x)-K_{0}|>t\}\big| & \le\big|\{x\in Q_{0}:|\overline{(\mathcal{M}_{1,s}^{\mathcal{D}})}_{1}(\mathbf{f}\chi_{Q_{0}},\mathbf{g}\chi_{Q_{0}})(x)|>t\}\big|.
\end{align*}
Now we are in the position to apply the Fefferman-Stein inequality
for vector valued maximal operators. Since we need to track the constants,
here we use the version in Grafakos' book \cite[Theorem 5.6.6]{G}:
\begin{equation}
\|\overline{M}_{q}(\mathbf{f})\|_{L^{1,\infty}}\le c_{n}q'\||\mathbf{f}|_{q}\|_{L^{1}}\label{eq:lg1}.
\end{equation}
We also need the H\"older's inequality for the weak type spaces, which
can also be found in \cite[p. 16]{G}:
\begin{equation}
\|f_{1}\cdots f_{k}\|_{L^{p,\infty}}\le p^{-\frac{1}{p}}\prod_{i=1}^{k}p_{i}^{\frac{1}{p_{i}}}\|f_{i}\|_{L^{p_{i},\infty}},\label{eq:lg2}
\end{equation}
where $\frac{1}{p}=\sum_{i=1}^{k}\frac{1}{p_{i}}$ and $0<p_{i}<\infty$.
With (\ref{eq:lg1}) and (\ref{eq:lg2}) at hand, we have that since
$1<r,s<\infty,$
\begin{align*}
\|\overline{(\mathcal{M}_{1,s})}_{1}^{\mathcal{D}}(\mathbf{f},\mathbf{g})\|_{L^{\frac{rs}{r+s},\infty}} & \le r^{\frac{1}{r}}s^{\frac{1}{s}}\left(\frac{r+s}{rs}\right)^{\frac{r+s}{rs}}\|\overline{M}_{q}\mathbf{f}\|_{L^{r,\infty}}\|\overline{(M_{s})}_{q'}\mathbf{g}\|_{L^{s,\infty}}\\
 & \leq r^{\frac{1}{r}}s^{\frac{1}{s}}\left(\frac{r+s}{rs}\right)^{\frac{r+s}{rs}}\|\overline{\left(M_{r}\right)}_{q}\mathbf{f}\|_{L^{r,\infty}}\|\overline{(M_{s})}_{q'}\mathbf{g}\|_{L^{s,\infty}}\\
 & \leq r^{\frac{1}{r}}s^{\frac{1}{s}}\left(\frac{r+s}{rs}\right)^{\frac{r+s}{rs}}\|\overline{M}_{q/r}\mathbf{|f}|^{r}\|_{L^{1,\infty}}^{\frac{1}{r}}\|\overline{M}_{q'/s}\mathbf{\mathbf{|g|}^{s}}\|_{L^{1,\infty}}^{\frac{1}{s}}\\
 & \leq c_{n}r^{\frac{1}{r}}s^{\frac{1}{s}}\left(\frac{r+s}{rs}\right)^{\frac{r+s}{rs}}\left(\frac{q}{r}\right)^{'}\left(\frac{q'}{s}\right)^{'}\||\mathbf{f}|_{q}^{r}\|_{L^{1}}^{\frac{1}{r}}\||\mathbf{g}|_{q'}^{s}\|_{L^{1}}^{\frac{1}{s}}.
\end{align*}
Now we observe that $c_{n}r^{\frac{1}{r}}s^{\frac{1}{s}}\left(\frac{r+s}{rs}\right)^{\frac{r+s}{rs}}\left(\frac{q}{r}\right)^{'}\left(\frac{q'}{s}\right)^{'}\leq c_{n}qq'=\kappa$.
Then,
\begin{align*}
\big|\{x\in Q_{0}:|\overline{(\mathcal{M}_{1,s}^{\mathcal{D}})}_{1}(\mathbf{f}\chi_{Q_{0}},\mathbf{g}\chi_{Q_{0}})(x)|>t\}\big| & \le\frac{\kappa^{\frac{rs}{r+s}}}{t^{\frac{rs}{r+s}}}\Big(\int_{Q_{0}}|\mathbf{f}|_{q}^{r}\Big)^{\frac{1}{r}\frac{rs}{r+s}}\Big(\int_{Q_{0}}|\mathbf{g}|_{q'}^{s}\Big)^{\frac{1}{s}\frac{rs}{r+s}}.
\end{align*}
Taking into account the preceding estimates, we have that
\begin{align*}
\omega_{\lambda}((\overline{\mathcal{M}_{1,s}^{\mathcal{D}}})_{1}(\mathbf{f},\mathbf{g}),Q_{0}) & \le\big((\overline{\mathcal{M}_{1,s}^{\mathcal{D}}})_{1}(\mathbf{f},\mathbf{g})-K_{0}\big)^{*}(\lambda|Q_{0}|)\\
 & \le c_{n}qq'\lambda^{-\frac{r+s}{rs}}\langle|\mathbf{f}|_{q}\rangle_{r,Q_{0}}\langle|\mathbf{g}|_{q'}\rangle_{s,Q_{0}}\chi_{Q_{0}}\\
 & \le c_{n}qq'\lambda^{-2}\langle|\mathbf{f}|_{q}\rangle_{r,Q_{0}}\langle|\mathbf{g}|_{q'}\rangle_{s,Q_{0}}\chi_{Q_{0}},
\end{align*}
where the last inequality holds since $0<\lambda<1$.
From this point, a direct application of Lerner-Nazarov formula (Theorem \ref{LernerFormula-1}) together with the $3^n$-dyadic lattices trick ends the proof.
\end{proof}

\section{Strong type estimate for commutators via conjugation method}\label{Sec:ProofApComm}
The proof of Theorem \ref{Thm:ApComm} is a straightforward consequence of Theorem \ref{Thm:ApT} and the following Lemma, which is based on the so called conjugation method, that consists in obtaining uniform estimates for a suitable family of operators $T_{z}(f):=e^{zb}T\left(\frac{f}{e^{zb}}\right)$ with $z\in\mathbb{C}$.

\begin{lem}\label{Thm:ConjMethodVec}
Let $1<p,q<\infty$ and  $w\in A_{p}$. Let $b\in \BMO$  and $T$ be a linear operator such that
\[\|\overline{T}_q\bm{f}\|_{L^p(w)}\leq c_{n,p,q}\varphi\left([w]_{A_p},[w]_{A_\infty},[\sigma]_{A_\infty}\right)\||\bm{f}|_q\|_{L^p(w)}.
\]
Then
\[
\left\Vert \overline{[b,T]}_{q}\bm{f}\right\Vert _{L^{p}(w)}
\leq c_{n,p,q}\|b\|_{\BMO}\kappa_{w,v}\||\bm{f}|_{q}\|_{L^{p}(w)},
\]
where
\[\kappa_{w,v}=\varphi\left(c_{n,p}[w]_{A_{p}},c_n[w]_{A_\infty}, c_n[\sigma]_{A_\infty}\right)\left([w]_{A_\infty}+[\sigma]_{A_\infty}\right).\]
\end{lem}
\begin{proof}
We know that
\[
[b,T]f=\left.\frac{d}{dz}e^{zb}T(fe^{-zb})\right|_{z=0}=\frac{1}{2\pi i}\int_{|z|=\varepsilon}\frac{T_{z}(f)}{z^{2}}\,dz\,,\quad\varepsilon>0
\]
where
\[
z\to T_{z}(f):=e^{zb}T\left(\frac{f}{e^{zb}}\right)\qquad z\in\mathbb{C}.
\]
Taking that into account
\[
\left\Vert \overline{[b,T]}_{q}\bm{f}\right\Vert _{L^{p}(w)}  =\left\Vert \left(\sum_{j=1}^{\infty}\left|[b,T] f_{j}(x)\right|^{q}\right)^{\frac{1}{q}}\right\Vert _{L^{p}(w)}=\left\Vert \left(\sum_{j=1}^{\infty}\left|\frac{1}{2\pi i}\int_{|z|=\varepsilon}\frac{T_{z}(f_{j})}{z^{2}}\,dz\right|^{q}\right)^{\frac{1}{q}}\right\Vert _{L^{p}(w)}.
\]
Now we use Minkowski inequality with respect to the measures $dz$
and $\ell^{q}$. Then
\[
\left\Vert \left(\sum_{j=1}^{\infty}\left|\frac{1}{2\pi i}\int_{|z|=\varepsilon}\frac{T_{z}(f_{j})}{z^{2}}\,dz\right|^{q}\right)^{\frac{1}{q}}\right\Vert _{L^{p}(w)}\leq\frac{1}{2\pi\varepsilon^{2}}\left\Vert \int_{|z|=\varepsilon}\left(\sum_{j=1}^{\infty}|T_{z}(f_{j})|^{q}\right)^{\frac{1}{q}}\,dz\right\Vert _{L^{p}(w)}.
\]
Using again Minkowski inequality again and we have that
\[
\frac{1}{2\pi\varepsilon^{2}}\left\Vert \int_{|z|=\varepsilon}\left(\sum_{j=1}^{\infty}|T_{z}(f_{j})|^{q}\right)^{\frac{1}{q}}\,dz\right\Vert _{L^{p}(w)}\leq\frac{1}{\varepsilon}\sup_{|z|=\varepsilon}\left\Vert \left(\sum_{j=1}^{\infty}|T_{z}(f_{j})|^{q}\right)^{\frac{1}{q}}\right\Vert _{L^{p}(w)}.
\]
From this point
\[\left\Vert \left(\sum_{j=1}^{\infty}|T_{z}(f_{j})|^{q}\right)^{\frac{1}{q}}\right\Vert _{L^{p}(w)}
=\left\Vert \left(\sum_{j=1}^{\infty}\left|T\left(\frac{f}{e^{zb}}\right)\right|^{q}\right)^{\frac{1}{q}}\right\Vert _{L^p({e^{Re(zb)p}w)}}, \]
and using the hypothesis on $T_q$ we obtain
\[
\left\Vert \left(\sum_{j=1}^{\infty}\left|T\left(\frac{f}{e^{zb}}\right)\right|^{q}\right)^{\frac{1}{q}}\right\Vert _{L^p({e^{Re(zb)p}w)}} \leq c_{n,p,q} \kappa\||\bm{f}|_q\|_{L^p(w)}
\]
with
\[\kappa=\varphi\left(
\left[e^{Re(zb)p}w\right]_{A_p},
\left[e^{Re(zb)p}w\right]_{A_\infty},
\left[\left(e^{Re(zb)p}w\right)^{\frac{1}{1-p}}\right]_{A_\infty}
\right).
\]

Now taking into account \cite[Lemma 2.1]{HConm}  and \cite[Lemma 7.3]{HPAinfty} we have that $[e^{\text{Re}(bz)p}w]_{A_{p}}\leq c_{n,p}[w]_{A_{p}}$, $[e^{\text{Re}(bz)p}w]_{A_{\infty}}\leq c_{n}[w]_{A_{\infty}}$ and $\left[\left(e^{Re(zb)p}w\right)^{\frac{1}{1-p}}\right]_{A_\infty} \leq c_{n}[\sigma]_{A_{\infty}}$ provided that $|z|\leq\frac{\varepsilon_{n,p}}{\|b\|_{\BMO}([w]_{A_{\infty}}+[\sigma]_{A_{\infty}})}$.
Combining all the estimates we obtain the desired result.
\end{proof}

\section{Proofs of $C_p$ condition estimates}\label{Sec:ProofThmCp}

In this section we will prove  the theorems stated in Section \ref{Cp} concerning the $C_p$ condition.

\subsection{Proofs of $M^\sharp$ approach results. Theorems  \ref{strong p}, \ref{strong pmultilinear} and \ref{commutator}}
\begin{proof}[Proof of Theorem  \ref{strong p}]       Let $\delta \in (0,1)$ be a parameter to be chosen. Then, by the Lebesgue differentiation theorem
\begin{equation*}
\|T(f)\|_{L^p(w)} \leq \|M(   T(f)^{\delta}   )^{\frac{1}{\delta}} \|_{L^p(w)}= \|M(   T(f)^{\delta}   )\|_{L^{p/\delta}(w)}^\frac{1}{\delta}.
\end{equation*}
Now we choose $ \delta\in(0,1)$ such that \[\max\{1,p\}<\frac{p}{\delta}<\max\{1,p\}+\varepsilon.\]
If we denote $\varepsilon_{1}=\max\{1,p\}+\varepsilon-\frac{p}{\delta}$ then, since $w\in C_{\max\left\{ 1,p\right\} +\varepsilon}$,  we have that $w\in C_{p/\delta+\varepsilon_{1}}$ and a direct application of Lemma \ref{yabuta} combined with the $(D)$ condition yields
\begin{equation*}
\|T(f)\|_{L^p(w)} \leq c\, \|M^{\#}(   T(f)^{\delta}   )\|_{L^{p/\delta}(w)}^\frac{1}{\delta}=
c\, \|M_{\delta}^{\#}(T(f) )\|_{L^{p}(w)} \leq c\, \| Mf\|_{L^{p}(w)},
\end{equation*}
which is the desired result. The vector valued case is analogous, assuming the $(D_q)$ condition instead so we omit the proof.

\end{proof}

\begin{proof}[Proof of Theorem \ref{strong pmultilinear}] The proof is similar to the case $m=1$.
Let $\delta \in (0,\frac1m)$ be a parameter to be chosen. Then, as above
\begin{equation*}
\|T(\vec f) \|_{L^p(w)} \leq \|M(   |T(\vec f)|^{\delta}   )^{\frac{1}{\delta}} \|_{L^p(w)}.
\end{equation*}

Now we choose $ \delta\in(0,\frac{1}{m})$ such that
\[m\max\{1,p\}<\frac{p}{\delta}<m\max\{1,p\}+\varepsilon.\]
If we denote $\varepsilon_{m}=m\max\{1,p\}+\varepsilon-\frac{p}{\delta}$ then, since $w\in C_{m\max\left\{ 1,p\right\} +\varepsilon}$,  we have that $w\in C_{p/\delta+\varepsilon_{m}}$ and a direct application of Lemma \ref{yabuta} combined with \eqref{pointw} yields

\begin{equation*}
\|T(\vec{f})\|_{L^p(w)} \leq c\, \|M^{\#}(   |T(\vec{f})|^{\delta}   )\|_{L^{p/\delta}(w)}^\frac{1}{\delta}=
c\, \|M_{\delta}^{\#}(T(\vec{f}) )\|_{L^{p}(w)} \leq c\, \| \mathcal{M}(\vec{f})\|_{L^{p}(w)},
\end{equation*}
as we wanted to prove.
\end{proof}

\begin{proof}[Proof of Theorem \ref{commutator}]

We will use the key pointwise estimate \eqref{Pointwise}: if $0<\delta<\delta_1$: there exists a positive constant $c=c_{\delta,\delta_1,T} $ such that,
\begin{equation}\label{Keyestimate}
M^{\#}_{\delta}([b,T]f)(x)\le c_{\delta,\delta_1,T} \|b\|_{BMO} \left( M_{\delta_1}(Tf)   +M^2(f)(x) \right) .
\end{equation}
By the Lebesgue differentiation theorem,
\begin{equation*}
\|[b,T]f\|_{L^p(w)} \leq \|M(   |[b,T]f|^{\delta}   )^{\frac{1}{\delta}} \|_{L^p(w)}.
\end{equation*}
We choose $ 0<\delta<\delta_1<1$ such that \[\max\{1,p\}<\frac{p}{\delta_1}<\frac{p}{\delta}<\max\{1,p\}+\varepsilon.\]
Now, if we denote $\varepsilon_{1}=\max\{1,p\}+\varepsilon-\frac{p}{\delta}$ then, since $w\in C_{\max\left\{ 1,p\right\} +\varepsilon}$,  we have that $w\in C_{p/\delta+\varepsilon_{1}}$ and a direct application of Lemma \ref{yabuta} yields
\begin{equation*}
\|[b,T]\|_{L^p(w)} \leq c\|M^{\#}(   |[b,T]|^{\delta}   )^{\frac{1}{\delta}} \|_{L^p(w)}.
\end{equation*}
Combining the preceding estimate with \eqref{Pointwise},
$$
\|[b,T]\|_{L^p(w)} \leq c\, \|b\|_{BMO}\left(\|M_{\delta_1}(Tf)  \|_{L^p(w)} + \|M^{2}f \|_{L^p(w)} \right).
$$
For the second term we are done, whilst for first one,  taking into account our choice for $\delta_1$ and arguing as in the proof of Theorem \ref{strong p}, $$
\|M_{\delta_1}(Tf)  \|_{L^p(w)} \leq c \|Mf  \|_{L^p(w)}
$$
and we are done.
Taking into account \eqref{Pointwiseq} the vector-valued case is analogous so we omit the proof.
\end{proof}

\subsection{Proofs of Theorems \ref{Thm:Cp} and \ref{Thm:CpComm}}
The proof of Theorem \ref{Thm:Cp} is actually a consequence of the sparse domination combined with the following Theorem.
\begin{thm} \label{thm:CpSparse}Let $1< p<q<\infty$. Let $\mathcal{S}$ be a sparse family and $w\in C_q$. Then
\[\begin{split}
\|\mathcal{A}_{\mathcal{S}}f\|_{L^p(w)}&\leq c\|Mf\|_{L^p(w)},\\
\|\mathcal{A}_{\mathcal{S}}f\|_{L^{p,\infty}(w)}&\leq c\|Mf\|_{L^{p,\infty}(w)}.
\end{split}\]\end{thm}

Something analogous happens with Therorem \ref{Thm:CpComm}. It is a consequence of the sparse domination combined with the following result.
\begin{thm} \label{thm:CpSparseComm}Let $1< p<q<\infty$. Let $\mathcal{S}$ be a sparse family and $w\in C_q$, and $b\in \BMO$. Then
\[\begin{split}
\|\mathcal{T}_{b,\mathcal{S}}f\|_{L^p(w)}&\leq c\|b\|_{\BMO}\|Mf\|_{L^p(w)},\\
\|\mathcal{T}_{b,\mathcal{S}}f\|_{L^{p,\infty}(w)}&\leq c\|b\|_{\BMO}\|Mf\|_{L^{p,\infty}(w)}.
\end{split}\]
and
\[\begin{split}
\|\mathcal{T}^*_{b,\mathcal{S}}f\|_{L^p(w)}&\leq c\|b\|_{\BMO}\|M^2f\|_{L^p(w)},\\
\|\mathcal{T}^*_{b,\mathcal{S}}f\|_{L^{p,\infty}(w)}&\leq c\|b\|_{\BMO}\|M^2f\|_{L^{p,\infty}(w)}.
\end{split}\]
\end{thm}

To establish the preceding results we will rely upon some Lemmas that are based on ideas of \cite{S}.

\subsubsection{Lemmata}

In this section we present the technical lemmas needed to establish Theorems \ref{Thm:Cp} and \ref{Thm:CpComm}. Results here are essentially an elaboration of Sawyer's arguments \cite{S}.

Let $\Omega_k:=\{f> 2^k\}$ and define
\[
  (M_{k,p,q}(f)(x))^p= 2^{kp} \int_{\Omega_k} \frac{d(y, \Omega_k^c)^{n(q-1)}}{d(y, \Omega_k^c)^{nq}+|x-y|^{nq}}dy.
\]
When $\Omega_k$ is open let $\Omega_k= \cup_j Q_j^k$ be the Whitney decomposition, i.e., $Q_j^k$ are pairwise disjoint and
\[
8< \frac{\text{dist} (Q_j^k, \Omega_k^c)}{\text{diam} Q_j^k}\le 10,\quad \sum_j \chi_{6Q_j^k} \le C_n\chi_{\Omega_k},
\] then it is easy to check that
\[
  M_{k,p,q}(f)^p\eqsim 2^{kp} \sum_j  M(\chi_{Q_j^k})^q.
\]
Our key lemma is the following
\begin{lem}\label{Lem:Tec1}
Suppose that $1<p<q<\infty$ and that $w$ satisfies the $C_q$ condition. Then for all compactly supported $f$,
\[
\sup_k \int (M_{k,p,q}(Mf))^p w \le C \|Mf\|_{L^{p,\infty}(w)}^p.
\]
\end{lem}
\begin{proof}
Let $\Omega_k:=\{Mf> 2^k\}=\cup_j Q_j^k$ be the Whitney decomposition. Let $N$ be a positive integer (to be chosen later) and fix a Whitney cube $Q_i^{k-N}$. We now claim
\begin{equation}\label{eq:claim1}
|\Omega_k\cap 3 Q_i^{k-N}|\le C_n 2^{-N}|Q_i^{k-N}|.
\end{equation}
Indeed, let $g=f\chi_{22 \sqrt n Q_i^{k-N}}$ and $h=f-g$. Let $x_0\in 22 \sqrt n Q_i^{k-N}\setminus \Omega_{k-N}$. It is easy to check that for any $x\in 3Q_i^{k-N}$, we have
\[
M(h)(x)\le c_n M(f)(x_0)\le c_n 2^{k-N}.
\]
Let $N$ be sufficiently large such that $c_n 2^{-N}\le 1/2$.  Then
\begin{align*}
|\Omega_k\cap 3 Q_i^{k-N}|&=|\{x\in 3Q_i^{k-N}: M(f)> 2^k\}|\\
&\le |\{x\in 3Q_i^{k-N}: M(g)> 2^{k-1}\}|\\
&\le 2^{1-k}c_n\int g\le 2^{1-k}c_n |22 \sqrt n Q_i^{k-N}| M(f)(x_0)\\
&\le C_n 2^{-N} |Q_i^{k-N}|.
\end{align*}
As that in \cite{S}, define $S(k)=2^{kp} \sum_j  \int M(\chi_{Q_j^k})^q w$ and $S(k; N,i)= 2^{kp} \sum_j  \int M(\chi_{Q_j^k})^q w$, where the latter sum is taken over those $j$ for which $Q_j^k\cap Q_i^{k-N}\neq \emptyset$. Since $Q_j^k\cap Q_i^{k-N}\neq \emptyset$ together with \eqref{eq:claim1} implies $\ell (Q_j^k)\le \ell (Q_i^{k-N})$ for large $N$, and this further implies $Q_j^{k}\subset Q_i^{k-N}$, we have
\begin{align*}
S(k; N,i)&\le \int 2^{kp}\sum_{j: Q_j^k \subset 3Q_i^{k-N}} |M\chi_{Q_j^k}|^q w\\
&={ \int_{5 Q_i^{k-N}} }+  {\int_{\bbR^n \setminus 5 Q_i^{k-N}}} :=I+II \quad \mbox{for $N$ large.}
\end{align*}
By the argument in \cite{S}, we know
\begin{align*}
I&\le C_\delta 2^{kp}w(5Q_i^{k-N})+ \delta 2^{kp} \int |M\chi_{Q_i^{k-N}}|^q w,
\end{align*}
where we have used $M(\chi_{6 Q_i^{k-N}})\eqsim M(\chi_{ Q_i^{k-N}})$. Next we estimate $II$, we have
\begin{align*}
II &\le c_n 2^{kp} \int_{\bbR^n \setminus 5 Q_i^{k-N}}\frac{\sum{|Q_j^k|^q}}{|x-c_{Q_i^{k-N}}|^{nq}} w(x) dx\\
&\le c_{n,q} 2^{kp} \int_{\bbR^n \setminus 5 Q_i^{k-N}}\Big(\frac{2^{-N}|Q_i^{k-N}|}{|x-c_{Q_i^{k-N}}|^n} \Big)^q w(x) dx\\
&\le c_{n,q} 2^{N(p-q)} 2^{(k-N)p}\int |M\chi_{Q_i^{k-N}}|^q w.
\end{align*}
Thus for $N$ large (depending on $p,q$),
\begin{align*}
S(k)&\le \sum_i S(k; N,i)\\
&\le C_\delta c_n 2^{kp} w(\Omega_{k-N}) + (\delta 2^{Np}+ c_{n,q} 2^{N(p-q)})S(k-N)\\
&\le C_{n,\delta} 2^{kp}w(\Omega_{k-N}) + \frac 12 S(k-N).
\end{align*}
Taking the supremum over $k\le M$, we get
\[
\sup_{k\le M} \int (M_{k,p,q}(Mf))^p w\le c_{n,p,q} \|M(f)\|_{L^{p,\infty}(w)}^p,
\]
provided that
\[
\sup_{k\le M} \int (M_{k,p,q}(Mf))^p w<\infty.
\]
By monotone convergence, we can assume that $f$ has compact support, say $\supp f\subset Q$. Without loss of generality, assume $f\ge 0$ and $2^s<\langle f \rangle_Q \le 2^{s+1}$. Then it is easy to check that
\[
M(f)\gtrsim 2^sM(\chi_Q).
\]
Moreover, for $k\ge s+1$, $\Omega_k \subset 3Q$ and we have
\begin{align*}
&\sup_{s+1< k \le M}2^{kp}\sum_j \int M(\chi_{Q_j^k})^{q} w\\&\le \sup_{s+1< k \le M}2^{kp}  \int M(\chi_{Q})^{q} w\\
&=\sup_{s+1< k \le M}2^{kp} \sum_{\ell\ge 1} \int_{2^{\ell+1}Q\setminus 2^{\ell }Q} M(\chi_{Q})^{q} w+ \sup_{s+1< k \le M}2^{kp}  \int_{2Q} M(\chi_{Q})^{q} w\\
&=I+II.
\end{align*}
First, we estimate $II$. We have \begin{align*}
II&\le 2^{Mp} w(2Q)\le 2^{Mp} w(\{x: M\chi_Q(x)\ge \frac 1{2^n} \})\\
&\le 2^{Mp} c_{n,p}\|M\chi_Q\|_{L^{p,\infty}(w)}^p\le 2^{Mp-sp} c_{n,p}\|Mf\|_{L^{p,\infty}(w)}^p<\infty.
\end{align*}
Next we estimate $I$. Direct calculations give us
\begin{align*}
I&\le \sup_{s+1< k \le M}2^{kp} \sum_{\ell\ge 1} c_{n,q}2^{-nq\ell} w(2^{\ell+1}Q \setminus 2^\ell Q)\\
&\le \sup_{s+1< k \le M}2^{kp} \sum_{\ell\ge 1} c_{n,q}2^{-nq\ell} w(\{x: M(\chi_Q)\ge \frac 1{2^{(\ell+1)n}} \})\\
&\le 2^{Mp} c_{n,q}\sum_{\ell\ge 1}2^{-n(q-p)\ell}\| M(\chi_Q)\|_{L^{p,\infty}(w)}^p\\
&\le c_{n,p,q} 2^{Mp} 2^{-sp}\| Mf\|_{L^{p,\infty}(w)}^p< \infty.
\end{align*}
It remains to consider the case $k\le s$. We still follow the idea of Sawyer, but with slight changes. In this case, $\Omega_k \subset (2^{\frac{s-k+2} n}+1)Q$. Then again,
\begin{align*}
\sup_{k\le s} 2^{kp}\sum_{j}  \int |M(\chi_{Q_j^k})|^q w&\le \sup_{k\le s} 2^{kp} c_{n,q} \int |M\chi_{2^{\frac{s-k}{n}}Q}|^q w\\
&= 2^{sp}\sup_{m \ge 0} 2^{-mp} c_{n,q} \int |M\chi_{2^{\frac{m}{n}}Q}|^q w\\
&\le 2^{sp}\sup_{m \ge 0} 2^{-mp} c_{n,q} \int_{2^{\frac mn +1}Q} |M\chi_{2^{\frac{m}{n}}Q}|^q w\\
&+  2^{sp}\sup_{m \ge 0} 2^{-mp} c_{n,q} \sum_{\ell \ge 1}\int_{2^{\frac mn +\ell+1}Q  \setminus 2^{\frac mn +\ell}Q} |M\chi_{2^{\frac{m}{n}}Q}|^q w\\
&\le c_{n,p,q} \| Mf\|_{L^{p,\infty}(w)}^p< \infty,
\end{align*}
where the last step follows from similar calculations as the ones above. Now
\[
\sup_{k\le M} \int (M_{k,p,q}(Mf))^p w\le c_{n,p,q} \|M(f)\|_{L^{p,\infty}(w)}^p
\]
and taking the supremum over $M$ we conclude the proof.
\end{proof}
Our last result in this subsection is the following technical lemma.
\begin{lem}\label{Lem:Tec2}
Let $\{Q_j^k\}_j$ be a collection of disjoint cubes in $\{Mf> 2^k\}$, then \[
 2^{kp} \sum_j  M(\chi_{Q_j^k})^q\lesssim   M_{k,p,q}(Mf)^p.\]
\end{lem}
\begin{proof}
The proof is straightforward. Indeed, let $c_{Q_j^k}$ be the center of $Q_j^k$, and $P$ be the cube from the Whitney decomposition of $\{Mf> 2^k\}$ which contains $c_{Q_j^k}$. Of course by the Whitney property, $Q_j^k \subset c_n P$ for some dimensional constant $c_n$. Then
\[
M(\chi_{Q_j^k}) \le M(\chi^{}_{c_n P})\le c_n' M(\chi^{}_{ P})
\]
and the result follows.
\end{proof}

\subsubsection{Proof of Theorem \ref{thm:CpSparse}}
We only provide the proof for the strong type $(p,p)$ estimate, since the weak type $(p,p)$ is analogous.
Let $\gamma>0$ a small parameter that will be chosen, then we have that
\begin{align*}
\|\mathcal{A}_{\mathcal{S}}f\|_{L^p(w)}^p&\le \sum_{k\in \bbZ} 2^{(k+1)p} w(\{x: 2^k<\mathcal{A}_{\mathcal{S}}f(x)\le 2^{k+1}\})\\
&\le c_p \sum_{k\in \bbZ} 2^{kp}w(\{x:  \mathcal{A}_{\mathcal{S}}f(x)> 2^{k}\})\\
&\le  c_p \sum_{k\in \bbZ} 2^{kp}w(\{x:  \mathcal{A}_{\mathcal{S}}f(x)> 2^{k}, M(f)(x)\le \gamma 2^k\})\\&+ c_p \sum_{k\in \bbZ} 2^{kp}w(\{x:   M(f)(x)> \gamma 2^k\})\\
&\le c_p \sum_{k\in \bbZ} 2^{kp}w(\{x:  \mathcal{A}_{\mathcal{S}}f(x)> 2^{k}, M(f)(x)\le \gamma 2^k\})+ c_{p,\gamma} \|Mf\|_{L^p(w)}^p.
\end{align*}
So we only need to estimate
\[
\sum_{k\in \bbZ} 2^{kp}w(\{x:  \mathcal{A}_{\mathcal{S}}f(x)> 2^{k}, M(f)(x)\le \gamma 2^k\}).
\]
Split $\mathcal S= \cup_{m}\mathcal S_{m}$, where
\[
\mathcal{S}_m:= \{Q\in \mathcal S: 2^{m}< \langle f \rangle_Q\le 2^{m+1}\}.
\]
It is easy to see that, if $2^{m} \ge \gamma 2^k$, then for $x\in Q \in \mathcal{S}_m$, $Mf(x)> \gamma 2^k$. Set $m_0= \lfloor \log_2 (\frac 1 \gamma) \rfloor+1$, then we have
%
\begin{align*}
&\sum_{k\in \bbZ} 2^{kp}w(\{x:  \mathcal{A}_{\mathcal{S}}f(x)> 2^{k},\, M(f)(x)\le \gamma 2^k\}) \\
&= \sum_{k\in \bbZ} 2^{kp} w\left(\left\{x: \sum_{m\le  k-m_0} \mathcal{A}_{\mathcal{S}_m}f(x)> 2^{k}(1-\frac 1{\sqrt 2})\sum_{m\le k-m_0}2^{\frac{m-k+m_0}2}, \,M(f)(x)\le \gamma 2^k\right\}\right)\\
&\le \sum_{k\in \bbZ} 2^{kp} \sum_{m\le  k-m_0} w\left(\left\{x:  \mathcal{A}_{\mathcal{S}_m}f(x)>  (1-\frac 1{\sqrt 2}) 2^{\frac{m+k+m_0}2}, \,M(f)(x)\le \gamma 2^k\right\}\right).
\end{align*}

Denote $b_m =\sum_{Q\in \mathcal S_m} \chi_Q $, then
$
\mathcal{A}_{\mathcal{S}_m}f\le 2^{m+1}b_m.
$ and therefore, if we denote  $\mathcal S_m^*$ is the collection of maximal dyadic cubes in $\mathcal S_m$, arguing as in the proof of Lemma \ref{Lem:Decay},
\begin{align*}
&\Big|\Big\{\mathcal{A}_{\mathcal{S}_m}f(x)>  (1-\frac 1{\sqrt 2}) 2^{\frac{m+k+m_0}2}\Big\}\Big|\le \Big|\Big\{b_m>   \frac {\sqrt 2-1}{2\sqrt 2}  2^{\frac{-m+k+m_0}2}\Big\}\Big|\\
&\leq \sum_{Q\in\mathcal{S}^*_m}\Big|\Big\{x\in Q? b_m>   \frac {\sqrt 2-1}{2\sqrt 2}  2^{\frac{-m+k+m_0}2}\Big\}\Big|\le \exp(-c2^{\frac{-m+k+m_0}2})\sum_{Q\in \mathcal S_m^*}|Q|.
\end{align*}
Now, by the $C_q$ condition, we have
\begin{align*}
&w\Big\{\mathcal{A}_{\mathcal{S}_m}f(x)>  (1-\frac 1{\sqrt 2}) 2^{\frac{m+k+m_0}2}\Big\}\\
&=\sum_{Q\in\mathcal{S}_m^*}w\Big\{x \in Q : \mathcal{A}_{\mathcal{S}_m}f(x)>  (1-\frac 1{\sqrt 2}) 2^{\frac{m+k+m_0}2}\Big\}\\
&\le \exp(-c\epsilon 2^{\frac{-m+k+m_0}2})\sum_{Q\in \mathcal S_m^*}\int M(\chi_Q)^q w.
\end{align*}
Since $\cup_{Q\in \mathcal S_m^*} Q \subset \{x: Mf(x)> 2^m\}$, a combined application of Lemmas  \ref{Lem:Tec1} and \ref{Lem:Tec2} yields the desired result.

\subsubsection{Proof of Theorem \ref{thm:CpSparseComm}}

We may assume that $\|b\|_{\BMO}=1$  Again we just settle the strong type estimate, since the weak-weak type $(p,p)$ estimate is analogous.

First we note that using \eqref{eq:HGen}
\[
\mathcal{T}_{b,\mathcal{S}}^{*}f(x)=\sum_{Q\in\mathcal{S}}\frac{1}{|Q|}\int_{Q}|b-b_{Q}||f|\chi_{Q} \lesssim \|b\|_{\BMO}\sum_{Q\in\mathcal{S}}\|f\|_{L\log L,Q}\chi_{Q}=\|b\|_{\BMO}\,\mathcal{A}_{L\log L,\mathcal{S}}f.
\]
Now we observe that we have that
\begin{align*}
\|\mathcal{A}_{L\log L,\mathcal{S}}f\|_{L^{p}(w)} & \le\sum_{k\in\bbZ}2^{(k+1)p}w(\{x:2^{k}<\mathcal{A}_{L\log L,\mathcal{S}}f\le2^{k+1}\})\\
 & \le c_{p}\sum_{k\in\bbZ}2^{kp}w(\{x:\mathcal{A}_{L\log L,\mathcal{S}}f(x)>2^{k}\})\\
 & \le c_{p}\sum_{k\in\bbZ}2^{kp}w(\{x:\mathcal{A}_{L\log L,\mathcal{S}}f(x)>2^{k},M_{L\log L}f(x)\le\gamma2^{k}\})\\
 & +c_{p}\sum_{k\in\bbZ}2^{kp}w(\{x:M_{L\log L}f(x)>\gamma2^{k}\})\\
 & \le c_{p}\sum_{k\in\bbZ}2^{kp}w(\{x:\mathcal{A}_{L\log L,\mathcal{S}}f(x)>2^{k},M_{L\log L}f(x)\le\gamma2^{k}\})\\
 &+c_{p,\gamma}\|M_{L\log L}f\|_{L^{p}(w)}^{p}.
\end{align*}
So we only need to estimate
\[
\sum_{k\in\bbZ}2^{kp}w(\{x:\mathcal{A}_{L\log L,\mathcal{S}}f(x)>2^{k},M_{L\log L}f(x)\le\gamma2^{k}\}).
\]
Split $\mathcal{S}=\cup_{m}\mathcal{S}_{m}$, where
\[
\mathcal{S}_{m}:=\{Q\in\mathcal{S}:2^{m}<\|f\|_{L\log L,Q}\le2^{m+1}\}.
\]
It is easy to see that, if $2^{m}\ge\gamma2^{k}$, then for $x\in Q\in \mathcal{S}_{m}$,
$M_{L\log L}f(x)>\gamma2^{k}$. Set $m_{0}=\lfloor\log_{2}(\frac{1}{\gamma})\rfloor+1$,
we have
\begin{align*}
 & \sum_{k\in\bbZ}2^{kp}w(\{x:\mathcal{A}_{L\log L,\mathcal{S}}f(x)>2^{k},M_{L\log L}f(x)\le\gamma2^{k}\})\\
 & =\sum_{k\in\bbZ}2^{kp}w(\{x:\sum_{m\le k-m_{0}}\mathcal{A}_{L\log L,\mathcal{S}_{m}}f(x)>2^{k}(1-\frac{1}{\sqrt{2}})\sum_{m\le k-m_{0}}2^{\frac{m-k+m_{0}}{2}},M_{L\log L}f(x)\le\gamma2^{k}\})\\
 & \le\sum_{k\in\bbZ}2^{kp}\sum_{m\le k-m_{0}}w(\{x:\mathcal{A}_{L\log L,\mathcal{S}_{m}}f(x)>(1-\frac{1}{\sqrt{2}})2^{\frac{m+k+m_{0}}{2}},M_{L\log L}f(x)\le\gamma2^{k}\}).
\end{align*}
Denote $b_{m}=\sum_{Q\in\mathcal{S}_{m}}\chi_{Q}$, then $\mathcal{A}_{L\log L,\mathcal{S}_{m}}f\le2^{m+1}b_{m}.$
and therefore, by sparseness,
\begin{align*}
 & \Big|\Big\{\mathcal{A}_{L\log L,\mathcal{S}_{m}}f(x)>(1-\frac{1}{\sqrt{2}})2^{\frac{m+k+m_{0}}{2}}\Big\}\Big|\\
 & \le\Big|\Big\{ b_{m}>\frac{\sqrt{2}-1}{2\sqrt{2}}2^{\frac{-m+k+m_{0}}{2}}\Big\}\Big|\le\exp(-c2^{\frac{-m+k+m_{0}}{2}})\sum_{Q\in\mathcal{S}_{m}^{*}}|Q|,
\end{align*}
where $\mathcal{S}_{m}^{*}$ is the collection of maximal dyadic cubes
in $\mathcal{S}_{m}$.  By the $C_{q}$ condition, we have
\begin{align*}
 & w\Big\{\mathcal{A}_{L\log L,\mathcal{S}_{m}}f(x)>(1-\frac{1}{\sqrt{2}})2^{\frac{m+k+m_{0}}{2}}\Big\}\\
 & \le\exp(-c\epsilon2^{\frac{-m+k+m_{0}}{2}})\sum_{Q\in\mathcal{S}_{m}^{*}}\int M(\chi_{Q})^{q}w.
\end{align*}
Since $\cup_{Q\in\mathcal{S}_{m}^{*}}Q\subset\{x:M_{L\log L}f(x)>2^{m}\}\subseteq\{x:M(Mf)(x)>2^{m-n}\}$
then we have that using Lemma \ref{Lem:Tec2}
\[
\sum_{Q\in\mathcal{S}_{m}^{*}}M(\chi_{Q})^{q}w\lesssim M_{m-n,p,q}(M(Mf))^{p}
\]
and consequently
\[
\sum_{Q\in\mathcal{S}_{m}^{*}}\int M(\chi_{Q})^{q}w\lesssim\int M_{m-n,p,q}(M(Mf))^{p}.
\]
Hence taking into account Lemma \ref{Lem:Tec1}
\[
\begin{split}w\Big\{\mathcal{A}_{L\log L,\mathcal{S}_{m}}f(x)>(1-\frac{1}{\sqrt{2}})2^{\frac{m+k+m_{0}}{2}}\Big\} & \lesssim\exp(-c\epsilon2^{\frac{-m+k+m_{0}}{2}})\int M_{m-n,p,q}(M(Mf))^{p}\\
 & \lesssim\exp(-c\epsilon2^{\frac{-m+k+m_{0}}{2}})\|M(Mf)\|_{L^{p}(w)}^{p}.\\
 & \simeq\exp(-c\epsilon2^{\frac{-m+k+m_{0}}{2}})\|M_{L\log L}f\|_{L^{p}(w)}^{p}.
\end{split}
\]
This yields
\[\begin{split}
&\sum_{k\in\bbZ}2^{kp}w(\{x:\mathcal{A}_{L\log L,\mathcal{S}}f(x)>2^{k},M_{L\log L}f(x)\le\gamma2^{k}\})\\
&\lesssim\sum_{k\in\bbZ}2^{kp}\sum_{m\le k-m_{0}}\exp(-c\epsilon2^{\frac{-m+k+m_{0}}{2}})\|M_{L\log L}f\|_{L^{p}(w)}^{p}
\end{split}
\]
and we are done.

Now we turn our attention to $\mathcal{T}_{b,\mathcal{S}}f(x)$. We observe that we have that arguing as before
\begin{align*}
\|\mathcal{T}_{b,\mathcal{S}}f(x)f\|_{L^{p}(w)}^{p} & \le\sum_{k\in\bbZ}2^{(k+1)p}w(\{x:2^{k}<\mathcal{T}_{b,\mathcal{S}}f(x)\le2^{k+1}\})\\
 & \le c_{p}\sum_{k\in\bbZ}2^{kp}w(\{x:\mathcal{T}_{b,\mathcal{S}}f(x)>2^{k},M(f)(x)\le\gamma2^{k}\})+c_{p,\gamma}\|Mf\|_{L^{p}(w)}^{p}.
\end{align*}
So we only need to estimate
\[
\sum_{k\in\bbZ}2^{kp}w(\{x:\mathcal{T}_{b,\mathcal{S}}f(x)>2^{k},M(f)(x)\le\gamma2^{k}\}).
\]
Split $\mathcal{S}=\cup_{m}\mathcal{S}_{m}$, where
\[
\mathcal{S}_{m}:=\{Q\in\mathcal{S}:2^{m}<\langle f\rangle_{Q}\le2^{m+1}\}.
\]

It is easy to see that, if $2^{m}\ge\gamma2^{k}$, then for $x\in Q\in \mathcal{S}_{m}$,
$Mf(x)>\gamma2^{k}$. Set $m_{0}=\lfloor\log_{2}(\frac{1}{\gamma})\rfloor+1$,
we have
\begin{align*}
 & \sum_{k\in\bbZ}2^{kp}w(\{x:\mathcal{T}_{b,\mathcal{S}}f(x)>2^{k},M(f)(x)\le\gamma2^{k}\})\\
 & =\sum_{k\in\bbZ}2^{kp}w\left(\left\{x:\sum_{m\le k-m_{0}}\mathcal{T}_{b,\mathcal{S}_{m}}f(x)>2^{k}(1-\frac{1}{\sqrt{2}})\sum_{m\le k-m_{0}}2^{\frac{m-k+m_{0}}{2}},M(f)(x)\le\gamma2^{k}\right\}\right)\\
 & \le\sum_{k\in\bbZ}2^{kp}\sum_{m\le k-m_{0}}w\left(\left\{x:\mathcal{T}_{b,\mathcal{S}_{m}}f(x)>(1-\frac{1}{\sqrt{2}})2^{\frac{m+k+m_{0}}{2}},M(f)(x)\le\gamma2^{k}\right\}\right).
\end{align*}
Now we observe that $\mathcal{T}_{b,\mathcal{S}_{m}}f(x)\le2^{m+1}\sum_{Q\in\mathcal{S}_{m}}|b(x)-b_{Q}|\chi_{Q}$,
therefore
\begin{align*}
 & \Big|\Big\{\mathcal{T}_{b,\mathcal{S}_{m}}f(x)>(1-\frac{1}{\sqrt{2}})2^{\frac{m+k+m_{0}}{2}}\Big\}\Big|\\
 & \le\Big|\Big\{\sum_{Q\in\mathcal{S}_{m}}|b(x)-b_{Q}|\chi_{Q}>\frac{\sqrt{2}-1}{2\sqrt{2}}2^{\frac{-m+k+m_{0}}{2}}\Big\}\Big|\\
 & =\sum_{Q\in\mathcal{S}_{m}^{*}}\Big|\Big\{ x\in Q:\sum_{P\in\mathcal{S}_{m},P\subseteq Q}|b(x)-b_{P}|\chi_{P}>\frac{\sqrt{2}-1}{2\sqrt{2}}2^{\frac{-m+k+m_{0}}{2}}\Big\}\Big|,\\
\end{align*}
where $\mathcal{S}_{m}^{*}$ is the collection of maximal dyadic cubes
in $\mathcal{S}_{m}$. Now taking into account \cite[Lemma 5.1]{LORR}, we have that
there exists a sparse family $\widetilde{\mathcal{S}_{m}}$ containing
$\mathcal{S}_{m}$ such that
\[
|b(x)-b_{P}|\chi_{P}(x)\leq\|b\|_{\BMO}c_{n}\sum_{R\subseteq P,P\in\mathcal{S}_{m}}\chi_{R}(x)=c_{n}\sum_{R\subseteq P,P\in\mathcal{S}_{m}}\chi_{R}(x).
\]
Taking that into account we can continue the preceding computation as follows
\begin{align*}
 & \sum_{Q\in\mathcal{S}_{m}^{*}}\left|\left\{ x\in Q:\sum_{P\in\mathcal{S}_{m},P\subseteq Q}\left(c_{n}\sum_{R\subseteq P,P\in\widetilde{\mathcal{S}_{m}}}\chi_{R}(x)\right)\chi_{P}>\frac{\sqrt{2}-1}{2\sqrt{2}}2^{\frac{-m+k+m_{0}}{2}}\right\}\right|\\
 & \sum_{Q\in\mathcal{S}_{m}^{*}}\left|\left\{ x\in Q:\left(\sum_{P\in\widetilde{\mathcal{S}_{m}},P\subseteq Q}\chi_{P}\right)^{2}>c\frac{\sqrt{2}-1}{2\sqrt{2}}2^{\frac{-m+k+m_{0}}{2}}\right\}\right|\\
 & \le\exp(-c2^{\frac{-m+k+m_{0}}{4}})\sum_{Q\in\mathcal{S}_{m}^{*}}|Q|.
\end{align*}
Hence, combining the preceding estimates and using the $C_{q}$ condition, we have
\begin{align*}
 & w\Big\{\mathcal{T}_{b,\mathcal{S}_{m}}f(x)>(1-\frac{1}{\sqrt{2}})2^{\frac{m+k+m_{0}}{2}}\Big\}\\
 & \le\exp(-c\epsilon2^{\frac{-m+k+m_{0}}{4}})\sum_{Q\in\mathcal{S}_{m}^{*}}\int M(\chi_{Q})^{q}w.
\end{align*}
Since $\cup_{Q\in\mathcal{S}_{m}^{*}}Q\subset\{x:Mf(x)>2^{m}\}$ then
we have that using Lemma \ref{Lem:Tec2}
\[
\sum_{Q\in\mathcal{S}_{m}^{*}}M(\chi_{Q})^{q}w\lesssim M_{m,p,q}(Mf)^{p}
\]
and consequently
\[
\sum_{Q\in\mathcal{S}_{m}^{*}}\int M(\chi_{Q})^{q}w\lesssim\int M_{m,p,q}(Mf)^{p}.
\]
Hence taking into account Lemma \ref{Lem:Tec1}
\[
\begin{split}w\Big\{\mathcal{T}_{b,\mathcal{S}_{m}}f(x)>(1-\frac{1}{\sqrt{2}})2^{\frac{m+k+m_{0}}{2}}\Big\} & \lesssim\exp(-c\epsilon2^{\frac{-m+k+m_{0}}{4}})\int M_{m,p,q}(Mf)^{p}\\
 & \lesssim\exp(-c\epsilon2^{\frac{-m+k+m_{0}}{4}})\|Mf\|_{L^{p}(w)}^{p}.
\end{split}
\]
This yields
\[\begin{split}
&\sum_{k\in\bbZ}2^{kp}w(\{x:\mathcal{T}_{b,\mathcal{S}_{m}}f(x)>2^{k},M(f)(x)\le\gamma2^{k}\})\\
&\lesssim\sum_{k\in\bbZ}2^{kp}\sum_{m\le k-m_{0}}\exp(-c\epsilon2^{\frac{-m+k+m_{0}}{4}})\|Mf\|_{L^{p}(w)}^{p}
\end{split}
\]
and we are done.

\section{Proof of local decay estimates }\label{Sec:ProofLocalDecay}

As we pointed out earlier, the proof of Theorem \ref{Thm:LocalDecay}
is a straightforward consequence of the combination of the sparse
domination for each operator considered in the statement of the result
and Lemma \ref{Lem:Decay}. Hence it will be enough to give a proof
for the former.
\begin{proof}[Proof of Lemma \ref{Lem:Decay}]
First we observe that if $P$ is an arbitrary cube such that $P\cap Q\not=\emptyset$
and $|P|\simeq|Q|$ for some cube $Q\in\mathcal{S},$ then
\begin{equation}
\left|\left\{ x\in P\,:\sum_{R\in\mathcal{S},\,R\subseteq Q}\chi_{R}(x)>t\right\} \right|\leq ce^{-\alpha t}|P|.\label{eq:SparseExp}
\end{equation}
Indeed we observe that actually
\[
\begin{split}\left|\left\{ x\in P\,:\sum_{R\in\mathcal{S},\,R\subseteq Q}\chi_{R}(x)>t\right\} \right| & =\left|\left\{ x\in P\cap Q\,:\sum_{R\in\mathcal{S},\,R\subseteq Q}\chi_{R}(x)>t\right\} \right|\\
 & \leq\left|\left\{ x\in Q\,:\sum_{R\in\mathcal{S},\,R\subseteq Q}\chi_{R}(x)>t\right\} \right|.
\end{split}
\]
Now we observe that in \cite[Theorem 2.1]{OCPR}, it was established
that
\[
\left|\left\{ x\in Q\,:\sum_{R\in\mathcal{S},\,R\subseteq Q}\chi_{R}(x)>t\right\} \right|\leq ce^{-\alpha t}|Q|,
\]
so recalling now that $|P|\simeq|Q|$ yields (\ref{eq:SparseExp}).

Let us assume now that $\supp f\subseteq Q$ for some arbitrary
cube $Q$. It is clear that $Q$ can be covered by $c_{n}$
pairwise disjoint cubes $Q_j$ in $\mathcal{D}$, such that $|Q|\simeq|Q_{j}|$
and $Q\cap Q_{j}\not=\emptyset$. Then we have that
\[
f=\sum_{j=1}^{c_{n}}f\chi_{Q_{j}}.
\]
Hence
\[
\begin{split}\left|\left\{ x\in Q\,:\,\mathcal{A}_{\mathcal{S}}^{r}|f|>tMf(x)\right\} \right| & \leq\sum_{j=1}^{c_{n}}\left|\left\{ x\in Q\,:\,\mathcal{A}_{\mathcal{S}}^{r}|f\chi_{Q_{j}}|>\frac{t}{c_{n}}M(f\chi_{Q_j})(x)\right\} \right|.\end{split}
\]
We shall assume that each $Q_{j}\in\mathcal{S}$. Indeed, if that
was not the case we can add those cubes to the family and call $\widetilde{\mathcal{S}}$
the resulting family. We observe that in that case if $R\subseteq Q_{j}$,
$R$ satisfies the same Carleson that satisfied before adding those
cubes, that if $R=Q_{j}$ for some $j$, we have that
\[
\sum_{P\subseteq Q_{j},\,P\in\widetilde{\mathcal{S}}}|P|=\sum_{\stackrel{{\scriptstyle R\in\widetilde{\mathcal{S}}}}{R\text{ maximal in }Q_{j}}}\sum_{P\subseteq R,\,P\in\widetilde{\mathcal{S}}}|P|\leq\frac{1}{\eta}\sum_{\stackrel{{\scriptstyle R\in\widetilde{\mathcal{S}}}}{R\text{ maximal in }Q_{j}}}|R|\leq\frac{1}{\eta}|Q_{j}|
\]
and in the case $R$ contains some $Q_{j}$
\[
\sum_{R\subseteq Q_{j},\,P\in\widetilde{\mathcal{S}}}|P|=\sum_{R\subseteq P,\,R\in\mathcal{S}}|P|+\sum_{Q_{j}\subseteq P}|P|\leq\left(\frac{1}{\eta}+c_{n}\right)|P|.
\]
In virtue of Lemma \ref{Lem:CarlesonSparse} the preceding estimates
yield that $\widetilde{\mathcal{S}}$ is a $\frac{\eta}{1+\eta c_{n}}$-sparse
family. Now we observe that
\[
\begin{split} & \mathcal{A}_{\mathcal{S}}^{r}|f\chi_{Q_{j}}|(x)=\left(\sum_{P\in\mathcal{S}}\left(\frac{1}{|P|}\int_{P}|f\chi_{Q_{j}}|\right)^{r}\chi_{P}(x)\right)^{\frac{1}{r}}>tM(f\chi_{Q_j})(x)\\
\iff & \frac{\sum_{P\in\mathcal{S}}\left(\frac{1}{|P|}\int_{P}|f\chi_{Q_{j}}|\right)^{r}\chi_{P}(x)}{M(f\chi_{Q_j})(x)}>t^{r}.
\end{split}
\]
Now we split the sparse operator as follows
\[
\begin{split} & \sum_{P\in\mathcal{S}}\left(\frac{1}{|P|}\int_{P}|f\chi_{Q_{j}}|\right)^{r}\chi_{P}(x)\\
 & =\sum_{P\in\mathcal{S},\,P\subsetneq Q_{j}}\left(\frac{1}{|P|}\int_{P}|f\chi_{Q_{j}}|\right)^{r}\chi_{P}(x)+\sum_{P\in\mathcal{S},\,P\supseteq Q_{j}}\left(\frac{1}{|P|}\int_{P}|f\chi_{Q_{j}}|\right)^{r}\chi_{P}(x).
\end{split}
\]
Now we observe that trivially
\[
\frac{\sum_{P\in\mathcal{S},\,P\subsetneq Q_{j}}\left(\frac{1}{|P|}\int_{P}|f\chi_{Q_j}|\right)^{r}\chi_{P}(x)}{M(f\chi_{Q_j})(x)^{r}}\leq\sum_{P\in\mathcal{S},\,P\subseteq Q_{j}}\chi_{P}(x).
\]
On the other hand we have that for every $x\in Q$, since $5Q_{j}\supset Q$,
\[
\begin{split}\frac{\sum_{P\in\mathcal{S},\,P\supseteq Q_{j}}\left(\frac{1}{|P|}\int_{P}|f\chi_{Q_j}|\right)^{r}\chi_{P}(x)}{M(f\chi_{Q_j})(x)^{r}}
& \leq\sum_{P\in\mathcal{S},\,P\supseteq Q_{j}}\frac{\left(\frac{1}{|P|}\int_{P}|f\chi_{Q_j}|\right)^{r}}{\left(\frac{1}{|5Q_{j}|}\int_{5Q_{j}}|f\chi_{Q_j}|\right)^{r}}\chi_{P}(x)\\
 & =\sum_{P\in\mathcal{S},\,P\supseteq Q_{j}}\frac{\left(\frac{1}{|P|}\int_{Q_{j}}|f|\right)^{r}}{\left(\frac{1}{|5Q_{j}|}\int_{Q_{j}}|f|\right)^{r}}\chi_{P}(x)\\
 & =\sum_{P\in\mathcal{S},\,P\supseteq Q_{j}}\left(\frac{|5Q_{j}|}{|P|}\right)^{r}\chi_{P}(x)\\
 & \leq5^{nr}\sum_{k=0}^{\infty}\frac{1}{2^{nrk}}=\frac{10^{nr}}{2^{nr}-1}.
\end{split}
\]
Combining those estimates and taking into account (\ref{eq:SparseExp}),
\[
\left|\left\{ x\in Q\,:\,\mathcal{A}_{\mathcal{S}}^{r}|f\chi_{Q_{j}}|>tMf(x)\right\} \right|\leq\left|\left\{ x\in Q\,:\,\sum_{\stackrel{{\scriptstyle P\in\mathcal{S}}}{P\subseteq Q_{j}}}\chi_{P}(x)>t^{r}-c_{n,r}\right\} \right|\leq c_{1}e^{-c_{2}t^{r}}|Q|
\]
and we are done.

Let us turn our attention now to the other estimate. By homogeneity
we shall assume that $\|b\|_{\BMO}=1$. Arguing as before it's clear
that $Q$ can be covered by $c_{n}$ pairwise disjoint cubes in $\mathcal{D}$,
such that $|Q|\simeq|Q_{j}|$ and $Q\cap Q_{j}\not=\emptyset$. Let
us denote by $\{Q_{j}\}$ that family of cubes. Then we have that
\[
f=\sum_{k=1}^{c_{n}}f\chi_{Q_{j}}.
\]
As we showed before we can assume that each $Q_{k}\in\mathcal{S}$.
We observe that
\[
\begin{split} & \left|\left\{ x\in Q\,:\,\mathcal{T}_{\mathcal{S},b}f(x)+\mathcal{T}_{\mathcal{S},b}^{*}f(x)>tM^{2}f(x)\right\} \right|\\
 & \leq\sum_{j=1}^{c_{n}}\left|\left\{ x\in Q\,:\,\frac{\mathcal{T}_{\mathcal{S},b}f\chi_{Q_{j}}(x)}{M^{2}\left(f\chi_{Q_{j}}\right)(x)}>\frac{t}{c_{n}2}\right\} \right|+\sum_{j=1}^{c_{n}}\left|\left\{ x\in Q\,:\,\frac{\mathcal{T}_{\mathcal{S},b}^{*}f\chi_{Q_{j}}(x)}{M^{2}\left(f\chi_{Q_{j}}\right)(x)}>\frac{t}{2c_{n}}\right\} \right|\\
 & =\sum_{j=1}^{c_{n}}I_{j}+II_{j}.
\end{split}
\]
Hence the estimate boils down to control both $I_{j}$ and $II_{j}.$
Let us focus first on $I_{j}$.
\[
\begin{split}  I_{j}
 & \leq\left|\left\{ x\in Q\,:\,\frac{\sum_{P\in\mathcal{S}}|b(x)-b_{P}| \langle |f\chi_{Q_{j}}|\rangle_{P}\chi_{P}(x)}{M^{2}(f\chi_{Q_{j}})(x)}>t\right\} \right|\\
 & \leq\left|\left\{ x\in Q\,:\,\frac{\left(\sum_{P\in\mathcal{S},\,P\subsetneq Q_{j}}+\sum_{P\in\mathcal{S},\,P\supseteq Q_{j}}\right)|b(x)-b_{P}| \langle |f\chi_{Q_{j}}|\rangle_{P}\chi_{P}(x)}{M(f\chi_{Q_{j}})(x)}>t\right\} \right|.
\end{split}
\]
First we work on $\frac{\sum_{P\in\mathcal{S},\,P\supseteq Q_{j}}|b(x)-b_{P}|  \langle |f\chi_{Q_{k}}|\rangle_{P}\chi_{P}(x)}{M(f\chi_{Q_j})(x)}$.
We observe that since $\supp f\subseteq Q_{j}$ and $Q\cap Q_{j}\not=\emptyset$
we have that for every $x\in Q$, since $5Q_{j}\supset Q$,
\[
\begin{split} & \frac{\sum_{P\in\mathcal{S},\,P\supseteq Q_{j}}|b(x)-b_{P}|  \langle |f\chi_{Q_{j}}|\rangle_{P}\chi_{P}(x)}{M(f\chi_{Q_{j}})(x)}\\
 & \leq\frac{\sum_{P\in\mathcal{S},\,P\supseteq Q_{j}}|b(x)-b_{P}| \langle |f\chi_{Q_{j}}|\rangle_{P}\chi_{P}(x)}{\frac{1}{|5Q_{j}|}\int_{Q_{j}}|f|}\\
 & \leq5^{n}\sum_{P\in\mathcal{S},\,P\supseteq Q_{j}}|b(x)-b_{P}|\frac{|Q_{j}|}{|P|}\chi_{P}(x)\\
 & \leq5^{n}\sum_{P\in\mathcal{S},\,P\supseteq Q_{j}}|b_{Q_{j}}-b_{P}|\frac{|Q_{j}|}{|P|}\chi_{P}(x)+5^{n}|b(x)-b_{Q_{j}}|\sum_{P\in\mathcal{S},\,P\supseteq Q_{j}}\frac{|Q_{j}|}{|P|}\chi_{P}(x).
\end{split}
\]
Arguing as before
\[
5^{n}|b(x)-b_{Q_{j}}|\sum_{P\in\mathcal{S},\,P\supseteq Q_{j}}\frac{|Q_{j}|}{|P|}\chi_{P}(x)\leq5^{n(r+1)}(2^{nr})'|b(x)-b_{Q_{j}}|\leq c_{n}|b(x)-b_{5Q_{j}}|+c_{n}.
\]
Now we observe that, in the worst case, we can find a sequence of
dyadic cubes $Q_{j}\subseteq P_{1}\subseteq P_{2}\subseteq\dots\subseteq P$
\[
|b_{Q_{j}}-b_{P}|\leq\sum_{j=1}^{\log_{2}\frac{|P|}{|Q_{j}|}}|b_{Q_{j}}-b_{P}|\leq2^{n}\log_{2}\frac{|P|}{|Q_{j}|}.
\]
Then
\[
\begin{split} & 5^{n}\sum_{P\in\mathcal{S},\,P\supseteq Q_{j}}|b_{Q_{j}}-b_{P}|\frac{|Q_{j}|}{|P|}\chi_{P}(x)\leq5^{n}2^{n}\sum_{P\in\mathcal{S},\,P\supseteq Q_{j}}\log_{2}\left(\frac{|P|}{|Q_{j}|}\right)\frac{|Q_{j}|}{|P|}\chi_{P}(x)\\
 & \leq5^{n}2^{n}\|b\|_{\BMO}\sum_{k=1}^{\infty}\frac{k}{2^{nk}}.
\end{split}
\]
Hence
\[
\frac{\sum_{P\in\mathcal{S},\,P\supseteq Q_{j}}|b(x)-b_{P}| \langle |f\chi_{Q_{j}}|\rangle_{P}\chi_{P}(x)}{M(f\chi_{Q_{j}})(x)}\leq c_{n}\left(|b(x)-b_{5Q_{j}}|+1\right),
\]
and we have that
\[
\begin{split} & \leq\left|\left\{ x\in Q\,:\,\frac{\left(\sum_{P\in\mathcal{S},\,P\subsetneq Q_{j}}+\sum_{P\in\mathcal{S},\,P\supseteq Q_{j}}\right)|b(x)-b_{P}| \langle |f\chi_{Q_{j}}|\rangle_{P}\chi_{P}(x)}{M(f\chi_{Q_{j}})(x)}>t\right\} \right|\\
 & \leq\left|\left\{ x\in Q\,:\,\frac{\sum_{P\in\mathcal{S},\,P\subsetneq Q_{j}}|b(x)-b_{P}| \langle |f\chi_{Q_{j}}|\rangle_{P}\chi_{P}(x)}{M(f\chi_{Q_{j}})(x)}+c_{n}\left(|b(x)-b_{5Q_{j}}|+1\right)>\frac{t}{2}\right\} \right|\\
 & \leq\left|\left\{ x\in Q\,:\,\frac{\sum_{P\in\mathcal{S},\,P\subsetneq Q_{j}}|b(x)-b_{P}| \langle |f\chi_{Q_{j}}|\rangle_{P}\chi_{P}(x)}{M(f\chi_{Q_{j}})(x)}>\frac{t}{2}\right\} \right|\\
 & +\left|\left\{ x\in Q\,:\,c_{n}\left(|b(x)-b_{5Q_{j}}|+1\right)>\frac{t}{2}\right\} \right|.
\end{split}
\]
Now we observe that the second term has exponential decay due to John-Nirenberg
theorem, so it sufices to deal with the first term that will have
the subexponential decay. Using Lemma \cite[Lemma 5.1]{LORR}, and
taking into account that $|Q|\simeq|Q_{j}|$ and also (\ref{eq:SparseMax}),
\[
\begin{split}\\
 & \left|\left\{ x\in Q\,:\,\frac{\sum_{P\in\mathcal{S},\,P\subsetneq Q_{j}}|b(x)-b_{P}| \langle |f\chi_{Q_{j}}|\rangle_{P}\chi_{P}(x)}{M(f\chi_{Q_{j}})(x)}>t\right\} \right|\\
 & =\left|\left\{ x\in Q\cap Q_{j}\,:\,\frac{\sum_{P\in\mathcal{S},\,P\subsetneq Q_{j}}|b(x)-b_{P}| \langle |f\chi_{Q_{j}}|\rangle_{P}\chi_{P}(x)}{M(f\chi_{Q_{j}})(x)}>t\right\} \right|\\
 & \leq\left|\left\{ x\in Q_{j}\,:\,\frac{\sum_{P\in\mathcal{S},\,P\subsetneq Q_{j}}|b(x)-b_{P}| \langle |f\chi_{Q_{j}}|\rangle_{P}\chi_{P}(x)}{M(f\chi_{Q_{j}})(x)}>t\right\} \right|\\
 & \leq\left|\left\{ x\in Q_{j}\,:\,\sum_{P\in\mathcal{S},\,P\subsetneq Q_{j}}\sum_{R\in\tilde{\mathcal{S}},\,R\subseteq P}\chi_{R}(x)\chi_{P}(x)>\frac{t}{c_{n}}\right\} \right|\\
 & \leq\left|\left\{ x\in Q_{j}\,:\,\left(\sum_{P\in\tilde{\mathcal{S}},\,R\subsetneq Q_{j}}\chi_{P}(x)\right)^{2}>\frac{t}{c_{n}}\right\} \right|\\
 & \leq\left|\left\{ x\in Q_{j}\,:\,\sum_{P\in\tilde{\mathcal{S}}}\chi_{P}(x)>\sqrt{\frac{t}{c_{n}}}\right\} \right|\leq c_{1}e^{-c_{2}\sqrt{t}}|Q_{j}|\simeq c_{1}e^{-c_{2}\sqrt{t}}|Q|.
\end{split}
\]
For $II_{j}$, we are going to prove that that term also has exponential decay. We observe that by generalized H\"older's inequality
\[
\sum_{Q\in\mathcal{S}}\langle |b-b_{Q}||f\chi_{Q_{j}}|\rangle_{Q}\chi_{Q}(x)\leq c\sum_{Q\in\mathcal{S}}\|f\chi_{Q_{j}}\|_{L\log L,Q}\chi_{Q}(x).
\]
Hence, since $M^{2}f\simeq M_{L\log L}f$, we can split the sparse
operator as we did before
\[
\begin{split}II_{j} & \leq\left|\left\{ x\in Q\,:\,\sum_{P\in\mathcal{S}}\|f\chi_{Q_{j}}\|_{L\log L,P}\chi_{P}(x)>\frac{M_{L\log L}(f\chi_{Q_{j}})(x)t}{c}\right\} \right|\\
 & =\left|\left\{ x\in Q\,:\,\frac{\left(\sum_{P\in\mathcal{S},\,P\subsetneq Q_{j}}+\sum_{P\in\mathcal{S},\,P\supseteq Q_{j}}\right)\|f\chi_{Q_{j}}\|_{L\log L,P}\chi_{P}(x)}{M_{L\log L}(f\chi_{Q_j})(x)}>\frac{t}{c}\right\} \right|.
\end{split}
\]
We shall deal with both terms analogously as in the case of the standard
sparse operator. First we observe that
\[
\frac{\sum_{P\in\mathcal{S},\,P\subsetneq Q_{j}}\|f\chi_{Q_{j}}\|_{L\log L,P}\chi_{P}(x)}{M_{L\log L}(f\chi_{Q_j})(x)}\leq\sum_{P\in\mathcal{S},\,P\subsetneq Q_{j}}\chi_{P}(x).
\]
For the other term we observe that since $5Q_{j}\supset Q$,
\[
\frac{\sum_{P\in\mathcal{S},\,P\supseteq Q_{j}}\|f\chi_{Q_{j}}\|_{L\log L,P}\chi_{P}(x)}{M_{L\log L}(f\chi_{Q_j})(x)}\leq\sum_{P\in\mathcal{S},\,P\supseteq Q_{j}}\frac{\|f\chi_{Q_{j}}\|_{L\log L,P}}{\|f\chi_{Q_{j}}\|_{L\log L,5Q_{j}}}\chi_{P}(x).
\]
Let us call $\Phi(t)=t(1+\log^{+}t)$. We recall that $\Phi(ab)\leq\Phi(a)\Phi(b)$.
Taking that into account we observe that
\[
\begin{split}\frac{1}{|P|}\int_{P}\Phi\left(\frac{f\chi_{Q_{j}}}{\frac{\|f\chi_{Q_{j}}\|_{L\log L,5Q_{j}}}{\Phi^{-1}\left(\frac{|P|}{|5Q_{j}|}\right)}}\right)dx & =\frac{1}{|P|}\int_{Q_{j}}\Phi\left(\frac{f\chi_{Q_{j}}}{\|f\chi_{Q_{j}}\|_{L\log L,5Q_{j}}}\Phi^{-1}\left(\frac{|P|}{|5Q_{j}|}\right)\right)dx\\
 & \leq\frac{1}{|P|}\int_{Q_{j}}\Phi\left(\frac{f\chi_{Q_{j}}}{\|f\chi_{Q_{j}}\|_{L\log L,5Q_{j}}}\right)\frac{|P|}{|5Q_{j}|}dx\\
 & \leq\frac{1}{|5Q_{j}|}\int_{5Q_{j}}\Phi\left(\frac{f\chi_{Q_{j}}}{\|f\chi_{Q_{j}}\|_{L\log L,5Q_{j}}}\right)dx\leq1,
\end{split}
\]
where the last estimate follows from the definition of $\|f\chi_{Q_{j}}\|_{L\log L,5Q_{j}}$.
This proves that
\[
\|f\chi_{Q_{j}}\|_{L\log L,P}\leq\frac{\|f\chi_{Q_{j}}\|_{L\log L,5Q_{j}}}{\Phi^{-1}\left(\frac{|P|}{|5Q_{j}|}\right)}.
\]
Then, since $\Phi^{-1}\left(t\right)\simeq\frac{t}{\log t}$
\[
\begin{split}\sum_{P\in\mathcal{S},\,P\supseteq Q_{j}}\frac{\|f\chi_{Q_{j}}\|_{L\log L,P}}{\|f\chi_{Q_{j}}\|_{L\log L,5Q_{j}}}\chi_{P}(x) & \leq\sum_{P\in\mathcal{S},\,P\supseteq Q_{j}}\frac{1}{\Phi^{-1}\left(\frac{|P|}{|5Q_{j}|}\right)}\chi_{P}(x)\\
 & \simeq\sum_{P\in\mathcal{S},\,P\supseteq Q_{j}}\frac{|5Q_{j}|}{|P|}\log\left(\frac{|P|}{|5Q_{j}|}\right)\chi_{P}(x)\leq c\sum_{k=0}^{\infty}\frac{k}{2^{nk}}.
\end{split}
\]
Summarizing
\[
\frac{\sum_{P\in\mathcal{S},\,P\supseteq Q_{j}}\|f\chi_{Q_{j}}\|_{L\log L,P}\chi_{P}(x)}{M_{L\log L}(f\chi_{Q_j})(x)}\leq c\sum_{k=0}^{\infty}\frac{k}{2^{nk}}.
\]
Combining estimates
\[
\begin{split}II_{j} &\leq \left|\left\{ x\in Q\,:\,\frac{\left(\sum_{P\in\mathcal{S},\,P\subsetneq Q_{j}}+\sum_{P\in\mathcal{S},\,P\supseteq Q_{j}}\right)\|f\chi_{Q_{j}}\|_{L\log L,P}\chi_{P}(x)}{M_{L\log L}(f\chi_{Q_j})(x)}>\frac{t}{c}\right\} \right|
\\ & \leq\left|\left\{ x\in Q\,:\,\sum_{P\in\mathcal{S},\,P\subsetneq Q_{j}}\chi_{P}(x)+c\sum_{k=0}^{\infty}\frac{k}{2^{nk}}>\frac{t}{c}\right\} \right|
\end{split}
\]
and it suffices to use (\ref{eq:SparseMax}).


\end{proof}

\appendix
\renewcommand*{\thesection}{\Alph{section}}

\section{Unweighted quantitative estimates}

In this appendix we collect some quantitative unweighted estimates
for Calder\'on-Zygmund satisfying
Dini condition and their vector-valued counterparts. These estimates
are somehow implicit in the literature and are a basic ingredient
for our fully-quantitative sparse domination results.

\subsection{A quantitative pointwise estimate involving $M_{\delta}^{\sharp}$
and $T$}

In this Section we follow the strategy devised in \cite{AP}. We will
track carefully the constants involved in the case that $T$ is an
$\omega$-Calder\'on-Zygmund with $\omega$ satisfying a Dini condition. Fo\-llo\-wing the notation in \cite{AP}, if $K$ is the kernel
associated to $T$, we define
\[
D_{B}K(y)=\frac{1}{|B|}\frac{1}{|B|}\int_{B}\int_{B}\left|K(x,y)-K(z,y)\right|dxdz.
\]
Let $B=B(x_{0},r)$. Now our purpose is to compute $D_{B}K(y)$ for
$y$ ``far enough'' from the ball $B$. First of all we observe
that
\[
D_{B}K(y)\leq2\frac{1}{|B|}\int_{B}\left|K(x,y)-K(x_{0},y)\right|dx.
\]
Now, for $|x_{0}-y|>2r\geq2|x-x_{0}|$ using the smoothness condition
we obtain
\[
D_{B}K(y)\leq2\omega\left(\frac{r}{|x_{0}-y|}\right)\frac{1}{|x_{0}-y|^{n}}.
\]
By standard computations
\begin{equation}
\begin{split}\sup_{r>0}\int_{|x_{0}-y|>2r}|f(y)|D_{B}K(y)dy & \leq2\sup_{r>0}\sum_{k=1}^{\infty}\int_{2^{k}r\leq|x_{0}-y|<2^{k+1}r}|f(y)|\omega\left(\frac{r}{|x_{0}-y|}\right)\frac{1}{|x_{0}-y|^{n}}dy\\
 & \leq2\sup_{r>0}\sum_{k=1}^{\infty}\int_{|x_{0}-y|<2^{k+1}r}|f(y)|\omega\left(\frac{r}{2^{k}r}\right)\frac{1}{\left(2^{k+1}r\right)^{n}}dy\\
 & \leq2^{n+1}\|\omega\|_{\text{Dini}}Mf(x_{0}).
\end{split}
\label{eq:DBK}
\end{equation}

\begin{prop}
\label{Prop:MdeltaSharpT}Let $T$ be an $\omega$-Calder\'on-Zygmund operator satisfying a Dini condition. For each $0<\delta<1$
we have that
\[
M_{\delta}^{\sharp}(Tf)(x_{0})\leq2^{n+1}\left(\frac{1}{1-\delta}\right)^{\frac{1}{\delta}}\left(\|T\|_{L^{2}\rightarrow L^{2}}+\|\omega\|_{\text{Dini}}\right)Mf(x_{0}).
\]
\end{prop}
\begin{proof}
We are going to prove the following
\[
\left(\frac{1}{|B|}\int_{B}\left|\left|Tf\right|^{\delta}-\left|c\right|^{\delta}\right|dx\right)^{\frac{1}{\delta}}\leq2^{n+1}\left(\frac{1}{1-\delta}\right)^{\frac{1}{\delta}}\left(\|T\|_{L^{2}\rightarrow L^{2}}+\|\omega\|_{\text{Dini}}\right)Mf(x_{0}).
\]
Let $f=f_{1}+f_{2}$ where $f_{1}=f\chi_{B(x_{0},2r)}$ and let us
take $c=\frac{1}{|B|}\int_{B}Tf_{2}$. Since $\left|\left|a\right|^{\delta}-\left|b\right|^{\delta}\right|\leq\left|a-b\right|^{\delta}$
for $\delta\in(0,1)$. Then
\[
\begin{split}\left(\frac{1}{|B|}\int_{B}\left|\left|Tf\right|^{\delta}-\left|c\right|^{\delta}\right|dx\right)^{\frac{1}{\delta}} & \leq2\left(\frac{1}{|B|}\int_{B}\left|Tf_{1}\right|^{\delta}dx\right)^{\frac{1}{\delta}}+2\left(\frac{1}{|B|}\int_{B}\left|Tf_{2}-\frac{1}{|B|}\int_{B}Tf_{2}\right|^{\delta}dx\right)^{\frac{1}{\delta}}\\
 & =2(I+II).
\end{split}
\]
Firstly we control $II$. Using Jensen's inequality
\[
II\leq\frac{1}{|B|}\int_{B}\left|Tf_{2}-\frac{1}{|B|}\int_{B}Tf_{2}\right|dx\leq\int_{|x_{0}-y|>2r}|f(y)|D_{B}K(y)dy.
\]
If we apply now the estimate we obtained before the statement of this
theorem
\[
II\leq2^{n+1}\|\omega\|_{\text{Dini}}Mf(x_{0}).
\]
For $I$ Kolmogorov inequality (See for instance \cite[Exercise 2.1.5]{G})
leads to the following estimate
\[
\left(\frac{1}{|B|}\int_{B}|Tf_{1}(x)|^{\delta}dx\right)^{\frac{1}{\delta}}\leq\left(\frac{1}{1-\delta}\right)^{\frac{1}{\delta}}\|T\|_{L^{1}\rightarrow L^{1,\infty}}\frac{1}{|B|}\|f_{1}\|_{L^{1}(\mathbb{R}^{n})}.
\]
Then, since $\|T\|_{L^{1}\rightarrow L^{1,\infty}}\leq c_{n}\left(\|T\|_{L^{2}\rightarrow L^{2}}+\|\omega\|_{\text{Dini}}\right)$
\[
I\leq c_{n}\left(\frac{1}{1-\delta}\right)^{\frac{1}{\delta}}\left(\|T\|_{L^{2}\rightarrow L^{2}}+\|\omega\|_{\text{Dini}}\right)Mf(x_{0}).
\]

\end{proof}

\subsection{A quantitative control of $\|\overline{T}_{q}\|_{L^{1}\rightarrow L^{1,\infty}}$ }
\begin{prop}
\label{Prop:TqFullyQuant}Let $1<q<\infty$ and $T$ be an $\omega$-Calder\'on-Zygmund
operator satisfying a Dini condition. Then
\[
\|\overline{T}_{q}\|_{L^{1}\rightarrow L^{1,\infty}}\leq c_{n}(\left\Vert \omega\right\Vert _{\text{Dini}}+\|T\|_{L^{q}\rightarrow L^{q}}).
\]
Furthermore, since $\|T\|_{L^{q}\rightarrow L^{q}}\leq c_{n}\left(\left\Vert \omega\right\Vert _{\text{Dini}}+\|T\|_{L^{2}\rightarrow L^{2}}\right)$
\[
\|\overline{T}_{q}\|_{L^{1}\rightarrow L^{1,\infty}}\leq c_{n}(\left\Vert \omega\right\Vert _{\text{Dini}}+\|T\|_{L^{2}\rightarrow L^{2}}).
\]
\end{prop}
\begin{proof}
Fix $\lambda>0$ and let $\{Q_{j}\}$ be the family of non overlapping
cubes that satisfy
\begin{equation}
\lambda\alpha<\frac{1}{|Q_{j}|}\int_{Q_{j}}|\bm{f}(x)|_{q}dx\leq2^{n}\alpha\lambda,\label{descompo}
\end{equation}
and that are maximal with respect to left hand side inequality. Let
us denote by $z_{j}$ and by $r_{j}$ the center and side-length of
each $Q_{j}$, respectively. If we denote $\Omega=\bigcup_{j}Q_{j}$,
then, it is clear that $|f(x)|_{q}\leq\alpha\lambda$ a.e.$x\in\mathbb{R}^{n}\setminus\Omega$.

Now we split $f$ as $f=g+b$, in a slightly different way to the
usual. We consider $g=\{g_{i}\}_{i=1}^{\infty}$ given by
\[
g_{i}(x)=\left\{ \begin{array}{cl}
f_{i}(x) & \mbox{for }x\in\mathbb{R}^{n}\setminus\Omega,\\
(f_{i})_{Q_{j}} & \mbox{for }x\in Q_{j},
\end{array}\right.
\]
where, as usual, $(f_{i})_{Q_{j}}$ is the average of $f_{i}$ on
the cube $Q_{j}$, and
\[
\bm{b}(x)=\{b_{i}(x)\}_{i=1}^{\infty}=\left\{ \sum_{Q_{j}}b_{ij}(x)\right\} _{i=1}^{\infty}
\]
with $b_{ij}(x)=(f_{i}(x)-(f_{i})_{Q_{j}})\chi_{Q_{j}}(x)$. Let $\widetilde{\Omega}=\cup_{j}2Q_{j}$.
We then have
\begin{equation}
\begin{split}\left|\{y\in\mathbb{R}^{n}:|\overline{T}_{q}\bm{f}(y)|>\lambda\}\right|\leq & \left|\{y\in\mathbb{R}^{n}\setminus\widetilde{\Omega}\,:\,|\overline{T}_{q}\bm{g}(y)|>\lambda/2\}\right|\\
 & +\left|\widetilde{\Omega}\right|\\
 & +\left|\{y\in\mathbb{R}^{n}\setminus\widetilde{\Omega}\,:\,|\overline{T}_{q}\bm{b}(y)|>\lambda/2\}\right|.
\end{split}
\label{eq:tresterminos}
\end{equation}
The rest of the proof can be completed following standard computations
(see for instance \cite{PT}). Choosing $\alpha=\frac{1}{\|T\|_{L^{q}\rightarrow L^{q}}}$
yields the desired conclusion.

\end{proof}

\subsection{Boundedness of $\overline{M}_{q}$ on $L^{p,\infty}$ }

In this Section we prove that $\overline{M}_{q} :L^{p,\infty}\rightarrow L^{p,\infty}$.
For that purpose we will use the following Fefferman-Stein type estimate
obtained in \cite[Theorem 1.1]{P}
\begin{thm}
\label{Thm:FefStein}Let $1<p<q<\infty$ then, if $g$ is a locally
integrable function, we have that
\[
\int_{\mathbb{R}^n}\overline{M}_{q}\bm{f}g\leq\int_{\mathbb{R}^n}|\bm{f}|_{q}Mg.
\]
\end{thm}
As we anounced, using the estimate in Theorem \ref{Thm:FefStein},
we can obtain the following result.
\begin{thm}
\label{Thm:WeakppMq}Let $1<p,\,q<\infty$. Then
\[
\left\Vert \overline{M}_{q}\bm{f}\right\Vert _{L^{p,\infty}}\leq c_{n,q}\left\Vert |\bm{f}|_{q}\right\Vert _{L^{p,\infty}}.
\]
\end{thm}
\begin{proof}
Let us fix $1<r<\min\left\{ p,\,q\right\} $. Then
\[
\left\Vert \overline{M}_{q}\bm{f}\right\Vert _{L^{p,\infty}}=\left\Vert \left(\overline{M}_{q}\bm{f}\right)^{\frac{r}{r}}\right\Vert _{L^{p,\infty}}=\left\Vert \left(\overline{M}_{q}\bm{f}\right)^{r}\right\Vert _{L^{\frac{p}{r},\infty}}^{\frac{1}{r}}.
\]
Now by duality
\[
\left\Vert \left(\overline{M}_{q}\bm{f}\right)^{r}\right\Vert _{L^{\frac{p}{r},\infty}}^{\frac{1}{r}}=\left(\sup_{\|g\|_{L^{\left(\frac{p}{r}\right)',1}}=1}\left|\int_{\mathbb{R}^{n}}\left(M_{q}\bm{f}\right)^{r}g\right|\right)^{\frac{1}{r}},
\]
and using Theorem \ref{Thm:FefStein} together with H\"older's inequality in the context of Lorentz spaces we have
\[
\begin{split}\left|\int_{\mathbb{R}^{n}}\left(\overline{M}_{q}\bm{f}\right)^{r}g\right| & \leq\int_{\mathbb{R}^{n}}\left|\left(\overline{M}_{q}\bm{f}\right)^{r}g\right|\leq\int_{\mathbb{R}^{n}}\left|\bm{f}\right|_{q}^{r}\left|Mg\right|\\
 & \leq\|\left|\bm{f}\right|_{q}^{r}\|_{L^{\frac{p}{r},\infty}}\|Mg\|_{L^{\left(\frac{p}{r}\right)',1}}\\
 & \leq c_{n,p,q}\|\left|\bm{f}\right|_{q}\|_{L^{p,\infty}}^{r}\|g\|_{L^{\left(\frac{p}{r}\right)',1}}\leq c_{n,p,q}\|\left|\bm{f}\right|_{q}\|_{L^{p,\infty}}^{r}.
\end{split}
\]
Summarizing
\[
\left\Vert \overline{M}_{q}\bm{f}\right\Vert _{L^{p,\infty}}=\left\Vert \left(\overline{M}_{q}\bm{f}\right)^{r}\right\Vert _{L^{\frac{p}{r},\infty}}^{\frac{1}{r}}\leq\left(c_{n,p,q}\|\left|\bm{f}\right|_{q}\|_{L^{p,\infty}}^{r}\right)^{\frac{1}{r}}\leq c_{n,p,q}\|\left|\bm{f}\right|_{q}\|_{L^{p,\infty}}.
\]
\end{proof}

\subsection{Weak type $(1,1)$ of $\overline{T^*}_q$ }

In this Section we present a fully quantitative estimate of the weak-type
$(1,1)$ of $\overline{T^{*}}_{q}$ via a suitable pointwise Cotlar inequality.

Now we recall Cotlar's inequality for $T^{*}$. In \cite[Theorem A.2]{HRT}
the following result is obtained
\begin{lem}
\label{Lem:CotlarEscalar}Let $T$ be an $\omega$-Calder\'on-Zygmund operator with $\omega$ satisfying a Dini condition and let $\delta\in(0,1)$.
Then
\[
T^{*}f(x)\leq c_{n,\delta}\left(M_{\delta}(|Tf|)(x)+\left(\|T\|_{L^{2}\rightarrow L^{2}}+\|\omega\|_{\text{Dini}}\right)Mf(x)\right).
\]
\end{lem}
Armed with this lemma we are in the position to prove the following pointwise vector-valued Cotlar's inequality.
\begin{lem}
Let $T$ be an $\omega$-Calder\'on-Zygmund operator with $\omega$ satisfying a Dini condition, $\delta\in(0,1)$ and $1<q<\infty$. Then
\[
\overline{T^{*}}_{q}\bm{f}(x)\leq c_{n,\delta}\left(\overline{M}_{\frac{q}{\delta}}(|\overline{T}\bm{f}|^{\delta})(x)^{\frac{1}{\delta}}+\left(\|T\|_{L^{2}\rightarrow L^{2}}+\|\omega\|_{\text{Dini}}\right)\overline{M}_{q}\bm{f}(x)\right),
\]
where $|\overline{T}\bm{f}|^{\delta}$ stands for $\left\{ |Tf_{j}|^{\delta}\right\} _{j=1}^{\infty}$.
\end{lem}
\begin{proof}
It suffices to apply Lemma \ref{Lem:CotlarEscalar} to each term of
the sum.
\end{proof}
\begin{thm}
\label{Thm:MaxTq}Let $T$ be an $\omega$-Calder\'on-Zygmund operator with
$\omega$ satisfying the Dini condition, and $1<q<\infty$. Then
\[
\|\overline{T^{*}}_{q}\bm{f}\|_{L^{1,\infty}}\leq c_{n,\delta,q}\left(\|T\|_{L^{2}\rightarrow L^{2}}+\|\omega\|_{\text{Dini}}\right)\||\bm{f}|_{q}\|_{L^{1}}.
\]
\end{thm}
\begin{proof}
Using the previous lemma
\[
\|\overline{T^{*}}_{q}\bm{f}\|_{L^{1,\infty}}\leq c_{n,\delta}\left(\left\Vert \overline{M}_{\frac{q}{\delta}}(|\overline{T}\bm{f}|^{\delta})^{\frac{1}{\delta}}\right\Vert _{L^{1,\infty}}+\left(\|T\|_{L^{2}\rightarrow L^{2}}+\|\omega\|_{\text{Dini}}\right)\left\Vert \overline{M}_{q}\bm{f}\right\Vert _{L^{1,\infty}}\right).
\]
For the second term we have that
\[
\left\Vert \overline{M}_{q}\bm{f}\right\Vert _{L^{1,\infty}}\leq c_{n,q}\||\bm{f}|_{q}\|_{L^{1}}
\]
so we only have to deal with the first term. There are two possible
ways to deal with that term.

\begin{enumerate}
\item Using that  $\overline{M}_{q}:L^{p,\infty}\rightarrow L^{p,\infty}$ (Theorem \ref{Thm:WeakppMq}).

We observe that
\[
\begin{split}\left\Vert \overline{M}_{\frac{q}{\delta}}(|\overline{T}\bm{f}|^{\delta})(x)^{\frac{1}{\delta}} \right\Vert _{L^{1,\infty}} & =\left\Vert \overline{M}_{\frac{q}{\delta}}(|\overline{T}\bm{f}|^{\delta}) \right\Vert _{L^{\frac{1}{\delta},\infty}}^{\frac{1}{\delta}}\leq C_{n,\delta,q}\left\Vert |T\bm{f}|_{\frac{q}{\delta}}^{\delta}\right\Vert _{L^{\frac{1}{\delta},\infty}}^{\frac{1}{\delta}}\\
 & =C_{n,\delta,q}\left\Vert \overline{T}_{q}\bm{f}\right\Vert _{L^{1,\infty}}\leq C_{n,\delta,q}\|\overline{T}_{q}\|_{L^{1}\rightarrow L^{1,\infty}}\||\bm{f}|_{q}\|_{L^{1}}.
\end{split}
\]
Now, taking into account Proposition \ref{Prop:TqFullyQuant} we have that
\[
\max\left\{ \|\overline{T}_{q}\|_{L^{1}\rightarrow L^{1,\infty}},\|T\|_{L^{2}\rightarrow L^{2}}+\|\omega\|_{\text{Dini}}\right\} \leq c_{n,q}\left(\|T\|_{L^{2}\rightarrow L^{2}}+\|\omega\|_{\text{Dini}}\right)
\]
and we are done.
\item Via extrapolation.

The argument relies upon \cite[Theorem 2.1]{CUMP1}. We recall here
the statement of that result.

\begin{thm*}
Given a family $\mathcal{F}$ of pairs of functions, suppose that
for some $p_{0}$, $0<p_{0}<\infty$, and every weight $w\in A_{\infty}$
\[
\int_{\mathbb{R}^{n}}f(x)^{p_{0}}w(x)dx\leq C\int g(x)^{p_{0}}w(x)dx\qquad(f,g)\in\mathcal{F}
\]
Then for all $0<p,q<\infty,$ and $0<s\leq\infty$ and $w\in A_{\infty}$
we have that
\[
\left\Vert \left(\sum_{j}(f_{j})^{q}\right)^{\frac{1}{q}}\right\Vert _{L^{p,s}(w)}\leq C\left\Vert \left(\sum_{j}(g_{j})^{q}\right)^{\frac{1}{q}}\right\Vert _{L^{p,s}(w)}\qquad(f_{j},g_{j})\in\mathcal{F}
\]
\end{thm*}
Now we recall that for every $0<\varepsilon < \delta <1$ that we have for
every $w\in A_{\infty}$ \large{}(see for instance
\cite{PRRR})
\[
\int_{\mathbb{R}^{n}}M_{\varepsilon}f(x)w(x)dx\leq c_{n}[w]_{A_{\infty}}\int_{\mathbb{R}^{n}}M_{\delta}^{\sharp}fw(x)dx.
\]
In particular
\[
\int_{\mathbb{R}^{n}}M_{\varepsilon}(Tf)(x)w(x)dx\leq c_{n}[w]_{A_{\infty}}\int_{\mathbb{R}^{n}}M_{\delta}^{\sharp}\left(Tf\right)(x)w(x)dx.
\]
Now the extrapolation Theorem yields
\[
\left\Vert \left(\sum_{j}M_{\delta}(|Tf_{j}|)^{q}\right)^{\frac{1}{q}}\right\Vert _{L^{1,\infty}(w)}\leq C\left\Vert \left(\sum_{j}M_{\delta}^{\sharp}\left(Tf_{j}\right)^{q}\right)^{\frac{1}{q}}\right\Vert _{L^{1,\infty}(w)},
\]
and if we choose $w=1$
\[
\left\Vert \overline{M}_{\frac{q}{\delta}}(|\overline{T}\bm{f}|^{\delta})(x)^{\frac{1}{\delta}}\right\Vert _{L^{1,\infty}}=\left\Vert \left(\sum_{j}M_{\delta}(Tf_{j})^{q}\right)^{\frac{1}{q}}\right\Vert _{L^{1,\infty}}\leq C\left\Vert \left(\sum_{j}M_{\delta}^{\sharp}\left(Tf_{j}\right)^{q}\right)^{\frac{1}{q}}\right\Vert _{L^{1,\infty}}.
\]
Applying Proposition \ref{Prop:MdeltaSharpT} we obtain
\[
\left\Vert \left(\sum_{j}M_{\delta}^{\sharp}\left(Tf_{j}\right)^{q}\right)^{\frac{1}{q}}\right\Vert _{L^{1,\infty}}\leq2^{n+1} c_{n,q,\delta}\left(\|T\|_{L^{2}\rightarrow L^{2}}+\|\omega\|_{\text{Dini}}\right)\left\Vert \overline{M}_{q}\bm{f}\right\Vert _{L^{1,\infty}}
\]
and this ends the proof since we know that $\left\Vert \overline{M}_{q}\bm{f}\right\Vert _{L^{1,\infty}}\leq c_{n,q}\||\bm{f}|_{q}\|_{L^{1}}$
and clearly
\[
\max\left\{ c_{n,q},c_{n,q,\delta}\left(\|T\|_{L^{2}\rightarrow L^{2}}+\|\omega\|_{\text{Dini}}\right)\right\} \leq c_{n,\delta,q}\left(\|T\|_{L^{2}\rightarrow L^{2}}+\|\omega\|_{\text{Dini}}\right).
\]
\end{enumerate}
\end{proof}

\bibliographystyle{plain}
\bibliography{refs}

\end{document}